\documentclass{amsart}

\usepackage{amssymb}
\usepackage{color}
\usepackage{tikz-cd}
\usepackage{stmaryrd}
\usepackage{hyperref}
\hypersetup{
    colorlinks=true,
    linkcolor=black,
    filecolor=blue,      
    urlcolor=blue,
    citecolor=blue,
    pdfpagemode=FullScreen,
    }
\usepackage[margin=1.6 in, top=1.4in, bottom=1.4 in]{geometry}
\usepackage{enumitem}

\newtheorem{thm}{Theorem}[section]
\newtheorem{lem}[thm]{Lemma}
\newtheorem{cor}[thm]{Corollary}
\newtheorem{prop}[thm]{Proposition}

\newtheorem{thm*}{Theorem}

\theoremstyle{definition}
\newtheorem{definition}[thm]{Definition}

\theoremstyle{remark}
\newtheorem{remark}[thm]{Remark}

\numberwithin{equation}{section}

\newcommand{\R}{{\mathbb{R}}}
\newcommand{\Z}{{\mathbb{Z}}}

\newcommand{\Q}{{\mathbb{Q}}}

\newcommand{\wh}[1]{\widehat{#1}}
\newcommand{\ol}[1]{\overline{#1}}

\newcommand{\overbar}{\overline}
\newcommand{\om}{\omega}

\newcommand{\la}{\lambda}

\newcommand{\eps}{\epsilon}

\newcommand{\cI}{\mathcal{I}}

\newcommand{\cO}{\mathcal{O}}

\newcommand{\cP}{\mathcal{P}}

\newcommand{\cN}{\mathcal{N}}

\newcommand{\rS}{\mathrm{S}}
\newcommand{\fix}{\mathrm{Fix}}
\newcommand{\per}{\mathrm{Per}}

\renewcommand{\hat}{\wh}

\DeclareMathOperator{\supp}{\mathrm{supp}}

\DeclareMathOperator{\Ham}{\mathrm{Ham}}

\DeclareMathOperator{\Symp}{\mathrm{Symp}}

\DeclareMathOperator{\id}{\mathrm{id}}

\DeclareMathOperator{\loc}{\mathrm{loc}}

\DeclareMathOperator{\cm}{\mathrm{CM}}
\DeclareMathOperator{\hm}{\mathrm{HM}}
\DeclareMathOperator{\cn}{\mathrm{CN}}
\DeclareMathOperator{\hn}{\mathrm{HN}}
\DeclareMathOperator{\cfn}{\mathrm{CFN}}
\DeclareMathOperator{\hfn}{\mathrm{HFN}}
\DeclareMathOperator{\cf}{\mathrm{CF}}
\DeclareMathOperator{\hf}{\mathrm{HF}}
\DeclareMathOperator{\h}{\mathrm{H}}

\DeclareMathOperator{\cz}{\mathrm{CZ}}

\newcommand{\mbmail}{marta.batoreo@ufes.br}
\newcommand{\bfmail}{brayan.ferreira@ufes.br}
\newcommand{\msaemail}{marcelo.sarkis.atallah@umontreal.ca}

\begin{document}

\title{The number of periodic points of surface symplectic~diffeomorphisms}

\author{Marcelo S. Atallah}
\address{Marcelo S. Atallah, Department of Mathematics and Statistics,
	University of Montreal, C.P. 6128 Succ.  Centre-Ville Montreal, QC
	H3C 3J7, Canada}
\email{\msaemail}

\author{Marta Batoréo}
\address{Marta Batoréo, Universidade Federal do Espírito Santo, Departamento de Matemática/CCE, 
Av. Fernando Ferrari, 514, Goiabeiras – Vitória – ES, Brasil, 29075-910}
\email{\mbmail}

\author{Brayan Ferreira}
\address{Brayan Ferreira, Universidade Federal do Espírito Santo, Departamento de Matemática/CCE, 
Av. Fernando Ferrari, 514, Goiabeiras – Vitória – ES, Brasil, 29075-910}
\email{\bfmail}


\begin{abstract}
We use symplectic tools to establish a smooth variant of Franks theorem for a closed orientable surface of positive genus $g$; it implies that a symplectic diffeomorphism isotopic to the identity with more than $2g-2$ fixed points, counted homologically, has infinitely many periodic points. Furthermore, we present examples of symplectic diffeomorphisms with a prescribed number of periodic points. In particular, we construct symplectic flows on surfaces possessing only one fixed point and no other periodic orbits. 
\end{abstract}

\maketitle
\setcounter{tocdepth}{1}
\tableofcontents
\section{Introduction and main results}
\subsection{A Symplectic analog of a theorem of Franks}
A remarkable theorem of Franks \cite{franks1992geodesics, franks1996area} shows that any area preserving homeomorphism of the two-sphere has either two or infinitely many periodic points. Collier, Kerman, Reiniger, Turmunkh and Zimmer \cite{collier2012symplectic} used only symplectic tools to prove a refinement of Franks theorem in the smooth category. It provided evidence in favor of a higher dimensional analog of this dynamical phenomenon which Hofer and Zehnder conjectured to hold for Hamiltonian diffeomorphisms. 

In contrast to the case of the sphere, Conley conjectured \cite{Co} that every Hamiltonian diffeomorphism of the $2$-torus possess infinitely many periodic points. In a positive answer to the conjecture, Franks and Handel \cite{franks2003periodic} showed that every Hamiltonian diffeomorphism of a closed orientable surface $\Sigma$ of positive genus $g$ has periodic points of arbitrarily large period; see also \cite{conley1984subharmonic, LeCalvez05, le2006periodic, hingston2009subharmonic, ginzburg2010conley, ginzburg2012conley, ginzburg2019conley}. A recent result due to Prasad \cite{prasad2023periodic}, independently and simultaneously proved by Guiheneuf, Le Calvez and Passeggi \cite{guiheneuf2023area},  implies that the same is true for symplectic diffeomorphisms of $\Sigma$ that are isotopic to the identity and whose flux is a multiple of a rational class in $\h^1(\Sigma)$. 

On the other hand, an irrational translation on the $2$-torus has no periodic points and, furthermore, for every genus $g\geq 2$ there are examples of symplectic flows with exactly $2g-2$ fixed points and no other periodic orbits; see \cite{batoreo2018periodic} and Theorem \ref{thm:2g-2_fx_pts} for two different constructions. In fact, in this article, we construct on a genus $g\geq2$ surface a symplectic diffeomorphism with only one periodic point; see Theorem \ref{thm:1_fx_pt}. Moreover, as described in Remark \ref{rmk:LC-parasitic}, one can produce examples of symplectic diffeomorphisms on $\Sigma$ with exactly $m$ periodic points (which are also fixed) for all $m\geq1$ by adding \emph{parasitic} fixed points to the symplectic diffeomorphism with exactly one periodic point obtained in Theorem~\ref{thm:1_fx_pt}.

The classification of Le Calvez \cite{LeCalvez_conservative2022} implies that a symplectic diffeomorphism of infinite order of a closed surface of genus $g\geq2$ either has periodic points of arbitrarily large period or has an iterate such that all the periodic points are fixed and which is isotopic to the identity relative the set of fixed points; the classification is similar in the case of the torus.  

Our first result shows that the modulus of the Euler characteristic of $\Sigma$ provides an upper bound for the count of fixed points, weighted by the rank of their local Floer homology group, of a symplectic diffeomorphism that is isotopic to the identity and possesses finitely many periodic points. Here, we only count fixed points  that are contractible with respect to a path of symplectic diffeomorphisms $\psi_t$ based at the identity with endpoint $\psi_1=\psi$; see Definition \ref{def:contractible}. When $g\geq2$, this definition only depends on the time-one map $\psi$ since the flux group vanishes.

\begin{thm}\label{thm:main_surfaces}
Let $(\Sigma,\om)$ be a closed orientable surface of positive genus $g$ and let $\psi\in\Symp_0(\Sigma,\om)$ have finitely many contractible fixed points with respect to a path $\psi_t$ of symplectic diffeomorphisms. If
	\begin{equation}\label{ineq_floc}
		\displaystyle\sum_{x\in\fix_0(\psi_t)}\dim\hf^{\loc}(\psi_t, x)> 2g-2,
	\end{equation}
	then $\psi$ has a simple $p$-periodic point for each sufficiently large prime $p$, which is contractible with respect to $\psi_t^p$. On the other hand, if $\psi$ has finitely many periodic points, then the mean index $\Delta(x,\psi_t)$ vanishes for all $x \in \mathrm{Fix}_0(\psi_t)$.
\end{thm}

Note that the left-hand-side of \eqref{ineq_floc} is always bounded from below by $2g-2$; this follows from \cite{van1995symplectic} or, in the case of surfaces, from the discussion in Remark \ref{rmk:LeCalvez}. Furthermore, we observe that $\dim\hf^{\loc}(\psi_t, x)=1$ when $x$ is a nondegenerate fixed point of $\psi$; see Section \ref{sec:localFH} for details. In particular, Theorem \ref{thm:main_surfaces} immediately implies the following corollary which generalizes a  result of Ginzburg and G\"{u}rel \cite[Theorem~1.7]{GG09_genericexistence} for the $2$-torus.

\begin{cor}\label{cor:main_surfaces}
Suppose $\psi\in\Symp_0(\Sigma,\om)$ is nondegenerate and has finitely many contractible fixed points with respect to a path $\psi_t$ of symplectic diffeomorphisms. If
	\begin{equation*}
		\#\fix_0(\psi_t)> 2g-2,
	\end{equation*}
	then $\psi$ has a simple $p$-periodic point for each sufficiently large prime $p$, which is contractible with respect to $\psi_t^p$. On the other hand, if $\psi$ has finitely many periodic points, then the mean index $\Delta(x,\psi_t)$ vanishes for all $x \in \mathrm{Fix}_0(\psi_t)$.
\end{cor}

Theorem \ref{thm:main_surfaces} can be understood as a generalization of Franks theorem for surfaces of higher genus. Indeed, $\Ham(\rS^2)=\Symp_0(\rS^2)$, and by Franks result if $$\#\fix_0(\psi_t)>|\chi(\rS^2)|=2,$$ then $\psi$ has infinitely many periodic points. Note that a direct generalization of the statement, with a simple count of fixed points, is false for  surfaces of higher genus. As previously mentioned, there are examples of symplectic diffeomorphisms with exactly $m$ periodic points for all $m\geq1$, in particular, for $m>|\chi(\Sigma)|=2g-2$.  

\begin{remark}\label{rmk:LeCalvez}
It is possible to appeal to results in low-dimensional conservative dynamics to reprove the first part of Theorem \ref{thm:main_surfaces}, i.e., to show that a symplectic diffeomorphism $\psi\in\Symp_0(\Sigma,\om)$ with finitely many fixed points that satisfies \eqref{ineq_floc} has periodic points of arbitrarily large period. What follows is a sketch of the argument.

By the classification in \cite{LeCalvez_conservative2022} we may suppose that there exists $k\geq1$ such that $\psi^k$ is isotopic to the identity relative its fixed point set. This implies, when added to the fact that $\pi_1(\Sigma)$ is torsion free, that we can find a symplectic path $\psi_t$ with $\psi_0=\id$, $\psi_1=\psi$ such that $\fix(\psi)=\fix_0(\psi_t)$. Furthermore, it follows that for an isolated fixed point $x\in\fix_0(\psi_t)$, we have $\dim\hf^{\loc}(\psi_t, x)=|L(\psi,x)|$, where $L(\psi,x)$ denotes the Lefschetz index of $x$; see e.g. \cite[(LF1)]{ginzburg2013closed}, \cite[Proposition~58]{le2021barcodes}. On the other hand, the Lefschetz-Hopf fixed point theorem implies that 
\begin{equation*}
	\sum_{x\in\fix_0(\psi_t)}L(\psi,x)=\chi(\Sigma)
\end{equation*}
since $\psi$ is homologically trivial. Thus, inequality \eqref{ineq_floc} is equivalent to 
\begin{equation*}
	\sum_{x\in\fix_{0}(\psi_t)}|L(\psi,x)|>\Bigg\vert\sum_{x\in\fix_0(\psi_t)}L(\psi,x)\,\Bigg\vert,
\end{equation*}
which in turn, is equivalent to the existence of $x\in\fix_0(\psi_t)$ with $L(\psi,x)\geq1$. Therefore, by \cite[Theorem~0.4]{le2006periodic}, $\psi$ has periodic points of arbitrarily large period.
\end{remark}

\begin{remark}
When $\Sigma$ has non-empty boundary and positive genus it is possible to impose a twist condition for $\psi_t$ on the boundary such that, after blowing down the boundary components, one would obtain a $\psi_t'$ on a closed surface of positive genus with at least one contractible fixed point with positive mean index and non-vanishing local Floer homology. It then follows from Theorem \ref{thm:mean index-zero} that $\psi_t$ has interior periodic points of arbitrarily large period. In the exact setting, this Poincaré-Birkhoff type phenomenon has been studied ever since Poincaré's work on the three-body problem. Moreno and van Koert show that Hamiltonian diffeomorphisms of Liouville domains satisfying a certain twist condition on the boundary have simple interior periodic points of arbitrarily large period \cite{moreno2022generalized}.
\end{remark}

\subsection{Surface symplectic diffeomorphisms with exactly one periodic point}
The Lefschetz fixed point theorem implies that a symplectic diffeomorphism $\psi$ of a closed orientable surface $\Sigma$ with non-zero Lefschetz number $L(\psi)$ must have at least one fixed point. Moreover, when $\psi$ is homotopic to the identity and $\Sigma$ has genus $g\geq 2$, the Lefschetz number $L(\psi)$ equals the Euler characteristic $\chi(\Sigma)=2-2g\not=0$ and, hence $\psi$ must have at least one fixed point.

When $g=1$, a symplectic diffeomorphism on a $2$-torus may have no periodic points as in an irrational translation. In fact, by \cite[Theorem~1.7]{GG09_genericexistence} a nondegenerate symplectic diffeomorphism isotopic to the identity defined on a $2$-torus has infinitely many periodic orbits, provided that it has one fixed point; see also Corollary \ref{cor:main_surfaces}. There are examples of degenerate symplectic diffeomorphisms (isotopic to the identity) defined on a $2$-torus with a prescribed number of periodic points; see Remark~\ref{rmk:torus_degenerate}.


In what follows in this section and in Section~\ref{sec:examples}, the surface $\Sigma$ is closed, orientable and has genus $g\geq 2$, unless otherwise explicitly stated.

In Section~\ref{sec:examples}, we present examples of symplectic flows on $\Sigma$ with a finite number $m$ of fixed points and no other periodic orbits. As seen, $m$ must be at least one. The following theorem shows that $m=1$ is possible and its proof is in Section~\ref{section:construction_1fx_pt}.

\begin{thm}\label{thm:1_fx_pt}
	For $g\geq 2$, there is a symplectic flow on $\Sigma$ with exactly one fixed point and no other periodic orbits.
\end{thm}
\begin{remark}
\label{rmk:flux_of_constructions}
It follows from the construction method that the flux of the symplectic flows presented in Section~\ref{sec:examples} is irrational in the sense of \cite[(1-1)]{batoreo2018periodic}, in particular, it does not lie in $\h^1(\Sigma;\Q)$; see also the discussion in \cite[p.35]{batoreo2018periodic}.
\end{remark}

The time-one map $\psi$ of the symplectic flow $\psi_t$ in Theorem~\ref{thm:1_fx_pt} is a symplectic diffeomorphism isotopic to the identity with exactly one (degenerate) fixed point. More precisely, by the Lefschetz fixed-point formula
\begin{eqnarray}\label{eqn:Lefschetz_indices}
	L(\psi)=\displaystyle\sum_{x\in \text{Fix}(\psi)} L(\psi,x),	
\end{eqnarray}

and, hence, the Lefschetz index $L(\psi,x)$ of the unique periodic point $x$ of the symplectic diffeomorphism $\psi$ must be equal to $2-2g$. In fact, by the Poincaré--Hopf index theorem, the index of the unique periodic point of the symplectic flow $\psi_t$ must be equal to $2-2g$.

Equality \eqref{eqn:Lefschetz_indices} for $\psi_t$ can be achieved, on the one hand, with exactly one fixed point with index $2-2g$ as in Theorem~\ref{thm:1_fx_pt} and, on the other hand, but not exclusively, by $2g-2$ fixed points each with index $-1$. The following theorem was proved in \cite[Section~3A]{batoreo2018periodic}. A similar but different construction is presented in Section~\ref{section:partition}. 

\begin{thm}\label{thm:2g-2_fx_pts}
	For $g\geq 2$, there is a symplectic flow on $\Sigma$ with exactly $2g-2$ fixed points, each with index $-1$, and no other periodic orbits.
\end{thm}

\begin{remark}\label{remark:partition}
	In fact, we expect that, for any partition $(a_1,\ldots,a_k)$ of the integer $2g-2$, i.e. $\sum_{i=1}^k a_i = 2g-2$ with $a_i\geq 1$ and $a_i\leq a_j$ (when $i<j$), there is a symplectic flow $\psi_t$ with exactly $k$ fixed points $p_1,\ldots,p_k$ where $L(\psi_t,p_i)=-a_i,\; i=1,\ldots,k$, and no other periodic orbits. A discussion about this statement is in Section~\ref{section:partition}, where we present the constructions for the case $g=2$ or $g=3$. Theorem~\ref{thm:1_fx_pt} corresponds to the case $k=1$ and $a_1=2g-2$ and Theorem~\ref{thm:2g-2_fx_pts} corresponds to the case $k=2g-2$ and $a_i=1$, for all $i=1,\ldots, k$.
\end{remark}

\begin{remark}
	\label{rmk:LC-parasitic}
	Note that, as pointed out by Le Calvez in \cite[page~2]{LeCalvez_conservative2022}, one may construct, on a surface, ``a smooth symplectic map with an arbitrarily large number of fixed points and no other periodic point". In fact, one may refine the previous statement and construct a smooth symplectic map with any number $m\geq 1$ of fixed points and no other periodic points by taking the time-one map obtained from Theorem~\ref{thm:1_fx_pt} and adding $m-1$ fixed points using the procedure mentioned in \cite[page~2]{LeCalvez_conservative2022}; each \emph{parasitic} fixed point added using the referred method has Lefschetz index $0$. For symplectic maps with a prescribed (finite) number of fixed points with non-zero Lefschetz indices recall Remark~\ref{remark:partition} and see Section~\ref{section:partition}; observe that, in the the previous remark, $a_i\not=0$, for all $i=1,\ldots, k.$ 
	It is also possible to \emph{remove} a fixed point with index $0$ of a given symplectic map on a surface; see \cite[Introduction and Theorem~3.1]{slaminka93}. 	
\end{remark}

\begin{remark}\label{rmk:torus_degenerate}
	When $g=1$, consider an irrational linear flow on a $2$-torus. This flow has no periodic orbits. One may \emph{introduce} as many fixed points as desired using the procedure described in \cite[page~2]{LeCalvez_conservative2022}. One obtains a symplectic flow on the $2$-torus with a prescribed number of fixed points and no other periodic orbits. Each fixed point has Lefschetz index $0$ and the obtained flow is degenerate.	
\end{remark}

\subsection{Symplectic dynamics in higher dimensions}
The Hofer-Zehnder conjecture is a higher dimensional variant of Franks theorem. In short, it predicts that a Hamiltonian diffeomorphism of a closed symplectic manifold with more fixed points than required by the Arnold conjecture must have infinitely many periodic points.

The groundbreaking work of Shelukhin \cite{shelukhin2022hofer} formalized an interpretation of the Hofer-Zehnder conjecture which is closest to the original statement \cite[p. 263]{hofer2012symplectic}. Indeed, a homological version of the Arnold conjecture, which has now been settled, states that a nondegenerate Hamiltonian diffeomorphism of a closed symplectic manifold $(M,\om)$ possesses at least $\dim\h_{*}(M)$ fixed points; see \cite{arnold2014stability, arnold1965proprietes} for the conjecture and \cite{Floer1, Floer2, Floer3, hofer1995floer, liu1998floer, fukaya1999arnold, ruan1999virtual, piunikhin1996symplectic, abouzaid2021arnold, bai2022arnold, rezchikov2022integral} for the results in arbitrary dimensions. When the Hamiltonian diffeomorphism $\phi$ has isolated fixed points and is possibly degenerate the count of each fixed point $x$ must be weighted by the rank of the local Floer homology group $\hf^{\loc}(\phi,x)$. It follows from \cite{shelukhin2022hofer} that a Hamilonian diffeomorphism $\phi$, of a closed monotone symplectic manifold whose even part of the quantum homology is semisimple, such that
\begin{equation}
\label{eq:HZ-Ham}
	\displaystyle\sum_{x\in\fix_0(\phi)}\dim\hf^{\loc}(\phi, x)>\dim\h_{*}(M),
\end{equation}
has a simple $p$-periodic orbit for every sufficiently large prime $p$. Here, the sum is over the contractible fixed points of $\phi$ as in Definition \ref{def:contractible}. See \cite{atallah2023hz} and \cite{bai2024hz} for generalizations to semipositive and toric symplectic manifolds, respectively. 

From this viewpoint, Theorem \ref{thm:main_surfaces} can be seen as evidence of a Hofer-Zehnder type phenomenon for symplectic diffeormorphisms that are isotopic to the identity and perhaps not Hamiltonian. Lê and Ono \cite{van1995symplectic, ono05FN, van2020floer}
showed that if $\psi_t$ is a path of symplectic diffeomorphisms based at the identity with a nondegenerate endpoint $\psi=\psi_1$, then the number of contractible fixed points is bounded from below by $\dim_{\Lambda_\theta} \hn(\theta)$. Here, $\hn(\theta)$ denotes the Morse--Novikov homology with respect to a closed $1$-form $\theta$ representing the class of the flux of $\psi_t$ in $\h^1(M;\R)$; see Section \ref{secmorsenovikov} for the definition of Morse--Novikov homology and Section \ref{sec:sympper} for that of flux. Note that for surfaces $\dim_{\Lambda_\theta} \hn(\theta)=2g-2$ whenever $\theta$ is not exact; see Corollary~\ref{HNranksurfaces}. Therefore, Inequality~\eqref{ineq_floc} translates to the following generalization of Inequality~\eqref{eq:HZ-Ham} in higher dimensions
\begin{equation*}
	\displaystyle\sum_{x\in\fix_0(\psi_t)}\dim\hf^{\loc}(\psi_t, x)>\dim_{\Lambda_\theta}\hn(\theta).
\end{equation*}
It would be interesting to know, and is subject of further investigation, whether such an inequality prompts the existence of infinitely many periodic points in higher dimensions. We note that the usual approach of using semisimplicity to obtain this type of result in higher dimensions only applies to Hamiltonian diffeomorphisms. Indeed, the recent work of Pieloch implies that closed symplectic manifolds with semisimple even quantum cohomology have finite fundamental groups \cite{pieloch2022fundamental}.

The following definition, which is a non-Hamiltonian variant of a pseudrotation, would capture the elements in $\Symp_0(M,\om)$ with finitely many periodic points.

\begin{definition}
\label{def:pseudotranslation}
We call a symplectic diffeomorphism $\psi$ a \emph{pseudotranslation} if there exists a symplectic path $\psi_t$ with $\psi_0=\id$, $\psi_1=\psi$, and flux $[\theta]$ for which there exists a sequence of positive integers $\{k_j\}_{j>0}$ diverging to infinity such that $\fix_0(\psi^{k_j}_t)$ is finite and independent of $j$, $\psi^{k_j}$ is an admissible iteration\footnote{An iterate $\psi^k$ is called admissible if for every fixed point $x$ of $\psi$ if $\la^k=1$ implies that $\la=1$ for every eigenvalue of $d\psi_x$; see~\cite[Section~4.2]{ginzburg2009action}. } of $\psi$, and
	\begin{equation*}
	\displaystyle\sum\dim\hf^{\loc}(\psi^{k_j}_t, x^{k_j})=\dim_{\Lambda_\theta}\hn(\theta)\quad\forall j>0.
	\end{equation*}
The sum is over all $x\in\fix_0(\psi_t)$. Note that  $\dim_{\Lambda_\theta}\hn(\theta)=\dim_{\Lambda_{r\theta}}\hn(r\theta)$ holds for all $r\in\R\setminus\{0\}$.
\end{definition}

\begin{remark}
For surfaces, Theorem \ref{thm:main_surfaces} implies that every $\psi\in\Symp_0(\Sigma)$ with finitely many periodic points is a pseudotranslation. In particular, irrational translations of the $2$-torus, and the constructions in Theorem~\ref{thm:1_fx_pt} and Theorem~\ref{thm:2g-2_fx_pts}, and, in fact, all of those contained in Section~\ref{sec:examples}, are examples of pseudotranslations. When $[\theta]=0$ the definition of a pseudotranslation is closely related to that of a pseudorotation as in \cite{atallah2020hamiltonian}. In higher dimensions, symplectic circle actions with isolated fixed points give rise to examples of pseudotranslations; see \cite{tolman2017non,jang2023non}. Indeed, a circle action induces a $\Z_p$ action for an arbitrarily large prime $p$. Since the fixed points are isolated, the generator of such an action for a sufficiently large $p$ only has nondegenerate contractible fixed points, and the cardinality of the fixed point set is equal to the Euler characteristic of the symplectic manifold and hence, by Proposition \ref{eulernumber}, equal to the Morse-Novikov Betti number of associated to its flux.
\end{remark}

The following proposition, which is proved in Section \ref{sec:localFH}, shows that the mean index of any Floer theoretically visible fixed point of a pseudotranslation must be zero.

\begin{prop}
\label{prop:pseudotranslation}
Let $\psi$ be a pseudotranslation of a closed rational symplectically Calabi--Yau symplectic manifold $(M,\om)$ and let $\psi_t$ be as in Definition~\ref{def:pseudotranslation}. Then, $\Delta(\psi_t,x)=0$ for all $x\in\fix_0(\psi_t)$ with non-trivial local Floer homology.  
\end{prop}

We note that the assertion of Proposition \ref{prop:pseudotranslation} does not hold for general closed symplectic manifolds as evidenced by an irrational rotation of the sphere. Furthermore, in the symplectically Calabi--Yau setting this is a purely non-Hamiltonian phenomenon since pseudorotations do not exist; see \cite{ginzburg2019conley, atallah2020hamiltonian}.

\subsection{Proof overview}
\label{sec:proof}
\subsubsection{An outline of the proof of Theorem~\ref{thm:main_surfaces}}
The symplectic proof of our main result has three key ingredients. The first is a generalization of the arguments contained in \cite[Section~4A]{batoreo2018periodic} to the degenerate setting and for symplectic diffeomorphisms of arbitrary flux.

\begin{thm}\label{thm:mean index-zero}
Let $\psi$ be a symplectic diffeomorphism of a surface of positive genus with isolated fixed points contractible with respect to $\psi_t$. If $x\in\fix_0(\psi_t)$ has non-trivial local Floer homology and $\Delta(\psi_t,x)$ is non-zero, then there exists a simple $p$-periodic point for each sufficiently large prime $p$, which is contractible with respect to $\psi_t^p$.
\end{thm}

The proof of Theorem \ref{thm:mean index-zero} is contained in Section \ref{sec:proof-mean index-zero}. In short, we first observe that the support of the Floer--Novikov homology group $\hfn(\psi_t)$ is contained in the closed interval $I=[-1,1]$; in fact, if $\psi_t$ has non-trivial flux, then the homology is concentrated in degree zero for surfaces of positive genus. This follows from the more general fact that, for closed symplectic manifolds, the rank in each degree of the version of Floer--Novikov homology developed in \cite{van1995symplectic, ono05FN} is constant along certain deformations of the flux; see Section \ref{hfnsection} for details. We reach a contradiction under the assumption that $p$ is a sufficiently large prime and there are no simple $p$-periodic orbits contractible with respect to $\psi_t^p$. The idea is to use the fact $\Delta(\psi_t,x)$ is non-zero by assumption and grows linearly with iterations to guarantee that the support of the local Floer homology associated to a sufficiently large iterate of $x$ is disjoint from $I$. To deal with the case $x$ is degenerate we use the techniques of blowing up and gluing, described in Section \ref{sec:blow-up_and_gluing}, to produce a new surface where a copy of $x^p$ is the only periodic point with mean index $p\Delta(\psi_t,x)$ and must be non-trivial (after perturbing) in homology. 

The next step is to show there exists a dichotomy for $\psi\in\Symp_0(\Sigma,\om)$ with finitely many contractible fixed points with respect to a path $\psi_t$. That is, either there exists a fixed point $x\in\fix_0(\psi_t)$ satisfying the hypothesis of Theorem \ref{thm:mean index-zero} or there exists a \emph{symplectically degenerate extremum}; see Definition \ref{defn:SDM}.  

\begin{thm}
\label{thm:dichotomy}
Let $\psi$ be a symplectic diffeomorphism of a surface of positive genus. Suppose $\psi$ has finitely many contractible fixed points with respect to a path $\psi_t$ of symplectic diffeomorphisms that has non-trivial flux and satisfies \eqref{ineq_floc}. Then, either there exists $x\in\fix_0(\psi_t)$ with non-trivial local Floer homology and $\Delta(\psi_t,x)\neq0$ or there exists a symplectically degenerate extremum.
\end{thm}

The proof of Theorem \ref{thm:dichotomy} is delicate and can be found in Section~\ref{sec:proof-dichotomy}. Roughly, we first observe that Inequality \eqref{ineq_floc} implies that the total Floer--Novikov differential is non-trivial. Then, assuming that no contractible fixed point satisfies the hypothesis of Theorem~\ref{thm:mean index-zero},  we show that a symplectically degenerate extremum must exist. Otherwise, one obtains a non-trivial class of degree $\pm1$ in $\hfn(\psi_t)$, which is impossible considering that it must be supported in degree zero whenever $\psi_t$ has flux.

From the discussion in Section \ref{subsection:SDM} it follows that the existence of a symplectically degenerate extremum implies that $\psi$ has a simple $p$-periodic orbit for each sufficiently large prime. The arguments in that section closely follow the theory developed in \cite{ginzburg2010conley, ginzburg2009action}; we simply adapt them to the non-Hamiltonian setting. With this in mind, a combination of Theorem \ref{thm:mean index-zero} and Theorem \ref{thm:dichotomy} proves the first assertion of Theorem \ref{thm:main_surfaces}. While the conditions of Theorem \ref{thm:dichotomy} supposes non-trivial flux, Theorem~\ref{thm:main_surfaces} is already a consequence of Ginzburg's original proof of the Conley conjecture in the Hamiltonian case \cite{ginzburg2010conley}. 

The second assertion, which concerns the mean indices, follows from Theorem \ref{thm:mean index-zero} whenever the fixed point has non-trivial local Floer homology. It remains to show that $\Delta(\psi_t,x)$ must be zero even when $x$ has trivial local Floer homology. We argue by contradiction. Note that since $x$ is contractible we may compose $\psi_t$ with a contractible loop of Hamiltonian diffeomorphisms (which does not change the mean index nor the rank of local Floer homology) to obtain a path $\psi_t'$ which fixes $x$ for all $t$. At this point, we observe that the sum of ranks of the local Floer homology groups of every $x\in\fix_0(\psi_t')$ is $2g-2$ by the first part of the Theorem \ref{thm:main_surfaces}. We then use the blow-up and gluing technique detailed in Section \ref{sec:blow-up_and_gluing} and applied in the proof of Theorem~\ref{thm:main_surfaces} to obtain a symplectic path $\phi_t$ on a genus $2g$ surface whose contractible fixed points are exactly two copies of all the contractible fixed points of $\psi_t'$ except for $x$. This implies the sum of the ranks of the local Floer homology groups of every contractible fixed point of $\phi_t$ is exactly $2(2g-2)=4g-4$. This contradicts the fact the rank of the Floer--Novikov homology associated to $\phi_t$ is $2(2g)-2=4g-2$.

We observe that in view of Remark \ref{rmk:LeCalvez} the argument contained in the previous paragraph does not require the use of Floer theory. 

\subsubsection{An outline of the construction of symplectic flows with a prescribed number of fixed points}
The construction of a symplectic flow $\psi_t$ on $\Sigma$ with exactly one fixed point has three steps: in step~$1$, we consider $g$ tori with irrational linear flows, in step~$2$, we construct a surface with $g$ \emph{cylindrical-ends} and define on it a Hamiltonian flow with exactly one fixed point and, in step~$3$, we glue each torus to  one of the \emph{ends} of the constructed surface in order to obtain $\Sigma$ with the desired symplectic flow $\psi_t$. A draft of the construction of $\Sigma$ is illustrated in Figure~\ref{figure:general}.

The constructions in Theorem \ref{thm:2g-2_fx_pts} and Remark \ref{remark:partition} can be similarly obtained by replacing the Hamiltonian flow defined in step 2 above.

\subsection*{Acknowledgements}
The authors are grateful to Viktor Ginzburg, Başak G\"{u}rel, and Egor Shelukhin for motivating the project and for the inspiring discussions. In particular, we thank Egor Shelukhin for suggesting that Remark \ref{rmk:LeCalvez} and Theorem \ref{thm:dichotomy} were true, and that an adaptation of Hamiltonian SDM theory to our setting was possible. M.S.A was partially supported with funding from the Fondation Courtois, the ISM, the FRQNT, and the Fondation J. Armand Bombardier. M.B. benefitted from a visit to IMPA hosted by Leonardo Macarini and is thankful for the hospitality and the valuable discussions. M.B. was partially supported by CAPES/PRAPG grant 88881.964878/2024-01. 
\section{Preliminaries}

In this section, we establish notations and conventions, and review some facts by providing a brief overview of Floer--Novikov homology for symplectic diffeomorphisms. Most of the material in this section can be found in \cite{van1995symplectic,ginzburg2009action,van2020floer}.

Let $(M^{2n},\om)$ be a closed symplectic manifold, $\psi\in\Symp_0(M,\om)$ and denote by $\fix(\psi)$ the set of fixed points of $\psi$. We say that $x\in M$ is a \textit{periodic point} of $\psi$ if there exists a positive integer $k$ such that $\psi^k(x)=x$, and denote by $\per(\psi)$ the set of periodic points of $\psi$.  A $k$-periodic point $x$ is called \emph{simple} if $k$ is the smallest positive integer such that $\psi^{k}(x)=x$.

\begin{definition}
	\label{def:contractible}
	A fixed point $x$ of a symplectic diffeomorphism $\psi\in\Symp_0(M,\om)$ is said to be contractible with respect to a path $\psi_t$ of symplectic diffeomorphisms based at the identity with $\psi_1=\psi$, when the loop $\{\psi_t(x)\}$ is contractible in $M$.
\end{definition}


\subsection{Morse--Novikov homology}\label{secmorsenovikov}
Let $N$ be a closed smooth manifold and $\theta \in \Omega^1(N)$ a closed $1$-form. Consider $\pi \colon \overbar{N}^\theta \to N$ the minimal covering space of $N$ where $\pi^*\theta$ is an exact $1$-form. Namely, it is the covering associated with the kernel of the period homomorphism
\begin{align*}
	I_\theta\colon \pi_1(N) \to \R, \quad [x] \mapsto \langle [\theta],[x] \rangle = \int_x \theta.
\end{align*}
The covering transformation group is isomorphic to 
$$\Gamma_\theta := \frac{\pi_1(N)}{\ker I_\theta}.$$
We fix $\Q$ as the ground field. Let $\Lambda_\theta$ be the Novikov ring given by the completion of the group ring of $\Gamma_\theta$ with respect to $I_\theta$, that is
\begin{align*}
	\Lambda_{\theta} = \Big\{\sum_i a_i g_i \ \vert \ a_i \in \Q, g_i \in \Gamma_{\theta}, \  \text{satisfying the condition below}\Big\}:
\end{align*}
\begin{itemize}
	\item the set $\{i \ \vert \ a_i \neq 0, I_\theta(g_i)>c\}$ is finite for all $c \in \R$.
\end{itemize}
Suppose that $\pi^*\theta$ is a Morse $1$-form, i.e., there exists a Morse function $f_\theta \colon \overline{N}^\theta \to \R$ such that $df_\theta = \pi^*\theta$. Let $f_\theta$ be such a primitive. We define the Morse--Novikov chain complex by
$$\cn_k(\theta) = \Big\{\sum_i a_i p_i \ \vert \ a_i \in \Q, \,p_i \in \mathrm{Crit}(f_\theta) \ \text{satisfying the conditions below}\Big\}:$$
\begin{itemize}
	\item the set $\{i \ \vert \ a_i\neq 0, f_\theta(p_i)>c\}$ is finite for any $c \in \R$;
	\item $\mathrm{index}_{f_\theta}(p_i) = k$.
\end{itemize}
This $\Q$-vector space is a finitely generated free module over $\Lambda_\theta$ and for a generic choice of Riemannian metric $g$ on $N$, it is endowed with a boundary operator as follows. For $p \in \mathrm{Crit}(f_\theta)$, we define
$$\partial(p) = \sum_{q} \langle p,q \rangle_{f_\theta} q,$$
where the sum runs over all $q \in \mathrm{Crit}(f_\theta)$ such that $\mathrm{index}_{f_\theta}(p)-\mathrm{index}_{f_\theta}(q) = 1$ and $\langle p,q \rangle_{f_\theta}$ denotes the count (with signs) of the rigid bounded antigradient flow lines of $f_\theta$, with respect to the metric $\pi^*g$, connecting $p$ and $q$. Here, $\mathrm{index}_{f_\theta}(p)$ denotes the Morse index of $p$. When $g$ is generic, this boundary operator is well defined and extends linearly to a map
$$\partial \colon \cn_k(\theta) \to \cn_{k-1}(\theta)$$
that satisfies $\partial^2 = \partial \circ \partial = 0$. In this case, the Morse--Novikov homology of $\theta$ is defined as usual by
$$\hn_k(\theta) = \frac{\ker (\partial \colon \cn_k(\theta) \to \cn_{k-1}(\theta))}{\mathrm{im}(\partial \colon \cn_{k+1}(\theta) \to \cn_{k}(\theta))}.$$
The resulting homology is independent of the generic metric $g$ and only depends on the cohomology class of $\theta$. Namely, if $[\theta_1]=[\theta_2] \in \h^1(N;\R)$, then $\hn_k(\theta_1) \cong \hn_k(\theta_2)$ as modules over $\Lambda_{\theta_1}=\Lambda_{\theta_2}$, for every $k$.

\begin{remark}
	When $\theta$ is exact, say $\theta=df$, where $f\colon N \to \R$ is a Morse function, we have $\overline{N}^\theta = N$, $\Lambda_\theta \cong \Q$, and the Morse--Novikov chain complex $\cn_k(\theta)$ coincides with the Morse chain complex $\cm_k(f)$. In particular, $\hn_k(\theta)=\hm_k(f)=\h_k(N)$, where $\h_k(N)$ denotes the $k$-th singular homology group with coefficients in $\Q$.
\end{remark}

We have the following difference between Morse--Novikov homology and Morse homology.
\begin{prop}[{\cite[Proposition 4.12]{ono2006floer}}]\label{vanishHN}
	Let $\theta$ be a closed but not exact $1$-form defined on a closed smooth manifold $N$. Then
	$$\dim_{\Lambda_\theta} \hn_k(\theta) = 0$$
	for $k=0$ and $k=\dim N$.
\end{prop}

In particular, when $[\theta]\neq0$, the Morse--Novikov homology $\hn_*(\theta) = \oplus_k \hn_k(\theta)$ is not isomorphic to the total Morse (or singular) homology $\h_*(N) = \oplus_k \h_k(N)$. Still, we have the following equality.

\begin{prop}[{\cite[Theorem C.4]{van1995symplectic}}]\label{eulernumber}
	The Euler number of the Morse--Novikov homology coincides with the standard Euler characteristic $\chi(N)$ of the closed manifold $N$, i.e.,
	$$\chi(\theta) := \sum_i (-1)^i \dim_{\Lambda_\theta} \hn_i(\theta) = \chi(N),$$
	for any $\theta \in \Omega^1(N)$.
\end{prop}

Combining Proposition \ref{vanishHN} and Proposition \ref{eulernumber}, we obtain the following consequence for closed surfaces.

\begin{cor}\label{HNranksurfaces}
	Let $\Sigma$ be a closed orientable surface of genus $g\geq 1$ and $\theta \in \Omega^1(\Sigma)$ a nonexact closed $1$-form. Then
	$$\dim_{\Lambda_\theta}\hn_*(\theta)=\dim_{\Lambda_\theta}\hn_1(\theta) = -\chi(\Sigma) = 2g-2.$$ 
\end{cor}

\subsection{Symplectic diffeomorphisms and periodic points}
\label{sec:sympper}
Let $(M^{2n},\om)$ be a closed symplectic manifold and $\psi\in\Symp_0(M,\om)$ a symplectic diffeomorphism in the connected component of the identity of the group of symplectic diffeomorphisms $\Symp(M,\om)$. Given a symplectic isotopy $\psi_t$ connecting the identity $\psi_0 = id$ to $\psi_1 = \psi$, we define a symplectic time-dependent vector field $X_t$ by
$$\frac{d}{dt}\psi_t = X_t \circ \psi_t.$$
The \emph{flux homomorphism} is defined by
\begin{align*}
	\mathrm{Flux} \colon \widetilde{\Symp_0}(M,\om) \to \h^1(M;\R),\quad  [\psi_t] \to \left[\int_0^1 \omega(X_t, \cdot)\right],
\end{align*}
where $\widetilde{\Symp_0}(M,\om)$ denotes the universal covering of $\Symp_0(M,\om)$. The kernel of this homomorphism is the universal covering $\widetilde{\Ham}(M,\omega)$ of the group of Hamiltonian diffeomorphisms $\Ham(M,\omega)\subset \Symp_0(M,\om)$. Recall that $\psi$ is a Hamiltonian diffeomorphism precisely when the isotopy $\psi_t$ can be chosen such that $\omega(X_t, \cdot) = - dH_t$, for some Hamiltonian $H \colon \R \times M \to \R$, where $H(t,\cdot) = H_t$.

Let $\theta$ be a closed $1$-form representing the flux of the isotopy $\psi_t$, that is, $\mathrm{Flux}[\psi_t] = [\theta]$. By the Deformation Lemma \cite[Lemma 2.1]{van1995symplectic}, we may suppose, up to composing $\psi_t$ with a loop of Hamiltonian diffeomorphisms, that $$\omega(\cdot, X_t) = \theta + dH_t := \theta_t,$$
for a periodic Hamiltonian $H_t \colon M \to \R$, $t\in S^1 = \R/\Z$.

In this case, each $\theta_t$ represents the same cohomology class $[\theta] \in \h^1(M;\R)$ and there is a one-to-one correspondence between the set $\mathcal{P}(\theta_t)$, consisting of $1$-periodic solutions of $\dot{x} = X_t(x)$, and the fixed points of $\psi = \psi_1$.

We recall that a $1$-periodic orbit $x \in \mathcal{P}(\theta_t)$ is called \emph{nondegenerate} if the linear symplectic map $d\psi_{x(0)} \colon T_{x(0)}M \to T_{x(0)}M$ does not admit $1$ as an eigenvalue. Moreover, $\psi$ is called \emph{nondegenerate} if all $1$-periodic orbits of $X_t$ are nondegenerate. If $\psi$ is nondegenerate, the set $\mathcal{P}(\theta_t)$ is finite, provided $M$ is compact.

The set $\mathcal{P}(\theta_t)$ agrees with the zero set of the following $1$-form $\alpha_{\psi_t}$ on the free loop space $\mathcal{L}M$ over $M$
\begin{equation}\label{1formzeroset}
	\alpha_{\psi_t}(\xi) = \int_0^1 \om(\dot{x}-X_t,\xi)dt = \int_0^1 \omega(\dot{x},\xi) + \theta_t(\xi)dt,
\end{equation}
where $x \in \mathcal{L}M$ and $\xi \in T_x\mathcal{L}M = \Gamma(x^*TM)$.

\subsection{Mean index and Conley--Zehnder index}
Let $x \in \mathcal{P}(\theta_t)$ be a contractible periodic orbit. A \emph{capping} of $x$ is a disk $u\colon D^2 \to M$ such that $u|_{\partial D^2} = x$. A \emph{capped periodic orbit} $(x,u)$ consists of a periodic orbit $x\in \mathcal{P}(\theta_t)$ and a capping $u$ of $x$.

Given a capped orbit $(x,u)$, we have a natural trivialization of $x^*TM$ coming from $u$ and, under this trivialization, the linearized flow $d\psi_t \colon T_{x(0)}M \to T_{x(t)}M$ can be interpreted as a path of symplectic matrices $\Psi \colon [0,1] \to Sp(2n)$. For such a path, one can define the \emph{mean index} $\Delta(\Psi) \in \R$ which measures, roughly speaking, the total rotation angle swept by certain eigenvalues of $d\psi_t$ on the unit circle. Moreover, if $x$ is a nondegenerate periodic orbit, the end point $\Psi(1)$ of the path of symplectic matrices is a nondegenerate matrix, i.e., $\Psi(1)$ does not admit $1$ as an eigenvalue. In this case, one can define the $\emph{Conley--Zehnder index}$  $\cz(\Psi) \in \Z$. For a capped orbit $(x,u)$, we define\footnote{Note that $\Delta(\psi_t,x,u)$ is well defined for any capped periodic orbit while $\cz(\psi_t,x,u)$ is defined when the periodic orbit $x$ is nondegenerate.} $\Delta(\psi_t,x,u) = \Delta(\Psi)$ and $\cz(\psi_t,x,u) = \cz(\Psi)$. For the precise definition of $\Delta(\Psi)$ and $\cz(\Psi)$, we refer the reader to \cite{salamon1992morse}.

Now, we recollect some properties of $\Delta(\psi_t,x,u)$ and $\cz(\psi_t,x,u)$ that we shall use in this paper. Let $(x,u)$ be a capped periodic orbit. The $k$-th iteration of $x\colon \R/\Z \to M$ is the orbit $x^k \colon [0,k] \to M, x^k(t)=x(t)$. We denote by $(x^k,u^k)$ the capped orbit obtained by iterating $k$ times the orbit $x$; here $u^k$ is the natural capping of $x^k$ obtained from $u$. 
\begin{enumerate}
	\item (Dependence on the capping) Let $u$ and $v$ be two cappings of a periodic orbit $x$. We have $$\Delta(\psi_t,x,u) = \Delta(\psi_t,x,v) + 2\langle c_1(TM), v \# (-u) \rangle$$ and $$\cz(\psi_t,x,u) = \cz(\psi_t,x,v) + 2\langle c_1(TM), v \# (-u) \rangle,$$ where $c_1(TM) \in \h^2(M;\Z)$ denotes the usual Chern class. \label{cappingdep} \\
	\item (Iteration formula) $\Delta(\psi_t^k,x^k,u^k) = k \Delta(\psi_t,x,u)$. \label{iteformula} \\
	\item (Continuity) Let $\psi^\prime_t$ be a symplectic isotopy obtained from $\psi_t$ by a $C^1$-small perturbation. If $(y,v)$ is a capped periodic orbit for $\psi^\prime_t$, then $$\vert \Delta(\psi_t,x,u) - \Delta(\psi^\prime_t,y,v) \vert$$ is small. \\
	\item (Relation between $\Delta$ and $\cz$) For a nondegenerate periodic orbit $x$ and a capping $u$ of $x$, we have \label{reldeltacz}
	\begin{equation*}
		\vert \cz(\psi_t,x,u) - \Delta(\psi_t,x,u) \vert < n.
	\end{equation*}
	Moreover, if $(x,u)$ splits into $(x_1,u_1),\ldots,(x_m,u_m)$ under a sufficiently $C^1$-small nondegenerate perturbation of $\psi_t$, then 
	\begin{equation*}
		\vert \cz(\psi^\prime_t,x_i,u_i) - \Delta(\psi_t,x,u) \vert \leq n,
	\end{equation*}
	for $i=1,\ldots,m$.
\end{enumerate}

\subsection{Floer--Novikov Homology}\label{hfnsection}
From this point forward, unless otherwise explicitly stated, we suppose for simplicity that $(M,\om)$ is \emph{symplectically Calabi--Yau}, i.e., $c_1(TM)|_{\pi_2(M)} = 0$. In this case, property \eqref{cappingdep} above implies that, when well-defined, neither the mean index nor the Conley--Zehnder index depends on the choice of a capping of the orbit. Hence, hereafter we drop the capping from both notations.

Recall that $\psi_t$ is a symplectic isotopy such that $\psi_0=id$ and $\psi_1=\psi \in \Symp_0(M,\omega)$ and $\mathrm{Flux}[\psi_t] = [\theta]$.

Let $\mathcal{L}_0M$ denote the component of contractible\footnote{We note that, as explained in \cite{van2020floer}, the construction can be adapted for non-contractible loops; however, in this work, we focus on contractible loops} loops on $M$ and fix a base point $p \in M$. As in Section \ref{secmorsenovikov}, let $\pi \colon \overbar{M}^\theta \to M$ be the minimal covering space of $M$ where $\pi^*\theta$ is an exact $1$-form. Again, its covering transformation group is isomorphic to 
$$\Gamma_\theta = \frac{\pi_1(M)}{\ker I_\theta}.$$
In the following, we construct a suitable covering space of $\mathcal{L}_0M$ where the $1$-form $\alpha_{\psi_t}$ becomes exact. Consider the evaluation map
\begin{align*}
	e \colon \mathcal{L}_0M &\to M \\ x &\mapsto x(0)
\end{align*}
and the homomorphism $\mathcal{I}_\theta = I_\theta \circ e_* \colon \pi_1(\mathcal{L}_0M)\to \R$, where $e_*\colon \pi_1(\mathcal{L}_0M) \to \pi_1(M)$ is the homomorphism induced by the evaluation map $e$. Also, consider the homomorphism
\begin{align*}
	\mathcal{I}_{\omega} \colon \pi_1(\mathcal{L}_0M) &\to \R \\ {x_s} &\mapsto \int_{C(\{x_s\})} \omega,
\end{align*}
where $C(\{x_s\})$ is the ``torus'' in $M$ swept out by $\{x_s\}$:
$$(s,t) \in S^1 \times S^1 \mapsto x_s(t) \in M.$$
Let $\Pi \colon \tilde{\mathcal{L}}\overbar{M}^\theta \to \mathcal{L}_0M$ be the covering space associated with
$$\ker \mathcal{I}_\om \cap \ker \mathcal{I}_\theta \subset \pi_1(\mathcal{L}_0M).$$
This covering space can be seen as the set of equivalence classes $[x,u]$ of the pairs $(x,u)$, where $x\in \mathcal{L}_0M$, $u$ is a capping of $x$ and $(x,u) \sim (y,v)$ if, and only if, $x = y$,
\begin{align*}
	\int_u \om = \int_v \om \quad \text{and}\quad\int_{l_{u,0}}\theta = \int_{l_{v,0}} \theta
\end{align*}
Here $l_{u,t}(s):= u(se^{2\pi it})$, for $s\in [0,1]$, and $l_{v,t}$ is defined analogously. From now on, we shall denote by $\overline{x}$ such an equivalence class $[x,u]$. The covering transformation group of $\Pi$ is isomorphic to
$$\Gamma_{\theta,\om} := \frac{\pi_1(\mathcal{L}_0M)}{\ker \mathcal{I}_{\om} \cap \ker \mathcal{I}_\theta}.$$
In this covering space, we choose $\mathcal{A}_{\psi_t} \colon \tilde{\mathcal{L}}\overbar{M}^\theta \to \R$ to be the following primitive for the $1$-form $\alpha_{\psi_t}$:
\begin{equation}\label{action}
	\mathcal{A}_{\psi_t}(\overline{x}) = -\int_u \om + \int_{S^1} H_t(x(t))dt + \int_0^1\int_{l_{u,t}}\theta dt 
\end{equation}
In particular, the set of critical points of the action functional $\mathcal{A}_{\psi_t}$ corresponds to the zero set of $\alpha_{\psi_t}$, which is exactly given by the lifts of the $1$-periodic orbits in $\mathcal{P}(\theta_t)$ to the covering $\tilde{\mathcal{L}}\overbar{M}^\theta$. Denote by $\mathcal{S}(\theta_t)$ the set of critical values of the functional $\mathcal{A}_{\psi_t}$, often called the \emph{action spectrum}.

To define the variant of Floer--Novikov Homology we shall use in this paper, first, suppose that the symplectic diffeomorphism $\psi = \psi_1$ is nondegenerate. Fix $\Q$ as the ground field. For $b \in (-\infty,\infty]$ not in $\mathcal{S}(\theta_t)$, we define
\begin{align*}
	\cfn_k^{<b}(\theta_t) = \Big\{\sum_i a_i \overline{x}_i \ \vert \ a_i \in \Q, \overline{x}_i \in \mathrm{Crit}\mathcal{A}_{\psi_t},\  \text{satisfying the conditions below}\Big\}:
\end{align*}
\begin{itemize}
	\item the set $\{i \ \vert \ a_i\neq 0, \mathcal{A}_{\psi_t}(\overline{x}_i)>c\}$ is finite for any $c \in \R$;
	\item $\cz(\psi_t,x_i) = k$;
	\item $\mathcal{A}_{\psi_t}(\overline{x}_i)<b$.
\end{itemize}
When $b=+\infty$, this $\Q$-vector space is a finitely generated free module over $\Lambda_{\omega,\theta}$, where the latter is the Novikov ring given by the completion of the group ring of $\Gamma_{\theta,\om}$. Equivalently,
\begin{align*}
	\Lambda_{\om,\theta} = \Big\{\sum_i a_i g_i \ \vert \ a_i \in \Q, g_i \in \Gamma_{\theta,\om}, \  \text{satisfying the condition below}\Big\}:
\end{align*}
\begin{itemize}
	\item the set $\{i \ \vert \ a_i \neq 0, (\mathcal{I}_\om+\mathcal{I}_{\theta})(g_i)>c\}$ is finite for any $c\in \R$.
\end{itemize}
Moreover, this vector space is supplied with the Floer boundary operator $\partial$ defined by counting (with signs) the isolated rigid solutions of the asymptotic boundary value problem on cylinders $\nu \colon \R \times S^1 \to M$ defined by the antigradient of $\mathcal{A}_{\psi_t}$. More precisely, the boundary operator counts, modulo $\R$-translation, the finite energy solutions $\nu =~\nu(\tau,t)$ of the \emph{Floer equation}
\begin{equation}\label{floereq}
	\frac{\partial}{\partial \tau} \nu(\tau,t)+J_t\left(\frac{\partial}{\partial t} \nu(\tau,t) - X_t(\nu(\tau,t))\right)=0,
\end{equation}
satisfying $\lim_{\tau \to \pm \infty} \nu(\tau,t) = x_{\pm}(t)$, for periodic orbits $x_{\pm} \in \mathcal{P}(\theta_t)$ such that $\cz(\psi_t,x_-) - \cz(\psi_t,x_+) = 1$. Here the energy of a solution $\nu$ is given by
\begin{equation}\label{energydef}
	E(\nu) = \int_{-\infty}^\infty \left\Vert \frac{\partial \nu}{\partial \tau} \right\Vert_{L^2(S^1)}d\tau=\int_{-\infty}^\infty \int_{S^1}\left\Vert \frac{\partial \nu}{\partial t}-X_t(\nu)\right\Vert dt d\tau .
\end{equation}
Furthermore, in \eqref{floereq}, $J_t$, $t \in S^1$, denotes a $1$-periodic family of $\om$-compatible almost complex structures on $M$ satisfying the standard regularity requirements that hold generically, see \cite{hofer1995floer,van1995symplectic}. Counting the solutions $\nu$ of \eqref{floereq} satisfying $[x_+,u_{-}\#\nu]=[x_+,u_+]$, yields a well-defined boundary operator:
$$\partial_J \colon \cfn_k^{<b}(\theta_t) \to \cfn_{k-1}^{<b}(\theta_t)$$
such that $\partial_J^2=\partial_J \circ \partial_J = 0$. Since this boundary operator decreases the action, i.e., if there exists such a solution $\nu$ of \eqref{floereq} \emph{connecting $x_-$ to $x_+$}, that is, satisfying $\lim_{\tau \to \pm \infty} \nu(\tau,t) = x_{\pm}(t)$, then $\mathcal{A}_{\psi_t}(\overline{x}_-)>\mathcal{A}_{\psi_t}(\overline{x}_+)$, it descends to a boundary operator on the quotient
$$\cfn_k^{(a,b)}(\theta_t) := \frac{\cfn_k^{<b}(\theta_t)}{\cfn_k^{<a}(\theta_t)},$$
where $a,b \notin \mathcal{S}(\theta_t)$, and $-\infty \leq a < b \leq \infty$. The resulting homology
$$\hfn_*^{(a,b)}(\theta_t) = \h_*(\cfn_k^{(a,b)}(\theta_t), \partial_J),$$
is called \emph{filtered Floer--Novikov homology} of the $1$-form\footnote{It is more common to denote Floer--Novikov homology by $\hfn_{*}(\psi_t)$ emphasizing the dependence on the isotopy $\psi_t$; however, in this work, we consider more convenient to highlight the dependence on the $1$-form when using filtered Floer--Novikov homology.} $\theta_t$ and it is independent of the family of almost complex structures $J_t$. Since $\cfn_*^{(b,c)}(\theta_t) = \cfn_*^{(a,c)}(\theta_t)/\cfn_*^{(a,b)}(\theta_t)$ when $a<b<c$, we have a long exact sequence
\begin{equation}\label{leshfn}
	\cdots \to \hfn_*^{(a,b)}(\theta_t) \to \hfn_*^{(a,c)}(\theta_t) \to \hfn_*^{(b,c)}(\theta_t) \to \cdots.
\end{equation}
The total homology
$$\hfn_*(\psi_t)=\hfn_*^{(-\infty,\infty)}(\theta_t)=\h_*(\cfn_k^{<\infty}(\theta_t),\partial_J)$$
is called the \emph{Floer--Novikov homology} of $\psi_t$ and it depends on the flux $[\theta] \in H^1(M;\R)$ but is independent of $\theta_t$.

\begin{thm}[{\cite[Theorem 4.3]{van1995symplectic}} and {\cite[Theorem 2.11]{van2020floer}}]
	Let $\psi_t$ and $\psi_t^\prime$ be two symplectic isotopies with the same flux. Then we have
	$$\hfn_k(\psi_t) \cong \hfn_k(\psi_t^\prime),$$
	for every $k$.
\end{thm}

Moreover, when the flux $[\theta]$ is sufficiently small, one can pick the representative $\theta$ being Morse and $C^1$-small such that the Floer--Novikov chain complex for $\psi_t$ can be described by the Morse--Novikov chain complex for $\theta$. In our convention, a critical point $q$ of a Morse primitive $f_\theta \colon \overline{M}^\theta \to \R$ can be interpreted as a $1$-periodic orbit $q \in \mathcal{P}(\theta_t)$ such that $\cz(\psi_t,q) = \mathrm{index}_{f_\theta}(q) + n$. So, we have the following

\begin{thm}[{\cite[Proposition 2.12]{van2020floer}}]
	If the flux $[\theta]= \mathrm{Flux}[\psi_t]$ is sufficiently small, then 
	$$\hfn_k(\psi_t) \cong \hn_{k+n}(\theta) \otimes_{\Lambda_\theta} \Lambda_{\om,\theta}.$$
\end{thm}

This isomorphism holds only for sufficiently small flux. Lê and Ono define different variants of this (co)homology to compare the rank of Floer--Novikov homology and the rank of the Morse--Novikov homology for the flux $[\theta]$.

\begin{thm}[{\cite[Theorem 3.1]{van2020floer}}]\label{hfnandhn}
	Let $\psi_t \in \Symp_0(M,\om)$ be a symplectic isotopy such that $\psi_0=id$ and $\psi_1$ is a nondegenerate symplectic diffeomorphism. Suppose $\mathrm{Flux}[\psi_t] = [\theta]$. Then we have
	$$\mathrm{rank}_{\Lambda_{\om,\theta}} \hfn_k(\psi_t) = \dim_{\Lambda_\theta} \hn_{k+n}(\theta).$$
\end{thm}

We note that the assertion of Theorem~3.1 in  \cite{van2020floer} is for the total (co)homologies, $\hfn^*$ and $\hn^*$. The same arguments used to prove that theorem show that the equality holds for a fixed degree of the homologies, when the manifold is symplectically Calabi--Yau, since the homotopy maps preserve the degree; see Section~\ref{sec:homotopy}.

As a consequence of these results, Lê and Ono proved the following lower bound for number of fixed points of $\psi$.

\begin{thm}[{\cite[Theorem 1.4]{van1995symplectic}}]\label{boundfix}
	Let $\psi_t$ be a symplectic isotopy on a closed symplectic manifold $(M,\om)$ such that $\psi_0 =id$ and $\psi_1=\psi$ is a nondegenerate symplectic diffeomorphism. Then, the number of contractible fixed points of $\psi$ is bounded from below by $\dim_{\Lambda_\theta}\hn(\theta)$, where $[\theta] = \mathrm{Flux}[\psi_t]$.
\end{thm}


For the particular case of closed orientable surfaces, we have the following.

\begin{cor}\label{rankforsurfaces}
	Let $\Sigma$ be a closed orientable surface of positive genus with area form $\om$ and $\psi_t \in \Symp_0(\Sigma,\om)$ a symplectic isotopy with non-trivial flux such that $\psi_0 = id$ and $\psi_1$ is a nondegenerate symplectic diffeomorphism. Then
	$$\dim_\Q \hfn_*(\psi_t) = \dim_\Q \hfn_0(\psi_t) = 2g-2.$$
	\begin{proof}
		First, we recall that $\pi_2(\Sigma) = 0$. In this case, $\pi_1(\mathcal{L}_0\Sigma) \cong \pi_1(\Sigma)$ and $\Gamma_{\theta,\om} \cong \Gamma_\theta$. Therefore, $\Lambda_{\om,\theta} \cong \Lambda_\theta$, for any $1$-form $\theta$. Then, Theorem \ref{hfnandhn} yields
		
		\begin{align}\label{dimsurface}
			\dim_\Q \hfn_k(\psi_t)=\mathrm{rank}_{\Lambda_{\om,\theta}} \hfn_k(\psi_t)= \dim_{\Lambda_\theta} \hn_{k+1}(\theta),
		\end{align}
		
		for $\theta \in \mathrm{Flux}[\psi_t]$ and for every $k \in \Z$. Furthermore, Corollary \ref{HNranksurfaces} yields that
		\begin{align}\label{hnsurface}
			\dim_{\Lambda_\theta} \hn_{k+1}(\theta) = \begin{cases} \vert\chi(\Sigma)\vert=2g-2, \ \text{if} \ k+1=1 \\ 0, \ \text{otherwise.} \end{cases}
		\end{align}
		Combining \eqref{dimsurface} and \eqref{hnsurface}, we obtain the desired equality.
	\end{proof}
\end{cor}

When $(M^{2n},\om)$ is a closed symplectically Calabi--Yau symplectic manifold, we do not have such a nice equality. In Corollary \ref{rankforsurfaces}, we used the fact that, in the case of surfaces, we have $n=1$. More generally, we have the following. Let $\{\psi_t\} \subset \Symp_0(M,\om)$ be a symplectic isotopy beginning at the identity with a nondegenerate endpoint $\psi_1$ and non-zero flux. We denote by $\supp \hfn_*(\psi_t) \subset \Z$ the set of numbers $k \in \Z$ such that $\hfn_k(\psi_t) \neq 0$ and define, similarly, $\supp \hn_*(\theta)$. Proposition~\ref{vanishHN} and Theorem~\ref{hfnandhn} yield
$$\supp \hfn_{*+n}(\psi_t) = \supp \hn_{*}(\theta) \subset (0, 2n).$$
In particular, we obtain
\begin{equation}\label{suppforcy}
	\supp \hfn_*(\psi_t) \subset (-n,n).
\end{equation}

In the present context, let $\psi_t^k$ be the symplectic isotopy induced by the symplectic vector field $kX_t$ defined by $\om(\cdot,kX_t) = k\theta_t$, for $k>0$. Then, for any $k$ we have 
$$\mathrm{Flux}[\psi_t^k] = k[\theta] = k\mathrm{Flux}[\psi_t].$$
Moreover, for an integer $k$, we have $\psi_1^k = \psi^k = \psi \circ \cdots \circ \psi$, the $k$-th iteration of our symplectic diffeomorphism $\psi$.

\begin{prop}\label{prop:nondegpseudotranslation}
	Let $(M,\om)$ be a closed symplectically Calabi--Yau manifold and $\psi$  a nondegenerate symplectic diffeomorphism which is the end point of a symplectic isotopy $\psi_t$ such that $\psi_0 = id$ and $\mathrm{Flux}[\psi_t] = [\theta]$. Suppose that the number of periodic points of $\psi$ is given by $\dim_{\Lambda_\theta}\hn_*(\theta)$. Then, they are all contractible fixed points with respect to $\psi_t$ and $\Delta(x,\psi_t) = 0$ for all $x \in \mathrm{Fix}_0(\psi_t)$.
		
	\begin{proof}
		The fact that $\mathrm{Per}(\psi) = \mathrm{Fix}_0(\psi_t)$ follows from the lower bound given in Theorem \ref{boundfix}. For any integer $k \geq 1$ such that $\psi^k$ is again nondegenerate, we have:
		\begin{align*}
			\mathrm{rank}_{\Lambda_\theta}\hn_*(\theta) = \# \mathrm{Fix}(\psi^k) \geq \# \mathcal{P}(k\theta_t) \geq \mathrm{rank}_{\Lambda_{\omega,\theta}} \hfn_*(\psi^k_t) = \mathrm{rank}_{\Lambda_\theta} \hn_{*+n}(k\theta),
		\end{align*}
		where, for the first inequality, we regard $\mathcal{P}(k\theta_t)$ as a subset of $\mathrm{Fix}(\psi^k)$, the second inequality comes from the fact that $\cfn_*(\psi^k_t)$ is freely generated, over $\Lambda_{\omega,\theta},$ by the set $\mathcal{P}(k\theta_t)$, and the last equality follows from Theorem~\ref{hfnandhn}. In particular,
		$$\mathrm{rank}_{\Lambda_{\om,\theta}} \hfn_*(\psi^k_t) = \dim_{\Lambda_\theta}\hn_*(\theta)$$
		and every point in $\mathrm{Fix}(\psi^k)$ contributes non-trivially to the Floer--Novikov homology. Suppose by contradiction that $\psi$ admits a fixed point $x$ such that $\Delta(\psi_t, x) = m \neq 0$. From the iteration formula \eqref{iteformula} and property \eqref{reldeltacz}, we have $\vert k\Delta(\psi_t, x) - \cz(\psi_t^k, x^k) \vert \leq n$ and hence, $\cz(\psi_t^k, x^k)$ lies in the interval $(km-n, km+n)$ for any integer $k \geq 1$ such that $\psi^k$ is nondegenerate. Therefore, for $k > \vert 2n/m \vert$, we have $\cz(\psi_t^k, x^k)$ outside $(-n,n)$, which contradicts the fact that $\supp \hfn_*(\psi^k_t) \subset (-n,n)$, as seen in \eqref{suppforcy}.  
	\end{proof}
\end{prop}

The discussion above extends to all symplectic diffeomorphisms as follows. For simplicity, suppose that $(M,\om)$ is \emph{rational} -- there is a constant $\lambda_0 \geq 0$ such that $\langle [\om], \pi_2(M) \rangle = \lambda_0 \Z$ -- and the periods $I_\theta(\pi_1(M))$ of $\theta$ are discrete. In this case, $\lambda_0$ is called the \emph{rationality constant} of $(M,\om)$. Let $\psi_t$ be a symplectic isotopy such that $\psi_0 = id$ and $\psi_1=\psi$. For $a,b \notin \mathcal{S}(\theta_t)$, we define
$$\hfn_*^{(a,b)}(\theta_t) := \hfn_*^{(a,b)}(\theta_t^\prime),$$
where $\psi_t^\prime$ is a small nondegenerate perturbation of $\psi_t$ such that $\om(\cdot, X_t^\prime) = \theta_t^\prime$ and $\mathrm{Flux}[\psi_t] = \mathrm{Flux}[\psi_t^\prime]$. Equivalently, $\psi_t^\prime = \psi_t \circ \varphi_t^H$, for some small Hamiltonian isotopy $\{\varphi^H_t\} \subset \Ham(M,\om)$. When $(M,\om)$ is not rational or $I_\theta(\pi_1(M))$ is dense in $\R$, $\hfn_*^{(a,b)}(\theta_t)$ can be defined as the limit over a certain class of Hamiltonian perturbations of $\psi_t$, see \cite[Remark 2.3]{ginzburg2009action} and \cite[Section 2.3]{hein2012conley}.

\begin{remark}\label{rmk=ham}
	When $\psi \in \Symp_0(M,\omega)$ is a Hamiltonian diffeomorphism, we have $[\theta] = 0$, i.e., $\theta_t = dH_t$ for some $1$-periodic Hamiltonian $H \colon S^1 \times M \to \R$, with $H(t,\cdot) = H_t$. In this case, $\overline{M}^\theta = M$, the action functional $\mathcal{A}_{\psi_t}$ agrees with the Hamiltonian action functional, and the construction above coincides with the standard Hamiltonian Floer homology setting, in particular, $\hfn_k^{(a,b)}(\theta_t) =\hf_k^{(a,b)}(H)$, for any $-\infty \leq a<b\leq \infty$.
\end{remark}

\subsection{Homotopy maps}
\label{sec:homotopy}
The fact that the Floer--Novikov homology of $\psi_t$ is independent of $\theta_t$ (but depends on the Flux $[\theta]$ of $\psi_t$) follows from a more general result concerning the construction of the often called \emph{homotopy maps} that we recall now.

Let $\psi_{s,t}$ be a family of symplectic isotopies starting at the identity, i.e., $\psi_{s,0} = id$, smoothly parametrized by $s \in [0,1]$, and such that $\mathrm{Flux}[\psi_{s,t}] = [\theta]$, for all $s$. Suppose that $\psi_{0,1}$ and $\psi_{1,1}$ are nondegenerate symplectic diffeomorphisms. Using again the Deformation Lemma \cite[Lemma~2.1]{van1995symplectic}, we can write
$$\theta_{s,t}:= \om(\cdot, X_{s,t}) = \theta + dH_{s,t},$$
where $X_{s,t}$ is the symplectic vector field generating the isotopy $\psi_{s,t}$, for each $s \in [0,1]$. It turns out that one can construct a chain map between the chain complexes defining $\hfn_*(\psi_{0,1})$ and $\hfn_*(\psi_{1,1})$ using the family $\psi_{s,t}$. In fact, let
$$c \geq \int_{-\infty}^\infty \int_{S_1} \max_M \frac{\partial}{\partial s} H_{s,t} \ dt \ ds.$$
Following the proof of \cite[Theorem 4.3]{van1995symplectic}, one defines a chain map
\begin{equation}\label{homotopymap}
	\Psi_{H_0,H_1} \colon \cfn^{(a,b)}_k(\theta_{0,1}) \to \cfn^{(a+c,b+c)}_k(\theta_{1,1})
\end{equation}
counting the rigid finite energy solutions $\nu \colon \R \times S^1 \to M$ of the equation
\begin{equation}\label{paramfloereq}
	\frac{\partial}{\partial s} \nu(s,t)+J_{s,t}\left(\frac{\partial}{\partial t} \nu(s,t) - X_{s,t}(\nu(s,t))\right)=0,
\end{equation}
satisfying $\lim_{\tau \to \pm \infty} \nu(\tau,t) = x_{\pm}(t)$, for $x_- \in \mathcal{P}(\theta_{0,t})$ and $x_+ \in \mathcal{P}(\theta_{1,t})$ such that $\cz(\psi_t,x_+) = \cz(\psi_t,x_-)$. Here $J_{s,t}$ denotes a $s$-dependent family of almost complex structures compatible with $\om$. It is straightforward to check that for the energy $E(\nu)$ defined in \eqref{energydef}, we have 
\begin{align*}
	E(\nu) - (\mathcal{A}_{\psi_{0,t}}(\overline{x}_-) - \mathcal{A}_{\psi_{1,t}}(\overline{x}_+)) &= \int_{-\infty}^\infty \int_{S_1} \frac{\partial}{\partial s} H_{s,t} dt ds \\ &\leq \int_{-\infty}^\infty \int_{S_1} \max_M \frac{\partial}{\partial s} H_{s,t} dt ds,
\end{align*}
for a solution $\nu$ of \eqref{paramfloereq}, where $\overline{x}_-=[x_-,u_-]=[x_+,u_+\#\nu]=\overline{x}_+$.

The chain map \eqref{homotopymap} descends to homology and we denote the obtained map again by $\Psi_{H_0,H_1}$. This map is an isomorphism in the case of total homologies $\hfn_k(\psi_{0,1})$ and $\hfn_k(\psi_{1,1})$, but, in general, the filtered homologies $\hfn^{(a,b)}_k(\theta_{0,1})$ and $\hfn^{(a+c,b+c)}_k(\theta_{1,1})$ are not isomorphic. In particular, if $c>b-a$, one has $\Psi = 0$. Moreover, if the homotopy is monotone decreasing homotopy, i.e., $(\partial/\partial s) H_{s,t} \leq 0$, we can take $c=0$ and, hence, $\Psi_{H_0,H_1}$ preserves the action filtration.

Furthermore, any homotopy $H_{s,t}$ from $H_{0,t}$ to $H_{1,t}$ such that $a,b$ are not in $\mathcal{S}(\theta_{s,t}),$ for all $s$, induces a natural isomorphism (in general, different from $\Psi$) between the groups $\hfn_k^{(a,b)}(\theta_{s,t})$ and, thus,
\begin{equation}\label{homotiso}
	\hfn_k^{(a,b)}(\theta_{0,t}) \cong \hfn_k^{(a,b)}(\theta_{1,t});
\end{equation}
see \cite[Proposition 1.1]{viterbo1999functors} and also \cite[Section 4.4]{biranpolsal}. These isomorphisms commute with the maps in the long exact sequence \eqref{leshfn} when the real numbers $a<b<c$ are outside $\mathcal{S}(\theta_{s,t})$.

As in \cite[Example 2.5]{ginzburg2009action}, the family $\theta_{s,t}$ is called \emph{isospectral} if the action spectrum $\mathcal{S}(\theta_{s,t})$ is independent of $s$. In the case of an isospectral family, isomorphism \eqref{homotiso} holds for any $a,b\notin \mathcal{S}(\theta_{s,t})$ with $a<b$. It follows from the construction of this isomorphism that if $K$ is another periodic Hamiltonian such that $K_t \geq H_{s,t}$, for all $s$, then we have the following commutative diagram
\begin{equation}\label{comdiag}
	\begin{tikzcd}
		\hfn_k^{(a,b)}(\theta + dK_t) \arrow[rd,"\Psi_{K,H_1}"] \arrow[d,"\Psi_{K,H_0}"] \\ \hfn_k^{(a,b)}(\theta_{0,t}) \arrow[r,"\sim"] & \hfn_k^{(a,b)}(\theta_{1,t}),
	\end{tikzcd}
\end{equation}
where the maps $\Psi_{K,H_i}$ are the homotopy maps induced by monotone \emph{linear homotopies} between $K$ and $H_0$, and between $K$ and $H_1$, that is, homotopies of the form $K^i_{s,t}=~(1-~s)K_t + s H_{i,t}$, for $i=0,1$ and $s \in [0,1]$.

\subsection{Local Floer Homology}
\label{sec:localFH}
We review the definition of local Floer homology of a fixed isolated $1$-periodic orbit. As before, let $\psi \in \Symp_0(M,\om)$ be a (possibly degenerate) symplectic diffeomorphism, $\psi_t$ a symplectic isotopy such that $\psi_0 = id$ and $\psi_1 = \psi$ and $X_t$ the symplectic vector field such that
$$\theta_t := \om(\cdot, X_t) = \theta + dH_t,$$
where $\mathrm{Flux}[\psi_t] = [\theta]$. Given an isolated $1$-periodic orbit $x \in \mathcal{P}(\theta_t)$, consider a small tubular neighborhood $U\subset M$ of $x$ (more precisely, $U$ is an isolating neighborhood of $S^1 \times \{x(t)\}$ in $S^1 \times M$ such the closure does not intersect $S^1 \times \{y(t)\}$ for any other orbit $y \in \mathcal{P}(\theta_t)$) and let $\psi^\prime_t$ be a nondegenerate $C^1$-small perturbation of $\psi_t$ with $\mathrm{Flux}[\psi_t^\prime] = [\theta]$ and supported in $U$. We can choose the neighborhood $U$ and the perturbation $\psi_t^\prime$ such that the solutions of the Floer equation between periodic orbits contained in $U$ for $\psi_t^\prime$ are themselves also contained in $U$. By compactness and gluing results, the same holds for \emph{broken} solutions.

In particular, this allows one to define a chain complex, as described in Section~\ref{hfnsection}, considering only the generators contained in $U$. We denote the resulting homology, which is independent of the perturbation, by $\hf_*^{\loc}(\psi_t,\overline{x})$. This notion goes back to Floer's original works \cite{Floer3,floer1989witten}; see also \cite[Section 3.2]{ginzburg2010local}.

Since we are assuming $M$ to be symplectically Calabi--Yau, this homology is also $\Z$-graded by the Conley--Zehnder index in the nondegenerate case. When $x$ is a nondegenerate periodic orbit, it follows from the definition of the local Floer homology that
$$\hf_k^{\loc}(\psi_t,\overline{x}) = \begin{cases} \Q, \ \text{if} \ k = \cz(\psi_t,x) \\ 0, \ \text{otherwise}. \end{cases}$$
Moreover, the relation between the mean index $\Delta(\psi_t,x)$ and the Conley--Zehnder index $\cz(\psi_t,x)$ in \eqref{reldeltacz} yields
\begin{equation}\label{supphfloc}
	\mathrm{supp} \hf^{\loc}_*(\psi_t,\overline{x}) \subset [\Delta(x,\psi_t) - n, \Delta(x,\psi_t) + n],
\end{equation}
even in the case where $x$ is degenerate.

\subsection{A canonical Floer-Novikov complex}
\label{sec:canonical}
Recall the assumption $(M,\om)$ is rational and symplectically Calabi-Yau. Let $\psi_t$ be a symplectic path with flux $\theta$ and a possibly degenerate endpoint $\psi=\psi_1$ with finitely many points. For each $x\in\fix_0(\psi_t)$, we fix a choice of capping $\ol{x}$ and an isolating neighborhood $U_x$ (as in Section \ref{sec:localFH}). Furthermore, let $\psi'_t$ be a sufficiently close Hamiltonian perturbation of $\psi_t$ such that $\psi'_1$ coincides with $\psi$ away from the isolating neighborhoods. Consider the universal Novikov ring
\begin{equation*}
	\Lambda =\Big\{\sum a_iT^{\lambda_i}\,|\,a_i\in\Q, \lambda_i\in\R,\lambda_i\rightarrow+\infty\Big\},
\end{equation*}
and its subring $\Lambda^0\subset\Lambda$ given by elements in $\Lambda$ satisying $\lambda_i\geq0$ for all $i$. Moreover, we denote by $\Lambda^0_{\om,\theta}$ the subring of $\Lambda_{\om,\theta}$ given by its elements satisying $(\cI_{\theta}+\cI_{\om})(g_i)\geq0$ $\forall i$. Note that the map $$g_i\in\Lambda^{0}_{\om,\theta}\longmapsto T^{(\cI_\om+\cI_\theta)(g_i)}\in\Lambda^{0}$$ gives a natural embedding. In this section, we denote by $$\cfn(\psi'_t,\Lambda_{\om,\theta})\quad\text{and}\quad\cfn(\psi'_t,\Lambda^0_{\om,\theta})$$ the Floer--Novikov complex defined in Section \ref{hfnsection} and the analogous $\Lambda^0_{\om,\theta}$ version, respectively. We follow the construction in \cite[Section 4.4.7]{shelukhin2022hofer}, which was upgraded to the semi-positive setting in \cite[Section4]{sugimoto2021hofer}, to obtain a canonical $\Lambda^0$-complex for which the properties we use are listed in the following proposition.

\begin{prop}
\label{prop:canonical}
Let $(M,\om)$ be a closed rational symplectically Calabi-Yau symplectic manifold. There is a homotopically canonical $\Lambda^0$-complex $\cfn(\psi_t,\Lambda^0)$ that satisfies the following properties:
	\begin{enumerate}
		\item As a graded $\Lambda^0$-module \[\cfn_*(\psi_t,\Lambda^0)=\bigoplus_{x\in\fix_0(\psi_t)}\hf^{\loc}_*(\psi_t,\ol{x})\otimes\Lambda^0.\]
		\item Its differential ${d}$ is defined over $\Lambda^0$. Moreover, if $d$ is non-trivial then there exists at least one integer $k$ such that $\cfn_k(\psi_t,\Lambda^0)$ and $\cfn_{k+1}(\psi_t,\Lambda^0)$ are non-trivial.
		\item The homology $\hfn(\psi_t,\Lambda)$ of \[\cfn_*(\psi_t,\Lambda):=\cfn_*(\psi_t,\Lambda^0)\otimes_{\Lambda^0}\Lambda\] is isomorphic to $\hfn(\psi'_t,\Lambda) = H(\cfn(\psi'_t,\Lambda_{\om,\theta})\otimes_{\Lambda_{\om,\theta}}\Lambda,\partial)$.
	\end{enumerate}
\end{prop}

\begin{proof}
The construction of the canonical complex in \cite[Section4]{sugimoto2021hofer} is for Hamiltonian diffeomorphisms in the semipositive setting where, just as in our case, the action spectrum is possibly dense in $\R$. One difference is that, since we are in the symplectically Calabi-Yau setting, the modules in the statement of Proposition \ref{prop:canonical} have a well-defined grading. To guarantee the same construction goes through in our setting, it is enough to show that the Floer--Novikov differential  $\partial$ of $$\cfn(\psi'_t,\Lambda^0)=\cfn(\psi'_t,\Lambda^0_{\om,\theta})\otimes_{\Lambda^0_{\om,\theta}}\Lambda^0$$ decomposes as $$\partial= d_{\loc,\psi'_t} + D,$$ where $d_{\loc,\psi'_t}$ is the direct sum of differentials in the local Floer complexes $\cf^{\loc}(\psi'_t,\ol{x})$ taken with $\Lambda^0$-coefficients and $D$ shifts action down by at least an $\epsilon>0$ depending only on $\psi_t$ and not on the choice of a sufficiently close perturbation $\psi'_t$. 

We now show that this is the case by finding an $\epsilon>0$, depending only on $\psi_t$, which bounds from below the energy of a Floer trajectory $v$, counted by $\partial$ on $\cfn(\psi'_t, \Lambda^0)$, between a capped orbit $\ol{x}'\in\cO(\psi'_t,\ol{x})$ and $\ol{y}'\in\cO(\psi'_t,\ol{y})$. Here, we include the possibility that $\ol{x}=[x,u]$ and $\ol{y}=[x,w]$ are different cappings of the same orbit  $x=y$. In case $x\neq y$, $E(v)$ is bounded below by the crossing energy that can be estimated as in \cite[Lemma~3.5]{van1995symplectic}. In fact, it is possible to find an $\eps>0$ such that if $v$ escapes any isolating neighborhood, then $E(v)>\eps$ by a standard Gromov compactness argument. So it remains to deal with the case $\ol{x}=[x,u]$ and $\ol{y}=[x,w]$ are different cappings of the same orbit  $x=y$ and $v$ remains inside $U_x$, connects $\ol{x}'$ and $\ol{y}'$, and satisfies $E(v)<\eps$. The cappings not being equivalent implies that either:
\begin{equation}\
\label{eq:equivalent-cap}
	\int_{u\#(-w)}\om =m\lambda_0 \neq 0\quad\text{or}\quad\int_{l_u\#(-l_w)}\theta\neq 0,
\end{equation}
or, possibly, both. Here $\lambda_0$ denotes the rationality constant of $(M,\om)$ and $m$ is an integer. Let $u',w'$ be the cappings inherited from $u,w$, that is, concatenations of $u$ and $w$ with cylinders $C_x,C_y$, contained in $U_x$ and with small symplectic area compared to $\lambda_0$, connecting $x$ to $x'$ and $y'$, respectively. The condition that $v$ is counted by $\partial$, implies that $[y',u'\#v]=[y',w']$ and, hence,
 \begin{equation*}
	m\lambda_0 + \int_{C_x\#v\#(-C_y)}\om=0 \quad\text{and}\quad \int_{l_{u'}\#l_{v}\#(-l_{w'})}\theta=0.
\end{equation*}
The first equality implies that $m=0$ since the second term on the right-hand side is small compared to $\la_0$. Therefore, the second assertion of \eqref{eq:equivalent-cap} must hold, which is a contradiction to the fact that the cappings are not equivalent. Indeed, the loop
\begin{equation*}
	{l_{u'}\#l_{v}\#(-l_{w'})} = {l_{u}\# l_{C_x}\#l_{v}\#(-l_{C_y})\#(-l_{w})}
\end{equation*}
is homotopic to ${l_{u}\#(-l_{w})}$ since the loop $l_{C_x}\#l_{v}\#(-l_{C_y})$ is clearly contractible in $M$. Thus, the differential decomposes as desired and the construction proceeds verbatim as in 
\cite[Section4]{sugimoto2021hofer}. 

That such a construction of a canonical complex is possible automatically implies parts $(1)$, $(3)$, and the first part of $(2)$. It remains to prove the second part of $(2)$ above. This follows from the fact that the differential $d$, obtained in the basic perturbation lemma of \cite{Markl}, factors through $D$. Therefore, if $d$ is not trivial, so is $D$. Since we have a well-defined grading and $D$ is of degree $1$, there must be at least two non-trivial elements of $\cfn(\psi_t,\Lambda^0)$ of consecutive degrees.
\end{proof}

Now we are ready to prove Proposition~\ref{prop:pseudotranslation} which generalizes Proposition~\ref{prop:nondegpseudotranslation}.

\begin{proof}[Proof of Proposition \ref{prop:pseudotranslation}]
Suppose $\psi$ is a pseudotranslation. Let $\psi_t$ and the sequence $k_j$ be as in Definition \ref{def:pseudotranslation}. Note that the condition of being a pseudotranslation, and part $(3)$ of Proposition \ref{prop:canonical} imply that the differential $d$ of the canonical complex $\cfn(\psi_t^{k_j},\Lambda)$ is trivial for all $j$. On the other hand, the fact that $(M,\om)$ is symplectically Calabi-Yau, implies that degrees where $\hfn_*(\psi_t^{k_j},\Lambda)$ is non-trivial is supported in a compact interval $I\subset \R$. If there exists $x$ such that $\Delta(\psi_t,x)\neq0$ and $\hf^{\loc}(\psi_t,x)\neq0$, then for a sufficiently large $j$ we have that $k_j\Delta(\psi_t,x)$ will belong to the complement of $I$. In particular, the support of $\hf^{\loc}(\psi^{k_j}_t,x^{k_j})$ is disjoint from $I$. This is in contradiction to the fact that the differential is trivial. 
\end{proof}

\subsection{Blowing-up and gluing}\label{sec:blow-up_and_gluing}
Let $\psi_t$ be a symplectic isotopy of $(\Sigma,\om)$ which is based at the identity and whose endpoint $\psi$ has isolated fixed points. Up to composing by a Hamiltonian loop, we may suppose that all of the contractible fixed points are constant. Fix a subset $E$ of $\fix_0(\psi_t)$. The purpose of this section is twofold. First, to define a symplectic isotopy $\psi^{+}_{t}$ on a surface $\Sigma_+$ whose boundary components are obtained by ``blowing-up" $\Sigma$ at each point in $E$ such that $\psi^{+}_{t}$ is conjugate to $\psi_t$ in the interior of $\Sigma_+$ and the dynamics on each boundary component is determined by $D\psi$. We similarly define $\psi^{-}_t$ on $\Sigma_{-}$ by first conjugating $\psi_t$ by the anti-symplectic involution $\tau:\Sigma\rightarrow\Sigma$, locally modelled on $(x,y)\mapsto(-x,y)$, and \emph{blowing-up} $\Sigma$ at each point in $\tau(E)$. Secondly, we glue $\Sigma_{+}$ and $\Sigma_{-}$ along the boundary components corresponding to the same point in $E$ to obtain a new surface $\Sigma'$. The symplectic isotopy $\psi_t'$ defined by $\psi_t^{\pm}$ on $\Sigma_{\pm}\subset\Sigma^{'}$ has $C^1$-regularity. Similar ideas have played an important role in the study of Hamiltonian dynamics on surfaces; see \cite{arnol2013mathematical, franks1990periodic, collier2012symplectic} for a few examples.

\subsubsection{Blowing-up}\label{sec:blow-up}
Fix $\delta>0$ such that $B_{r}(z)$ is a Darboux ball for all $r<2\delta$ and $z\in E$, and such that $B_{r}(z)\cap\fix(\psi)=\{z\}$. It is possible to find a symplectic diffeomorphism between the punctured surface $\Sigma\setminus E$ and the interior of a surface $\Sigma_{+}$ whose boundary components correspond to the punctures. Indeed, on $B_{\delta}(z)$, the identification is given by the  symplectic diffeomorphism between the punctured disc $(B_{\delta}(0)\setminus\{0\}, rdr\wedge d\theta)$ and the open cylinder $((0,\delta^2/2)\times\R/2\pi\Z, d\rho\wedge d\theta)$ taking $(r,\theta)$ to $(\rho=r^2/2,\theta)$. The annuli $B_{\delta\leq r\leq 2\delta}(z)$, $z\in E$, are used to ``interpolate" between the two symplectic structures. To this end, let \[f: [\delta,2\delta]\rightarrow[\delta^2/2,2\delta]\] be a diffeomorphism satisfying the following:
\begin{enumerate}
	\item There exists a positive $\epsilon\ll\delta$ such that $f(x)=x^2/2$ on $[\delta,\delta+\epsilon]$ and $f(x)=x$ on $[2\delta-\epsilon,2\delta]$.
	\item The derivative $f'(x)$ is positive for all $x\in[\delta,2\delta]$.
\end{enumerate}
Note that the second condition when added to the fact that $f$ is positive implies that the product $h(x):=f^{-1}(x)(f^{-1})'(x)$ is positive for all $x\in[\delta^2/2,2\delta]$. In particular, we have the following symplectic diffeomorphism  
\begin{align*}
	(B_{\delta\leq r\leq 2\delta}(0), rdr\wedge d\theta)&\longrightarrow([\delta^2/2,2\delta]\times\R/2\pi\Z, h(\rho)d\rho\wedge d\theta)\\
	(r,\theta)&\longmapsto(\rho=f(r),\theta).
\end{align*}
We obtain a symplectic diffeomorphism that identifies each punctured disk $B_{2\delta}(z)\setminus\{z\}$ with a smoothed out cylindrical end and is the identity in the complement of $\cup_{z\in E}B_{2\delta}(z)$. We denote by $\Sigma_{+}$ the closure of the resulting surface. 

In the interior of $\Sigma_+$ we define $\psi_t^{+}$ by conjugating the restriction of $\psi_t$ to $\Sigma\setminus E$ by the symplectic diffeomorphism detailed in the previous paragraph. On the boundary components, which are identified with $\{2\delta\}\times\R/2\pi\Z$, we define \[\psi_t^{+}(2\delta,\theta)=\lim_{\rho\rightarrow 2\delta}\psi_t^{+}(\rho,\theta).\]
Under the identification of $\{2\delta\}\times\R/2\pi\Z$ with the unit vectors in $T_{z}\Sigma=T_{0}\R^2$ the restriction of $\psi_t^{+}$ is given by $(D\psi_t)_{z}/||(D\psi_t)_{z}||$.

The surface $\Sigma_{-}$ is the \emph{reflected} copy of $\Sigma_{+}$. It is obtained by ``blowing-up" $\Sigma$ at each $z\in\tau(E)$. We define $\psi_t^{-}$ just as in the previous case with the only difference being that we consider $\tau\circ\psi_t\circ\tau$ in place of $\psi_t$.  

\subsubsection{Gluing}\label{sec:gluing}
The objective of this section is to induce a symplectic isotopy $\psi_t'$ on a closed orientable surface $\Sigma'$, constructed by \emph{gluing} $\Sigma_{+}$ to $\Sigma_{-}$, in a way that extends the symplectic isotopies $\psi_t^{+}$ and $\psi_t^{-}$. The surface $\Sigma'$ is obtained by identifying, for all $z\in E$, the connected component of $\partial\Sigma_{+}$ corresponding to $z$ with the connected component of $\partial\Sigma_{-}$ corresponding to $\tau(z)$. 

Since the problem of gluing is local and $\psi_t$ is Hamiltonian on a neighborhood of each $z\in\ E$, we work out the details for a Hamiltonian isotopy $\phi_t$ of $(\R^2,dx\wedge dy)$ which fixes the origin and is generated by $H$. Without loss of generality, we may suppose that $H_t(0)=0$ and $\partial H_t/\partial t(0)=0$ for all $t$. In this case, the \emph{blowing-up} procedure described in Section \ref{sec:blow-up} is simpler; we identify $\R^2\setminus\{0\}$ with $\R_{>0}\times\R/2\pi\Z$ by mapping $
(r,\theta)$ to $(\rho=r^2/2,\theta)$ and obtain $\phi_t^{+}$ and $\phi_t^{-}$ on $C_{\pm}:=\R_{\geq0}\times\R/2\pi\Z$ as before. Note that since $\tau(x,y)=(-x,y)$, locally, $\tau\circ\phi_t\circ\tau$ is generated by $-H_t(-x,y)$. 

Consider an infinite cylinder $C=\R\times\R/2\pi\Z$. Map $C_{+}$ to $C$ by the natural inclusion and $C_{-}$ to $C$ by $\mu:(\rho,\theta)\mapsto(-\rho,-\theta)$. Define the function:
\begin{align*}
	F_t(\rho,\theta) = 
	\begin{cases} 
		H_{t}(\sqrt{2\rho}\cos\theta,\sqrt{2\rho}\sin\theta),&\rho>0 \\ 
		-H_{t}(-\sqrt{-2\rho}\cos\theta,-\sqrt{-2\rho}\sin\theta),&\rho<0\\ 
		0, &\rho=0,
	\end{cases}
\end{align*}
which has $C^1$-regularity and induces a $C^1$-symplectic isotopy $\phi'_t$ that coincides with $\phi_t^{+}$ on $C_+$ and with $\mu\circ\phi_t^{-}\circ\mu$ on $\mu(C_{-})$.

\begin{remark}
	We note that $C^{1}$-regularity is the most one can obtain in general when using this method of blowing-up and gluing. For example, if $H(x,y)=x(x^2+y^2)$ then $F(\rho,\theta)=(2\rho)^{3/2}\cos\theta$ for all $\rho>0$. In particular, the second partial derivative $\partial_{\rho}^{2}F(\rho,\theta)$ diverges as $\rho$ approaches $0$ from the right.
\end{remark}

\section{Symplectically degenerate extrema and periodic orbits}
In this section, we describe how the presence of a symplectically degenerate extremum implies the existence of infinitely many periodic orbits of a symplectic diffeomorphism $\psi \in \Symp_0(M,\om)$. We follow closely the ideas of Ginzburg in the proof of the Conley Conjecture \cite{ginzburg2010conley} and its refinements by Ginzburg and G{\"u}rel \cite{ginzburg2009action}, which rely on Hamiltonian Floer homology, and we adapt the arguments to nonzero flux symplectic diffeomorphisms using Floer--Novikov homology. The presence of a symplectic degenerate maximum (SDM) - see Definition~\ref{defn:SDM} - yields the following nontriviality at the homology level.

\begin{thm}[cf. {\cite[Theorem 1.17]{ginzburg2009action}}]\label{thm:SDMHFN}
	Assume that $(M^{2n},\om)$ is a closed rational symplectically Calabi--Yau manifold and let $\overline{x}$ be a symplectic degenerate maximum of $\psi_t$. Set $c = \mathcal{A}_{\psi_t}(\overline{x})$. Then for every sufficiently small $\varepsilon>0$, there exists $k_\varepsilon$ such that
	\begin{equation}
		\hfn_{n+1}^{(kc+\delta_k,kc+\varepsilon)}(k\theta_t)\neq 0 \ \text{for all $k> k_\varepsilon$ and some $\delta_k \in (0,\varepsilon)$.}
	\end{equation}
\end{thm}
Theorem~\ref{thm:SDMHFN} implies the existence of infinitely many periodic points of $\psi$ under the assumption that at least one symplectic degenerate maximum exists.

\begin{thm}[cf. {\cite[Theorem 1.18]{ginzburg2009action}}]\label{thm:infty}
Let $(M^{2n},\om)$ be a closed rational symplectically Calabi--Yau manifold and $\psi_t$ a symplectic isotopy such that $\psi_0 = id$. Assume that $\psi_t$ has a symplectic degenerate maximum and that $\psi_1=\psi$ has finitely many fixed points. Then $\psi$ has a simple $p$-periodic point for each sufficiently large prime $p$, which is contractible with respect to $\psi_t^p$.
\end{thm}

The proofs of Theorem~\ref{thm:SDMHFN} and Theorem~\ref{thm:infty} follow the proofs of \cite[Theorem~1.17]{ginzburg2009action} and \cite[Theorem~1.18]{ginzburg2009action}, respectively. The main arguments are explained in Section~\ref{subsection:proofinfty}.

In Section~\ref{subsection:decomposition} we describe, as in \cite[Section~5.1]{ginzburg2009action}, a direct sum decomposition of the filtered Floer--Novikov homology for a particular type of $1$-form associated with the vector field that generates the symplectic isotopy. In Section~\ref{subsection:SDM}, we present the definition and a geometrical characterization of symplectically degenerate maxima.

\begin{remark}
In the proof of Theorem \ref{thm:main_surfaces}, we also use the fact that the presence of a symplectic degenerate minimum (SDm) - see Definition \ref{defn:SDM} - implies the existence of infinitely many periodic orbits of $\psi \in \Symp_0(M,\om)$. This follows from Theorem \ref{thm:infty} and the fact that if $y$ is a SDm of $\psi_t$, the closed curve $x(t) = y(-t)$ is a SDM of $\psi_t^{-1}$.
\end{remark}

\subsection{Decomposition of filtered Floer--Novikov homology}\label{subsection:decomposition}
Suppose that $(M,\om)$ is a geometrically bounded\footnote{For the definition of geometrically bounded see, e.g. \cite{cieliebak2004symplectic}. This condition is implicitly used in the proof of Theorem \ref{thm:SDMHFN} when $M = \R^{2n}$.} rational symplectic manifold. Let $U \subset W \subset M$ be two open sets with smooth boundary and compact closure. We assume that the closed sets $\overline{U}$ and $\overline{W}$ are isotopic in $M$, in fact, we shall assume that $W$ is a small neighborhood of $\overline{U}$. Consider a Hamiltonian function $F,$ which is constant on $M\setminus U$, say $F|_{M\setminus U} = C$. For a $1$-form $\eta$ that vanishes on $W$, let $\eta_F = \eta + dF$ and, as before, we denote by $X_t$ the symplectic vector field such that $\eta_F = \om(\cdot, X_t)$, and by $\psi_t$ the symplectic isotopy corresponding to $X_t$. Furthermore, consider $a,b \notin \mathcal{S}(\eta_F)$ such that $a<b$ and
$$(C+\lambda_0 \Z) \cap (a,b) = \emptyset,$$
where $\lambda_0$ is the rationality constant of $(M,\om)$. Under these conditions, the filtered homology $\hfn_*^{(a,b)}(\eta_F)$ is well-defined; see Section \ref{hfnsection}. 

Suppose, for now, that all $1$-periodic orbits in $\mathcal{P}(\eta_F)$ with action in the interval $(a,b)$ are nondegenerate. Denote by $\cfn^{(a,b)}_*(\eta_F;U)$ the vector space, over $\Q$, generated by such orbits with cappings equivalent to those contained in $U$. In particular, such generators are contained in $U$. Let $\cfn^{(a,b)}_*(\eta_F;M,U)$ be the vector space generated by the remaining capped $1$-periodic orbits with action in $(a,b)$.

\begin{remark}\label{rmk:decomp_MU}
In contrast with \cite[Section 5.1]{ginzburg2009action}, the orbit of a capped $1$-periodic orbit in $\cfn^{(a,b)}_*(\eta_F;M,U)$ may not be contained in $U$. \emph{A priori}, a generator of $\cfn^{(a,b)}_*(\eta_F;M,U)$ is a capped $1$-periodic orbit with action in $(a,b)$ where either the orbit is contained in $U$ and its capping is not equivalent to a capping contained in $U$ or the orbit is outside $U$. Note that there are no orbits not contained in $U$ that intersect $U$, since $\eta=0$ on $W$ and $F$ is constant in $W\setminus U$. On the other hand, in light of Remark \ref{rmk=ham}, it is straightforward to check that $\cfn^{(a,b)}_*(\eta_F;U)$ coincides with the vector space $\cf^{(a,b)}_*(F;U)$ defined in \cite[Section 5.1]{ginzburg2009action}.
\end{remark}

Hence, we have the direct sum decomposition
\begin{equation}\label{cfdecomp}
\cfn^{(a,b)}_*(\eta_F) = \cfn^{(a,b)}_*(\eta_F;U) \oplus \cfn^{(a,b)}_*(\eta_F;M,U),
\end{equation}
on the level of vector spaces. Now, we fix an arbitrary $\om$-compatible almost complex structure $J_0$.

\begin{lem}[{cf.~\cite[Lemma 5.1]{ginzburg2009action}}]\label{lemepsilon1}
There exists a constant $\epsilon(U,W,M)>0$, independent of $F$ and $\eta$, such that whenever $b-a < \epsilon(U,W,M)$, the decomposition \eqref{cfdecomp} is a direct sum of complexes for any regular almost complex structure $J$ sufficiently close to $J_0$ and compatible with $\om$. Thus, we have in the obvious notation, the decomposition
\begin{equation}\label{hfndecomp}
\hfn^{(a,b)}_*(\eta_F) = \hfn^{(a,b)}_*(\eta_F;U) \oplus \hfn^{(a,b)}_*(\eta_F; M, U).
\end{equation}
\end{lem}

The proof of Lemma \ref{lemepsilon1} follows word-for-word the proof in \cite[Lemma 5.1]{ginzburg2009action} with a small adaptation when considering generators of $\cfn^{(a,b)}_*(\eta_F;M,U)$; recall Remark~\ref{rmk:decomp_MU}.

\begin{proof}
It is sufficient to check that if $\overline{x}$ and $\overline{y}$ are two capped $1$-periodic orbits of $\eta_F$ with action in $(a,b)$ connected by a Floer trajectory such that one of these capped orbits, say $x$, has capping in $U$, then the orbit $y$ contained in $U$ and its capping is equivalent to a capping contained in $U$, provided that $b-a< \epsilon(U,W,M)$. This follows from a standard argument based on the Gromov compactness theorem and the constant $\epsilon(U,W,M)$ is sometimes called the \emph{crossing energy} required to cross the \emph{shell} $\overline{W}\setminus U$.
\end{proof}

\begin{remark}
We note that if the action intervals are shorter than $\epsilon(U,W,M)$, then the direct decompositions in \eqref{cfdecomp} and \eqref{hfndecomp} are preserved by the long exact sequence maps in \eqref{leshfn}. In addition, these decompositions are also preserved by monotone decreasing homotopy maps when the Hamiltonians are constant on $M\setminus U$; recall Section~\ref{sec:homotopy}. Moreover, the constant $\epsilon(U,W,M)$ depends on $J_0$ and is bounded away from zero when $J_0$ varies within a compact set.
\end{remark}

In the case where $\mathcal{P}(\eta_F)$ has degenerate (capped) periodic orbits with action in $(a,b)$, we can approximate $F$ by a $C^2$-close Hamiltonian $\widetilde{F}$ such that all (capped) $1$-periodic orbits in $\mathcal{P}(\eta + d\widetilde{F})$ with action in $(a,b)$ are nondegenerate and set, similarly to \cite[p.~2784]{ginzburg2009action},
$$\hfn_*^{(a,b)}(\eta_F;U) = \hfn^{(a,b)}_*(\eta + d\widetilde{F};U)$$
and
$$\hfn^{(a,b)}_*(\eta_F; M, U)=\hfn^{(a,b)}_*(\eta+d\widetilde{F}; M, U).$$

\begin{prop}[{\cite[Proposition 5.2]{ginzburg2009action}}]
There is a constant $\epsilon(U,W,M)>0$, independent of $F$ and $\eta$, such that the direct sum decomposition \eqref{hfndecomp} holds and is preserved by the long exact sequence and decreasing monotone homotopy maps, as long as all action intervals are shorter
than $\epsilon(U,W,M)$ and the Hamiltonians are constant on $M\setminus U$.
\end{prop}

Moreover, the proof of Lemma \ref{lemepsilon1} yields the following conclusion.

\begin{prop}[{\cite[Proposition 5.4]{ginzburg2009action}}]
The homology $\hfn_*^{(a,b)}(\eta_F;U)$ and the monotone homotopy maps are independent of the ambient manifold $M$ when $b-a$ is sufficiently small.
\end{prop}

Furthermore, since $\eta$ vanishes on $W \supset U$, the direct summand of the filtered Floer--Novikov homology $\hfn_*^{(a,b)}(\eta_F=\eta+dF;U)$ agrees with the direct summand in the Hamiltonian setting $\hf_*^{(a,b)}(F;U)$ defined in \cite[Section 5.1]{ginzburg2009action}.

\begin{lem}\label{hfn=hfinU}
For any $a,b\notin \mathcal{S}(\eta_F)$ with $a<b$, we have $\hfn_*^{(a,b)}(\eta_F;U) = \hf_*^{(a,b)}(F;U)$,  Moreover, if $M$ is symplectically aspherical, i.e., $\om|_{\pi_2(M)}=0$, then $\hf_*^{(a,b)}(F;U) = \hf_*^{(a,b)}(F)$.
\begin{proof}
By Remark \ref{rmk=ham}, the fact that $\eta_F = \eta + dF = dF$ on $U \subset W$ yields that the action functional $\mathcal{A}_{\eta_F}$ agrees with the Hamiltonian action functional and we have the same chain complex in both constructions. When $M$ is symplectically aspherical, the Hamiltonian $F$ is a function such that $F|_{M\setminus U} = C$ with $C \notin (a,b)$ and, hence, $\cf_*^{(a,b)}(F;U) = \cf_*^{(a,b)}(F)$ from the definition. 
\end{proof}
\end{lem}

\subsection{Symplectic Degenerate Maximum}\label{subsection:SDM}
In this section, we recall the definition of a symplectic degenerate maximum and a geometrical characterization of such an orbit; see~\cite[Section~1.5 and Section~5.2]{ginzburg2009action}.

\begin{definition}\label{defn:SDM}
	An isolated capped $1$-periodic orbit $\overline{x}$, with  $x\in \mathcal{P}(\theta_t)$, is called a \emph{symplectically degenerate maximum (SDM) of $\psi_t$} if $\Delta(\psi_t,x) = 0$ and $\hf^{\loc}_n(\psi_t,\overline{x}) \neq 0$. On the other hand, if $\Delta(\psi_t,x) = 0$ and $\hf^{\loc}_{-n}(\psi_t,\overline{x}) \neq 0$, $\overline{x}$ is a \emph{symplectically degenerate minimum (SDm) of $\psi_t$}.
\end{definition}

\begin{remark}
	Here, in the definition of a SDM, we use the conventions of the degree as in e.g. \cite{ginzburg2010conley} or \cite{GG15_ccbeyond}, in contrast with the conventions in \cite{ginzburg2009action}.
\end{remark}

For a capped closed curve $(x,u)$ in $M$ and a loop of Hamiltonian diffeomorphisms $\zeta^t$ with $\zeta^0=id$, we can define a closed curve $\Phi_\zeta(x,u)$ in $M$ as follows. Let $x(0)=p$. First, we visualize the capping $u\colon D^2 \to M$ as a map $u\colon [0,1] \times S^1 \to M$ such that $u([0,1] \times \{0\}) = u(\{0\}\times S^1) = \{p\}$ and $u(\{1\} \times S^1) = x$. Then $x_s = u(s,\cdot)$, for $s \in [0,1]$, is a family of closed curves with $x_0 \equiv p$, $x_1 = x$ and $x_s(0) = p$ for all $s$. Hence, $\Phi_\zeta(x,u)(t) := \zeta^t(x(t))$ is again a closed curve with capping $(t,s) \mapsto \zeta^t_s(x_s(t))$. 

There is a geometrical characterization of a symplectically degenerate maximum which we recall now. Recall that given two periodic Hamiltonians $K_t$ and $H_t$, their \emph{composition} is the Hamiltonian
$$(K\#H)_t = K_t + H_t \circ (\varphi^K_t)^{-1},$$
generating the flow $\varphi_t^K \circ \varphi_t^H$.


\begin{lem}[cf. {\cite[Proposition 5.8]{ginzburg2009action}}]\label{fixsdm}
	Let $\{\psi_t\} \subset \Symp_0(M,\om)$ be a symplectic isotopy such that $\psi_0 = id$. Suppose that $x \in \mathcal{P}(\theta_t)$ is a $1$-periodic orbit which is a symplectically degenerate maximum of $\psi_t$. Let $p=x(0)$ and $W$ be a neighborhood of $p$ where $\theta$ is exact, say $\theta = df$, and hence, on $W$, $\theta_t = d\hat{H}_t$, where $\hat{H}_t=f+H_t$ . Then, there exists a sequence of contractible loops $\zeta_i$ of Hamiltonian diffeomorphisms generated by Hamiltonians $G^i$ such that $\Phi_{\zeta_i}(p)=x$ and the following conditions are satisfied:
	\begin{enumerate}
		\item The point $p$ is a strict local maximum for $\hat{K}^i_t$, for all $t\in S^1$ and all $i$, where $\hat{K}^i$ is given by $\hat{H} = G^i \# \hat{K}^i$.
		\item There exist symplectic bases $\Theta^i$ in $T_pM$ such that
		$$\Vert d^2(\hat{K}_t^i)_p\Vert_{\Theta_i} \to 0 \quad \text{uniformly in} \ t\in S^1.$$
		Here $\Vert \cdot \Vert_{\Theta_i}$ denote the norm with respect to the Euclidean inner product for which $\Theta_i$ is an orthonormal basis.
		\item The linearization of the loop $(\zeta_i^{-1} \circ \zeta_j)$ at $p$ is the identity map for all $i$ and $j$ and, moreover, the loop $(\zeta_i^t)^{-1} \circ \zeta_j^t$ is contractible to $id$ in the class of loops fixing $p$ and having the identity linearization at $p$.
	\end{enumerate} 
\end{lem}

\subsection{SDM and Infinitely Many Periodic Orbits}\label{subsection:proofinfty}
%
The proof of Theorem~\ref{thm:SDMHFN} follows closely the arguments in the proof of \cite[Theorem 1.17]{ginzburg2009action} with minor adjustments to our setting.

\begin{proof}[Proof of Theorem \ref{thm:SDMHFN}]
The (local) loop $\eta_{\text{loc}}^t:=(\varphi^t_{\hat{H}})^{-1}\circ \varphi^t_{\hat{K}}$, where $\hat{H}$ and $\hat{K}:=\hat{K}^1$ are as in Lemma~\ref{fixsdm}, extends to a loop $\eta^t$ of global Hamiltonian diffeomorphisms on $M$ generated by a periodic Hamiltonian say $G_t$; see \cite[Lemma~2.8]{ginzburg2010conley}. Note that $\hat{K}_t=\hat{H}_t+G_t\circ \varphi^t_{\hat{H}}$ (up to adding a constant). The symplectic isotopy $\psi^{\prime}_t$ defined by $\psi_t\circ\eta^t$ is based at the identity with endpoint $\psi^\prime_1=\psi_1$ and satisfies $\text{Flux}(\{\psi^{\prime}_t\})=\text{Flux}(\{\psi_t\})$. Hence, it is sufficient to prove Theorem~\ref{thm:SDMHFN} for the $1$-form $\theta^\prime_t = \theta^\prime +dK_t$, where $\theta^\prime: = \theta - d\bar{f}$ and $\bar{f}\colon M \to \R$ is a smooth extension of the primitive $f$ of $\theta$ in $W$, as in Lemma~\ref{fixsdm}, and $K_t=H_t+G_t\circ \varphi^t_{\hat{H}}+\bar{f}$, which coincides with $\hat{K}_t$ on $W$. In particular, the $1$-form $\theta^\prime$ vanishes on $W$. In this case, $p$ is a strict local maximum of $K_t$, for all $t \in S^1$.


Let $H_+$ and $H_-$ be the functions such that $H_+ \geq K \geq H_-$ defined in \cite[\S 5.3.1]{ginzburg2009action}. We recall some properties of these two functions. 
The function $H^+$ is constant and equal to its maximum outside the neighborhood $U$ of $p$. Hence, its periodic orbits outside $U$ are trivial. Moreover, the periodic orbits in $U$ are either trivial or fill in spheres of certain radii. There is a nondegenerate $C^2$-small perturbation of $H_+$ such that each of these spheres splits into $2n$ nondegenerate orbits. For sufficiently large $T$, this perturbation has only one $T$-periodic orbit of index $n$ and action in $(Tc-\delta,Tc+\varepsilon)$, which is the constant orbit $p$, and only one orbit $x$ of index $n+1$ with action in the latter interval. There are no orbits with index $n-1$ and action within this range. For more details about $H_+$, see \cite[\S 7.3]{ginzburg2010conley}. The function $H_-$ is defined as the composition $H_- = G^0 \# F$, where $G^0$ is given in Lemma \ref{fixsdm} and $F$ is an auxiliary bump function defined on a small neighborhood of $p$ such that $F\leq K$, $F(p) = c = K(p)$ is the absolute maximum of $F$ and $F^s := G^s \# F \leq H_+$ for all $s$. In particular, the homotopy $F^s$ is  isospectral beginning in $F^0=H_-$ and ending with $F^1 = F$. Lemma \ref{fixsdm} is also used in the construction of the function $F$; see \cite[\S 7.4]{ginzburg2010conley}.

As explained in \cite[p. 2790]{ginzburg2010conley}, for $0<\varepsilon<\epsilon(U,W,M)$, there exists $k_\varepsilon$ such that
\begin{equation}\label{hfh+F}
\Q \cong \hf_{n+1}^{(kc+\delta_k,kc+\varepsilon)}(H_+^{(k)};U) \xrightarrow[\Psi]{\sim} \hf_{n+1}^{(kc+\delta_k,kc+\varepsilon)}(F^{(k)};U) \cong \Q,
\end{equation}
for all $k> k_\varepsilon$ and some $\delta_k \in (0,\varepsilon)$. Here $U \Subset W$ are Darboux balls centered at $p = x(0)$. We emphasize that, by construction, the homotopy $(F^s)^{(k)}$ from $H_-^{(k)}$ to $F^{(k)}$ is also isospectral and $(F^s)^{(k)} \leq  H_+^{(k)}$ for all $s$. In particular, the isomorphism map $\Psi$ in the middle of \eqref{hfh+F} is the homotopy map induced by the monotone linear homotopy.  Combining \eqref{hfh+F} with Lemma \ref{hfn=hfinU}, we obtain 
$$\Q \cong \hfn_{n+1}^{(kc+\delta_k,kc+\varepsilon)}(k(\theta^\prime + dH_+);U) \xrightarrow{\sim} \hfn_{n+1}^{(kc+\delta_k,kc+\varepsilon)}(k(\theta^\prime + dF);U).$$
Since $(kc+\varepsilon) - (kc+\delta_k) = \varepsilon - \delta_k < \epsilon(U,W,M)$, Lemma \ref{lemepsilon1} guarantees that these groups enter the filtered Floer--Novikov homology of $k(\theta^\prime + dH_{+})$ and $k(\theta^\prime + dF)$ as direct summands. In particular, we have the nonzero map
$$\hfn_{n+1}^{(kc+\delta_k,kc+\varepsilon)}(k(\theta^\prime + dH_+)) \xrightarrow{\neq 0} \hfn_{n+1}^{(kc+\delta_k,kc+\varepsilon)}(k(\theta^\prime + dF)).$$
We are in an analogous situation to that in \eqref{comdiag}, and hence, we have the following commutative diagram
\begin{equation}
\begin{tikzcd}
\hfn_{n+1}^{(kc+\delta_k,kc+\varepsilon)}(k(\theta^\prime + dH_+)) \arrow[rd,"\neq 0"] \arrow[d,"\Psi_{H_+,H_-}"] \\ \hfn_{n+1}^{(kc+\delta_k,kc+\varepsilon)}(k(\theta^\prime + dH_-)) \arrow[r,"\cong"] & \hfn_{n+1}^{(kc+\delta_k,kc+\varepsilon)}(k(\theta^\prime + dF)),
\end{tikzcd}
\end{equation}
where the horizontal isomorphism is induced by the isospectral homotopy $\theta^\prime_{s,t} = k(\theta^\prime + dF^s)$ from $k(\theta^\prime+dH_-)$ to $k(\theta^\prime + dF)$ and the other two maps correspond to homotopy maps induced by the monotone linear homotopies. Therefore $\Psi_{H_+,H_-}$ is a nonzero map. This proves the theorem because the homotopy map
$$\Psi_{H_+,H_-} \colon \hfn_{n+1}^{(kc+\delta_k,kc+\varepsilon)}(k(\theta^\prime + dH_+)) \to \hfn_{n+1}^{(kc+\delta_k,kc+\varepsilon)}(k(\theta^\prime + dH_-))$$
factors through the group $\hfn_{n+1}^{(kc+\delta_k,kc+\varepsilon)}(k\theta^\prime_t)$, where $\theta^\prime_t = \theta^\prime + dK_t$, provided that $H_+ \geq K \geq H_-$.
\end{proof}

We finish this section by recalling why Theorem~\ref{thm:SDMHFN} implies the existence of infinitely many periodic points whenever a SDM exists, namely we prove Theorem~\ref{thm:infty}.

\begin{proof}[Proof of Theorem \ref{thm:infty}]
Since $\psi$ has finitely many fixed points, they are all isolated. Let $p$ be a prime number larger than the $k_{\epsilon}$ obtained in Theorem \ref{thm:SDMHFN}. We may suppose that $\psi^p$ has only isolated fixed points, otherwise there would be infinitely many simple $p$-periodic orbits. Let $\overline{x}$ be a SDM and $\mathcal{A}_{\psi_t}(\overline{x})=c$. We can write
$$\hfn_*^{(pc-\varepsilon,pc+\varepsilon)}(p\theta_t) = \bigoplus \hf^{\loc}_*(\psi^p_t,\overline{y}),$$
where the direct sum is over capped $p$-periodic orbits in $\cP(p\theta_t)$ with action in the interval $(pc-\varepsilon,pc+\varepsilon)$ for a sufficiently small $\varepsilon$. Therefore, by Theorem \ref{thm:SDMHFN} and property \eqref{supphfloc}, there must exist a $p$-periodic orbit $y$ in $\mathcal{P}(p\theta_t)$ such that $\Delta(\psi^p_t,y) \in [1 , 2n+1]$. Since $p$ is prime $y$ must be a simple $p$-periodic orbit. This holds because, for a fixed point $z$, either we have $\Delta(\psi^k_t,z)=0$ or $\lim_{k\to \infty} \vert \Delta(\psi^k_t,z)\vert = \infty$. Therefore, $\psi$ has a simple $p$-periodic point for each sufficiently large prime $p$, which is contractible with respect to $\psi_t^p$.
\end{proof}

\begin{remark}
As in \cite[Theorem 1.18]{ginzburg2009action}, under the assumption that $\psi^{k_0}$ has finitely many $k_0$-periodic points and a SDM, one can prove that $\psi$ has infinitely many periodic points for more general symplectic manifolds, namely, closed, rational and {semi-positive}. The proof in the symplectically Calabi--Yau case is much simpler, and since we shall focus on the case where $M = \Sigma_g$ is a closed surface with genus $g \geq 1$, we only need that case.
\end{remark}

\section{Proofs}

\subsection{Proof of Theorem \ref{thm:mean index-zero}}
\label{sec:proof-mean index-zero}
Since composing with a loop of Hamiltonian diffeomorphism does not alter the dimension of the local Floer homology of a fixed point nor the homotopy class of its orbit, it is possible to assume that all contractible orbits of $\psi_t$ are constant. Let $\rS$ be the finite set of fixed points $y$ contractible with respect to $\psi_t$ such that $\lvert\Delta(\psi_t,y)\rvert\neq\Delta(\psi_t,x)$. 
	
	Fix a positive integer  $k_0$ such that
	\begin{equation}
		\label{eq:mean index-zero_tau}
		k\lvert \Delta(\psi_t,x) \pm \Delta(\psi_t,y)\rvert > 4\quad\forall k\geq k_0,\forall y\in\rS,
	\end{equation}
	$k_0\lvert\Delta(\psi_t,x)\rvert>2$, and such that for every prime $p>k_0$, $\psi^p$ is an admissible iteration as in \cite[Section~4.2]{ginzburg2009action} with respect to $x$. 
	
Fix $p>k_0$ and put $\phi_t=\psi_t^p$. Suppose to reach a contradiction that $\fix_0(\phi_t)=\fix_0(\psi_t)$. 

First, we argue when $x$ is nondegenerate, which is simpler and includes the case where $\Delta(\psi_t,x)$ is an odd integer. The fact that $\psi^p$ is an admissible iteration implies that $x^p$ is nondegenerate. It follows from \eqref{eq:mean index-zero_tau} that $x^p$ is non-trivial in $\hfn(\phi_t)$. Indeed, there are no Floer trajectories connecting $x^p$ to $y^p$ for all $y\in S$ since this would imply their mean indices differ by at most $2$. If $y\notin S$ and $\Delta(\psi_t,x)=\Delta(\psi_t,y)$, then $\cz(\phi_t,x^p)=\cz(\phi_t,y^p)$ so there are no Floer trajectories connecting $x^p$ and $y^p$. Finally, if $y\notin S$ and $\Delta(\psi_t,x)=-\Delta(\psi_t,y)$, then $p|\Delta(\psi_t,x)|>2$ implies that $p|\Delta(\psi_t,x)-\Delta(\psi_t,y)|>4$ and, hence, that $|\cz(\phi_t,x^p)-\cz(\phi_t,y^p)|>2$, which again implies there is no Floer trajectory between them. In particular, $x^p$ is closed and non-exact. Finally, $p|\Delta(\psi_t,x)|>2$ also yields $|\cz(\phi_t,x^p)|>1$, which is in contradiction to the fact that $\hfn(\phi_t)$ is supported in the interval $[-1,1]$.
	
It remains to treat the case $x$ is degenerate. In particular, we may suppose that $\Delta(\psi_t,x)$ is a non-zero even number. Following the construction in Section \ref{sec:blow-up_and_gluing}, blow-up $\Sigma$ at each $z\in E$, where 
	\begin{equation*}
		E=\{y\in\fix_0(\phi)\,|\,y\notin S, y\neq x\},
	\end{equation*}
	
	to obtain a surface with boundary $\Sigma_+$ and consider the path of symplectic diffeomorphisms $\phi^+_{t}$ induced by $\phi_t$ whose dynamics on each boundary component of $\partial\Sigma^+$ (corresponding to each $z$) is given by $D\phi_t(z)/||D\phi_t(z)||$. We similarly construct $\Sigma_{-}$ by blowing up $\Sigma$ at $\tau(E)$, where $\tau$ is the anti-symplectic involution given by \emph{reflection}, and define $\phi_t^{-}$. We glue $\Sigma_+$ and  $\Sigma_{-}$ along their common boundaries to obtain a surface $\Sigma'$ with the induced $C^{1}$-symplectic isotopy $\phi_{t}'$.
	
	Since $\Delta(\phi_t,x)$ is a non-zero even integer, the only new orbits that can appear are non-contractible periodic orbits of $\phi_t'$ whose image is the corresponding boundary component of $\partial\Sigma_+=\partial\Sigma_-\subset\Sigma'$. Consequently, the only contractible fixed points of $\phi_{t}'$ are the two copies of the contractible fixed points of $\phi_t$ which were not blown-up. Moreover, there is a tubular neighborhood $\cN$ of $\partial\Sigma_+=\partial\Sigma_-\subset\Sigma'$ with no contractible fixed points. Therefore, we may approximate $\phi_t'$ by a smooth symplectic isotopy $\phi_t''$ that coincides with $\phi_t'$ on the complement of $\cN$ and has no contractible fixed points in $\cN$. 
	
	We will adopt the convention that the fixed point $y_{+}$ (resp. $y_-$) of $\phi_{t}''$ is the copy of the fixed point $y$ of $\phi_t$ that is in $\Sigma_{+}$ (resp. $\Sigma_{-}$) and
	\begin{equation*}
		\Delta(\phi_t,y)=\Delta(\phi''_{t},y_{+})=-\Delta(\phi_{t}'',y_{-}).
	\end{equation*}
	Choose a nondegenerate perturbation $\xi_t$ of ${\phi}_{t}''$ such that
	\begin{equation}
		\label{eq:rational_pert}
		\lvert \Delta(\xi_{t},y'_{\pm}) - \Delta(\phi_t'',y_{\pm})\rvert<\delta  \quad\forall y'_{\pm}\in\cO(\xi_t,y_\pm), \forall y_{\pm}\in\fix_0(\phi_{t}'')
	\end{equation}
	for $\delta<1/2$. Here, $\cO(\xi_t,y_\pm)$ denotes the $1$-periodic orbits $y_\pm$ splits into under the perturbation $\xi_t$ of $\phi_t''$. Combining Equation (\ref{eq:rational_pert}) with the fact that $\hf^{\loc}(\phi_t'', y_{\pm})$ is supported in the interval $[\Delta(\phi_t'',y_\pm)-1,\Delta(\phi_t'',y_\pm)+1]$ implies that, for all $y\in\rS$, there are no Floer trajectories connecting orbits from $\cO(\xi_t,x_\pm)$ to orbits from $\cO(\xi_t,y_\pm)$ and vice-versa. Indeed, there are two cases to be considered: both orbits are in the same component $\Sigma_{\pm}$; the orbits belong to different components. 
	
	In the first case, let us suppose without loss of generality that both belong to $\Sigma_{+}$. If there is a Floer trajectory connecting $x'_+\in\cO(\xi_t,x_+)$ and $y'_+\in\cO(\xi_t,y_+)$, then \[\lvert\cz(\xi_t,x'_+)-\cz(\xi_t,y'_+)\rvert=1,\] which implies
	\begin{align*}
		p\lvert \Delta(\psi_t,x) - \Delta(\psi_t,y)\rvert&=\lvert \Delta(\phi_t'',x_+) - \Delta(\phi_t'',y_+)\rvert\\
		&\leq\lvert \Delta(\xi_t,x'_+) - \Delta(\xi_t,y'_+)\rvert + 2\delta\\
		&\leq\lvert \cz(\xi_t,x'_+) - \cz(\xi_t,y'_+)\rvert+2+2\delta<4;
	\end{align*}
	this contradicts Equation (\ref{eq:mean index-zero_tau}). 
	
	A similar argument shows that in the second case we would have \[p\lvert \Delta(\psi_t,x) + \Delta(\psi_t,y)\rvert<4,\] which also contradicts Equation (\ref{eq:mean index-zero_tau}). 
	
	Furthermore, there are no Floer trajectories connecting orbits from $\cO(\xi_t,x_+)$ to orbits from $\cO(\xi_t,x_-)$ and vice-versa. Indeed, since $k_0\lvert\Delta(\psi_t,x)\rvert >2$, we have
	\begin{align*}
		\lvert \Delta(\xi_t,x'_+) - \Delta(\xi_t,x'_-)\rvert&\geq\lvert \Delta(\phi_t'',x_+) - \Delta(\phi_t'',x_-)\rvert - 2\delta\\
		&=2p\lvert \Delta(\psi_t,x)\rvert-2\delta > 3.
	\end{align*}
	Hence, $\lvert\cz(\xi_t,x'_+)-\cz(\xi_t,x'_-)\rvert>1$, which implies the orbits cannot be connected by a Floer trajectory.
	
	It follows that cycles representing a non-trivial class of $\hf^{\loc}(\phi_t'',x_\pm)$, which exist since $p$ is admissible, remain non-trivial in $\hfn(\xi_t)$. Moreover, $k_0\lvert\Delta(\psi_t,x)\rvert >2$ implies that their grading is non-zero. This contradicts the fact that $\hfn(\xi_t)$ is supported in the interval $[-1,1]$. Therefore, there must be a simple $p$-periodic orbit contractible with respect to $\psi_t^p$; this is what we wanted to show.

\subsection{Proof of Theorem \ref{thm:dichotomy}}
\label{sec:proof-dichotomy}

Suppose that every contractible fixed point $x$ of $\psi_t$ with non-trivial local Floer homology satisfies $\Delta(\psi_t,x)=0$. Hence,
	\begin{equation*}\label{eq:SDM1_1}
		\supp(\hf^{\loc}(\psi_t,x))=\{-1,0,1\}\quad \text{for all}\quad x\in\fix_0(\psi_t).
	\end{equation*}
In particular, the same is true for the canonical complex $(\cfn(\psi_t,\Lambda^0),d)$ defined in Section \ref{sec:canonical}, i.e.,
	\begin{equation*}\label{eq:SDM1_2}
		\supp(\cfn(\psi_t,\Lambda^0))=\{-1,0,1\}
	\end{equation*}
On the other hand, Part $(3)$ of Proposition \ref{prop:canonical} combined with Inequality \eqref{ineq_floc}, imply that differential $d$ is non-trivial. Therefore, by Part $(2)$ of Proposition \ref{prop:canonical}, we must have two consecutive integers $k,k+1\in\{-1,0,1\}$ for which $\cfn(\psi_t,\Lambda^0)$ is non-trivial. If $k=0$, then there exists at least one symplectically degenerate maximum. Otherwise, there exists at leat one symplectically degenerate minimum. This finishes the proof.
	
\section{Surface symplectic diffeomorphisms with finitely many periodic points}\label{sec:examples}
In Section~\ref{section:construction_1fx_pt}, we prove Theorem~\ref{thm:1_fx_pt}; we present a construction of a symplectic flow with exactly one fixed point and no other periodic orbits. In Section~\ref{section:partition}, firstly, we describe a construction of a symplectic flow with exactly $2g-2$ fixed points and no other periodic orbits, giving an alternative (with respect to \cite[Section~3A]{batoreo2018periodic}) proof of Theorem~\ref{thm:2g-2_fx_pts}, and, secondly, comment on Remark~\ref{remark:partition} and present some examples of symplectic flows with finitely many fixed points which satisfy the \emph{partition condition} in Remark~\ref{remark:partition}.   

In the following sections, let $(\Sigma,\omega)$ be a closed orientable surface of genus $g\geq 2$ equipped with an area form $\omega$.

\subsection{Construction of a symplectic flow on $(\Sigma,\omega)$, $g\geq 2$ with exactly one fixed point and no other periodic orbits}\label{section:construction_1fx_pt}

\

The construction of a symplectic flow $\psi_t$ on $\Sigma$ with exactly one fixed point (as in Theorem~\ref{thm:1_fx_pt}) has three steps: in step~$1$, we consider $g$ tori with irrational linear flows, in step~$2$, we construct a surface (with $g$ \emph{circle-ends}) and define on it a Hamiltonian flow with exactly one fixed point and, in step~$3$, we glue each torus to  one of the \emph{ends} of the constructed surface in order to obtain $\Sigma$ with the desired symplectic flow $\psi_t$. A draft of the construction of $\Sigma$ is in Figure~\ref{figure:general}.

\begin{figure}[htb!]
	\includegraphics[scale=0.11,trim={0cm 0cm 0cm 0cm}]{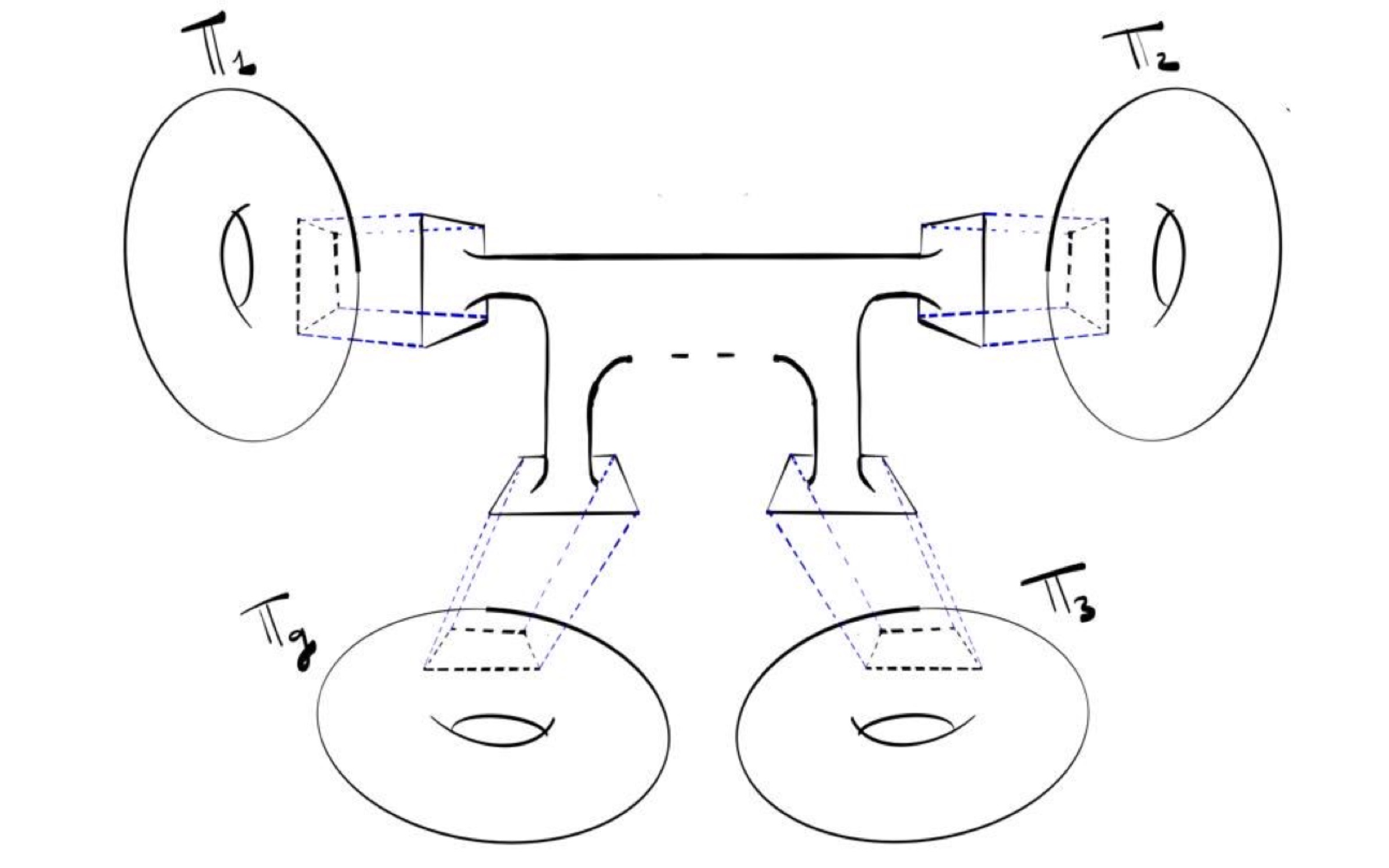}
	\caption{The construction of $\Sigma$.}\label{figure:general}
\end{figure}

\paragraph{\textbf{Step 1}}\label{paragraph:step1}
In this step, consider $g$ tori, $\mathbb{T}_1, \ldots,\mathbb{T}_g$, and the linear flow $\phi_i^t$ on each torus $\mathbb{T}_i, \,i=1,\ldots, g$:
\[
\phi_i^t(x_i,y_i)=(tu_ix_i,tv_iy_i) \quad\text{with}\quad u_i\not=0 \;\text{and}\; \frac{v_i}{u_i}\in\R\backslash\Q,\quad i=1,\ldots, g.
\]  
Here $x_i,y_i$ are the coordinates on $\mathbb{T}_i=\R^2 /\Z^2, \; i=1,\ldots,g.$

Representing each torus by a square $[0,1]\times[0,1]$, where the sides $\{0\}\times [0,1]$ and $[0,1]\times\{0\}$ are identified with $\{1\}\times [0,1]$ and $[0,1]\times\{1\}$, respectively, consider, for each $i=1,\ldots,g,$ a square $R_i$ in $\mathbb{T}_i$ with side length $\varepsilon>0$ such that two parallel sides are segments of a linear flow line of $\phi_i^t$; see Figure~\ref{figure:two_tori}. (The real number $\varepsilon$ is small enough so that each square $R_i$ is inside the square $[0,1]\times [0,1]$.) 
\begin{figure}[htb!]
	\includegraphics[scale=0.12, trim={0 2cm 0 1cm}]{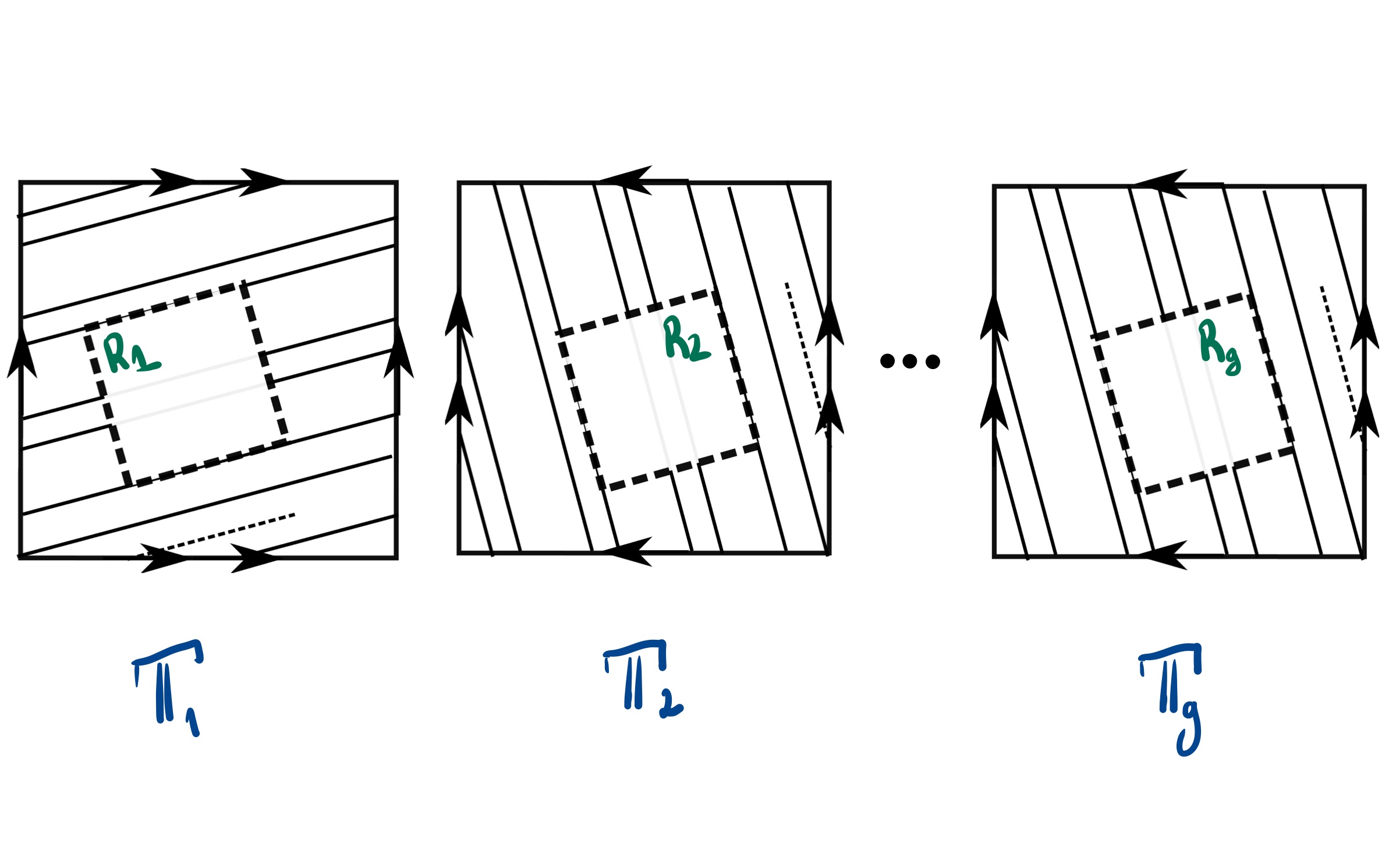}
	\caption{Tori $\mathbb{T}_1,\mathbb{T}_2,\ldots,\mathbb{T}_g$ with irrational linear flow lines and the squares $R_1,R_2,\ldots,R_g$}\label{figure:two_tori}
\end{figure}

\begin{remark}\label{remark_notsquares} In order to distinguish the boundaries of the squares from the interiors of the squares, we denote $R_i$ the boundary and by $\mathring{R_i}$ its interior, $i=1,\ldots,g$.
\end{remark}

\paragraph{\textbf{Step 2}}\label{paragraph:step2}In this step, we define a Hamiltonian flow on a \emph{connecting surface} $U$. We anticipate that the fixed point of $\psi_t$ lies on the surface $U$ where $\psi_t$ is Hamiltonian and, \emph{locally}, the Hamiltonian function is given by 
\begin{eqnarray}\label{eqn:Ham}
	H^0(x,y)=\prod_{j=1}^{h/2} \left(x-\tan(j a) y\right),	
\end{eqnarray}
where $h=4g-2$ and $a=2\pi/h$. Flow lines of $X_{H^0}$ are depicted in Figure~\ref{figure:flow_lines_H0} when $g=2$ (left) and $g=3$ (right).

\begin{figure}[htb!]
	\includegraphics[scale=0.12, trim={0 4cm 0 0cm}]{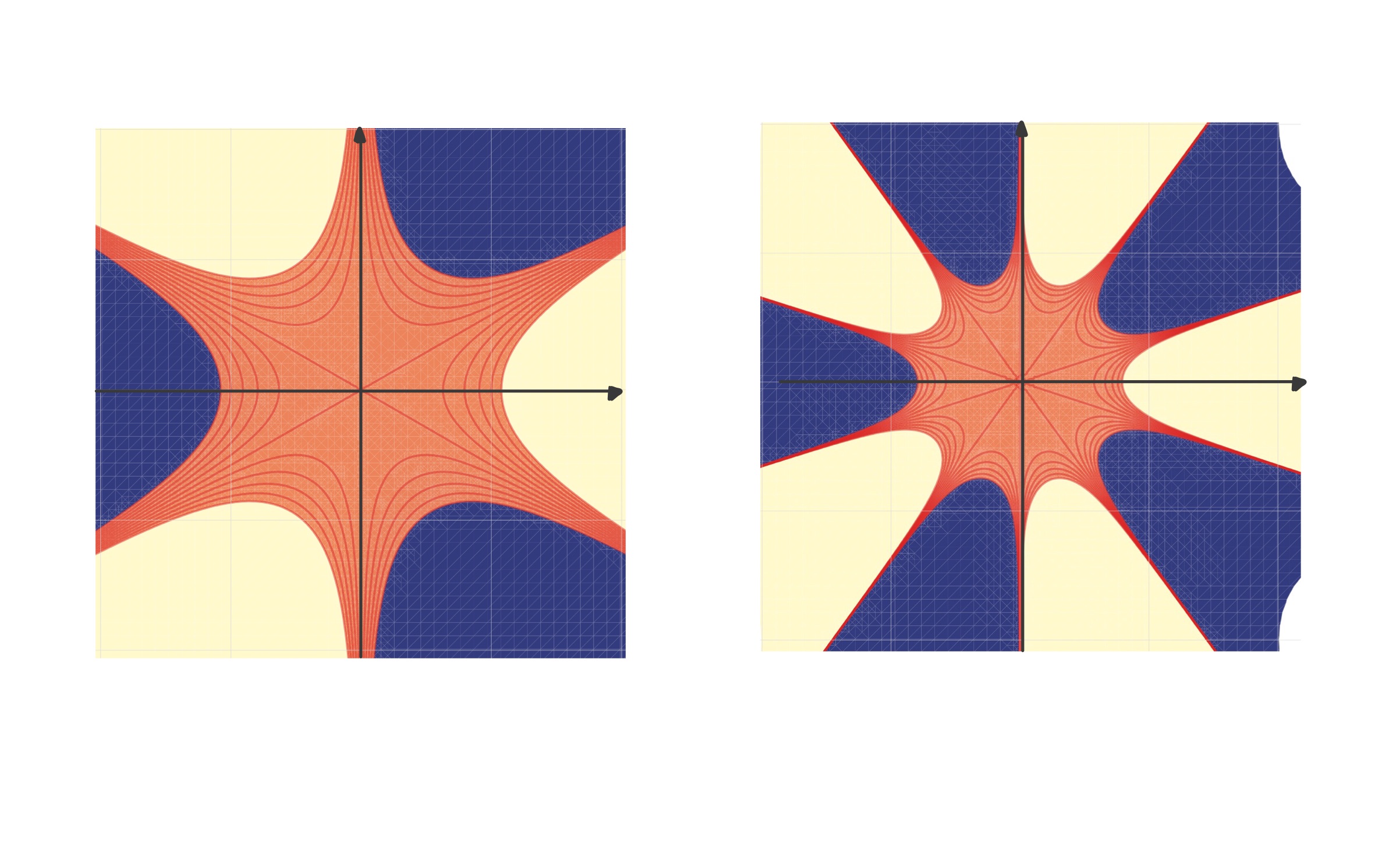}
	\caption{Flow lines of $H^0$, with $g=2$ (left) and $g=3$ (right)}\label{figure:flow_lines_H0}
\end{figure}

\begin{remark}\label{rmk:index_hyperbolicregion}
	The point $(0,0)$ is the unique fixed point of the Hamiltonian vector field $X_{H^0}$. The index of the fixed point is $2-2g$ and it is given by $1-h/2$ where $h$ is the number of hyperbolic regions of the fixed point; see, e.g., \cite[Theorem~9.1 (page 166) and Exercise~9.1 (page 173)]{hartman1982ordinary}. Hence, the number $h=4g-2$ in \eqref{eqn:Ham} is the number of hyperbolic regions of $(0,0)$. For instance, $h=6$, when $g=2$ (and, in this case, the fixed point is a, usually called, Monkey saddle), and $h=10$, when $g=3$, as one can recognize in Figure~\ref{figure:flow_lines_H0}.
\end{remark}

Firstly, we will define the Hamiltonian flow on a convenient neighborhood $D^{0}$ of $(0,0)$ in $\R^2$ and then construct the surface $U$ using the neighborhood $D^{0}$.

\subparagraph{\textbf{Neighborhood} $\mathbf{D^{0}}$} Let $h=4g-2$ and $a=2\pi/h$ and consider three neighborhoods of $(0,0)$ in $\R^2$, $D^{0}$, $D^{'}$ and $D^{''}$, with the following properties:
\begin{enumerate}[label=(\alph*)]
	\item $D^{0} \supset D^{'} \supset D^{''}$		
	\item the interior of each $D^{k}, k=0, {'}, {''},$ is diffeomorphic to an open disc in $\R^2$
	\item each $D^{k}, k=0, {'}, {''},$ is symmetric with respect to the $x$-axis
	and to the $y$-axis
	\item for $k=0, {'}, {''},$ the boundary of $D^{k}$ is a polygon with $l=4g-4$ sides which are labeled by $s^k_1, \ldots, s^k_l,$ and
	\begin{itemize}
		\item sides $s^k_{g-1}$ and $s^k_{3g-3}$ are vertical segments,
		\item sides $s^k_{2g-2}$ and $s^k_l$ are horizontal segments,
		\item each vertex of the boundary of $D^{''}$ lies in a line of the form 
		\begin{eqnarray}\label{eqn:vertices_D}
			x-\tan\left(\frac{a}{2}+(j-1)a\right)y=0,\end{eqnarray} with $j=1,\ldots,h/2=2g-1$ and $j\not=g$,
		\item each line of the form $x-\tan(ja)y=0,$ with $j=1,\ldots, h/2=2g-1$ and $j\not= g-1,g$, is the bisector of some side $s_i$, where $i=1,\ldots, l$ and $i\not= g-1, 2g-2, 3g-3,l$.
		\item Moreover, 
		\begin{itemize}
			\item[$\circ$] when $g$ is even,
			\begin{itemize}
				\item[$\cdot$] $s_i{''}\subset s_i{'} \subset s_i^{0},$ as segments in $\R^2$, when $i\in \{1,\ldots, l\}$ is odd and
				\item[$\cdot$] $s_i{''}, s_i{'}$ and $s_i^{0}$ lie in parallel lines \emph{near} each other (see Remark~\ref{rmk:choices_D}), when $i\in \{1,\ldots, l\}$ is even and,
			\end{itemize}
			\item[$\circ$] when $g$ is odd,
			\begin{itemize}
				\item[$\cdot$] $s_i{''}\subset s_i{'} \subset s_i^{0}$ as segments in $\R^2$, when $i\in \{1,\ldots, l\}$ is even and
				\item[$\cdot$] $s_i{''}, s_i{'}$ and $s_i^{0}$ lie in parallel lines \emph{near} each other (see Remark~\ref{rmk:choices_D}), when $i\in \{1,\ldots, l\}$ is odd
			\end{itemize}
		\end{itemize}
	\end{itemize} 
\end{enumerate}

The neighborhoods $D^{k},\, k=0, {'}, {''},$ are depicted in Figure~\ref{figure:nbhd_D23}, when $g=2$ (left) and when $g=3$ (right), and a sketch of the neighborhood $D^{0}$ is illustrated in Figure~\ref{figure:nbhd_D}.

\begin{remark}
	The neighborhood $D^{0}$ is also denoted by $D$ and, similarly, each side $s_i^{0}$ is denoted by $s_i$, $i=1,\ldots, l$. 
\end{remark}

\begin{figure}[htb!]
	\includegraphics[scale=0.15, trim={0cm 10cm 0cm 0cm}]{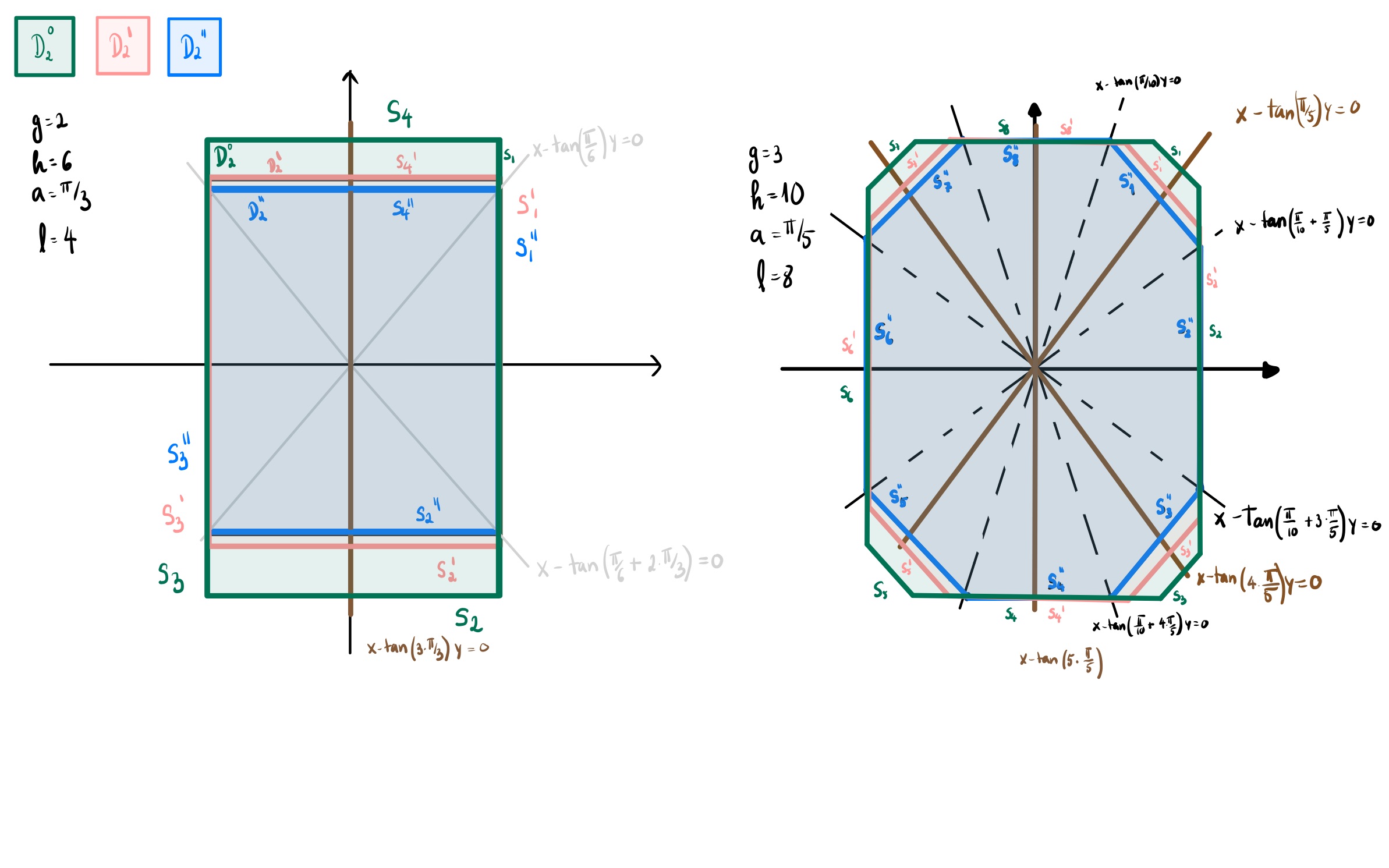}
	\caption{Neighborhoods $D^{k}, k=0, {'}, {''}$ when $g=2$ (left) and when $g=3$ (right)}\label{figure:nbhd_D23}
\end{figure}

\begin{figure}[htb!]
	\includegraphics[scale=0.12, trim={0cm 0cm 0 0cm}]{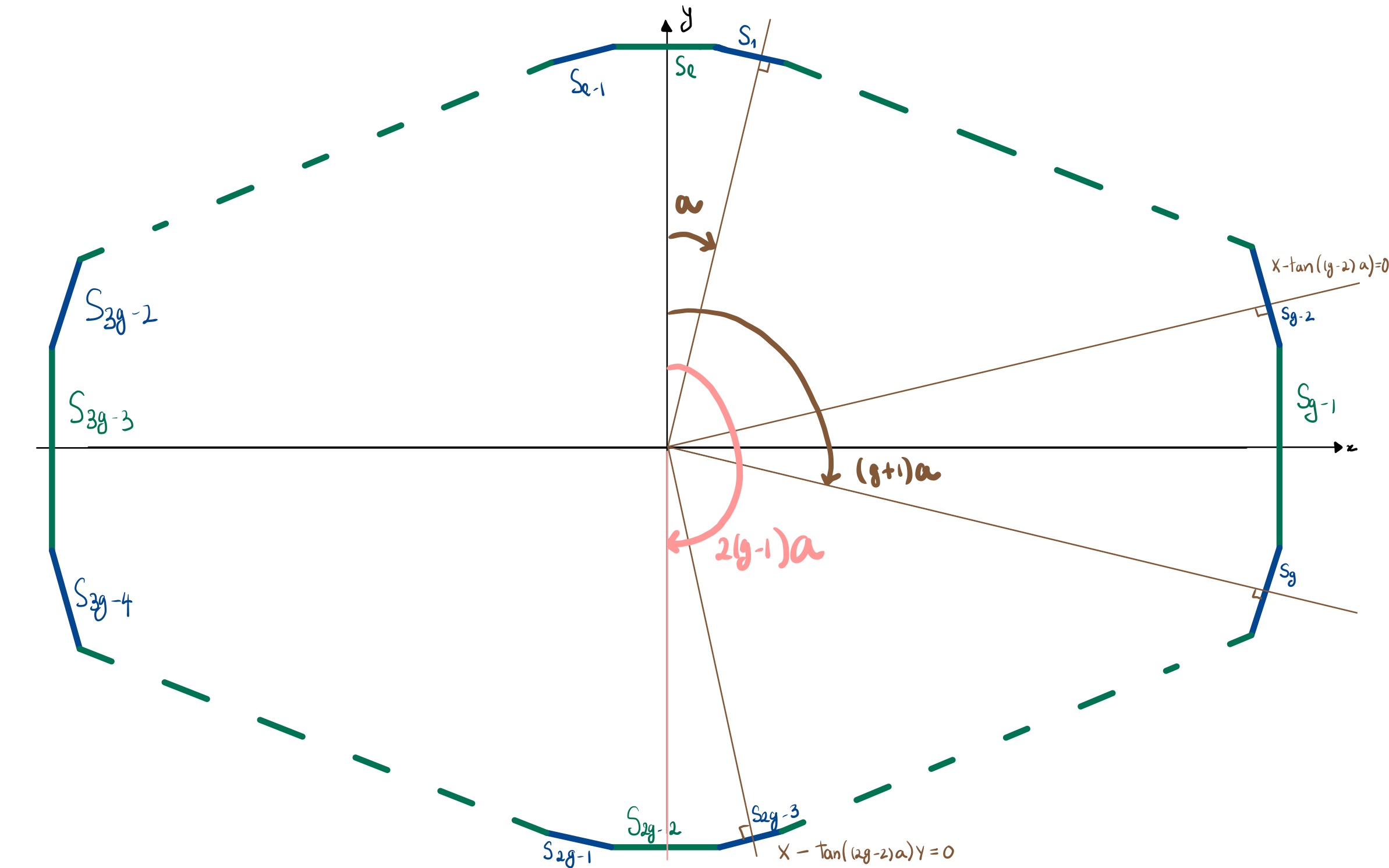}
	\caption{A neighborhood $D$, $g>2$. ($l=4g-4$.)}\label{figure:nbhd_D}
\end{figure}

\subparagraph{\textbf{Surfaces} $\mathbf{S}$ \textbf{and} $\mathbf{U}$}
In order to construct the connecting surface $U$ with a convenient Hamiltonian flow, we will, now, restrict $H^0$ to the domain $D$ and identify appropriately the sides of the boundary of $D$ in order to obtain a surface $S$ whose boundary has $g$ components where each one is a circle and, then, glue each of these circle-ends to a surface $P$ as depicted in Figure~\ref{figure:surfaceP}, where we consider a \emph{height} function. (See Remark~\ref{remark:height_symmetric}.) The surface $S$ is depicted in Figure~\ref{figure:surfaceS}, when $g=2$ (left) and $g=3$ (right).

More precisely, consider a smooth surface (with boundary) $S$ obtained from $D$ identifying the sides in the following way: 
\begin{enumerate}
	\item if $g$ is even, then, when $j$ is even, identify each side $s_j$ with side 
	\begin{itemize}
		\item $s_{2g-j-2}$, if $j\in\{1,\ldots,g-2\}$ or
		\item $s_{6g-j-6}$, if $j\in\{2g-2,\ldots,3g-3\}$; 
	\end{itemize}
	(For instance, when $g=2$, side $s_2$ is identified with $s_4$. Surface $S_2$ is a cylinder; see Figure~\ref{figure:surfaceS} (left).) 
	\item if $g$ is odd, then, when $j$ is odd, identify each side $s_j$ with side 
	\begin{itemize}
		\item $s_{2g-j-2}$, if $j\in\{1,\ldots,g-2\}$ or
		\item $s_{6g-j-6}$, if $j\in\{2g-2,\ldots,3g-3\}$. 
	\end{itemize}	
	(For instance, when $g=3$, sides $s_1$ and $s_5$ are identified with $s_3$ and $s_7$, respectively; surface $S_3$ is depicted in Figure~\ref{figure:surfaceS}~(right)).
\end{enumerate}

\begin{figure}[htb!]
	\includegraphics[scale=0.15, trim={0cm 10cm 0cm 0cm}]{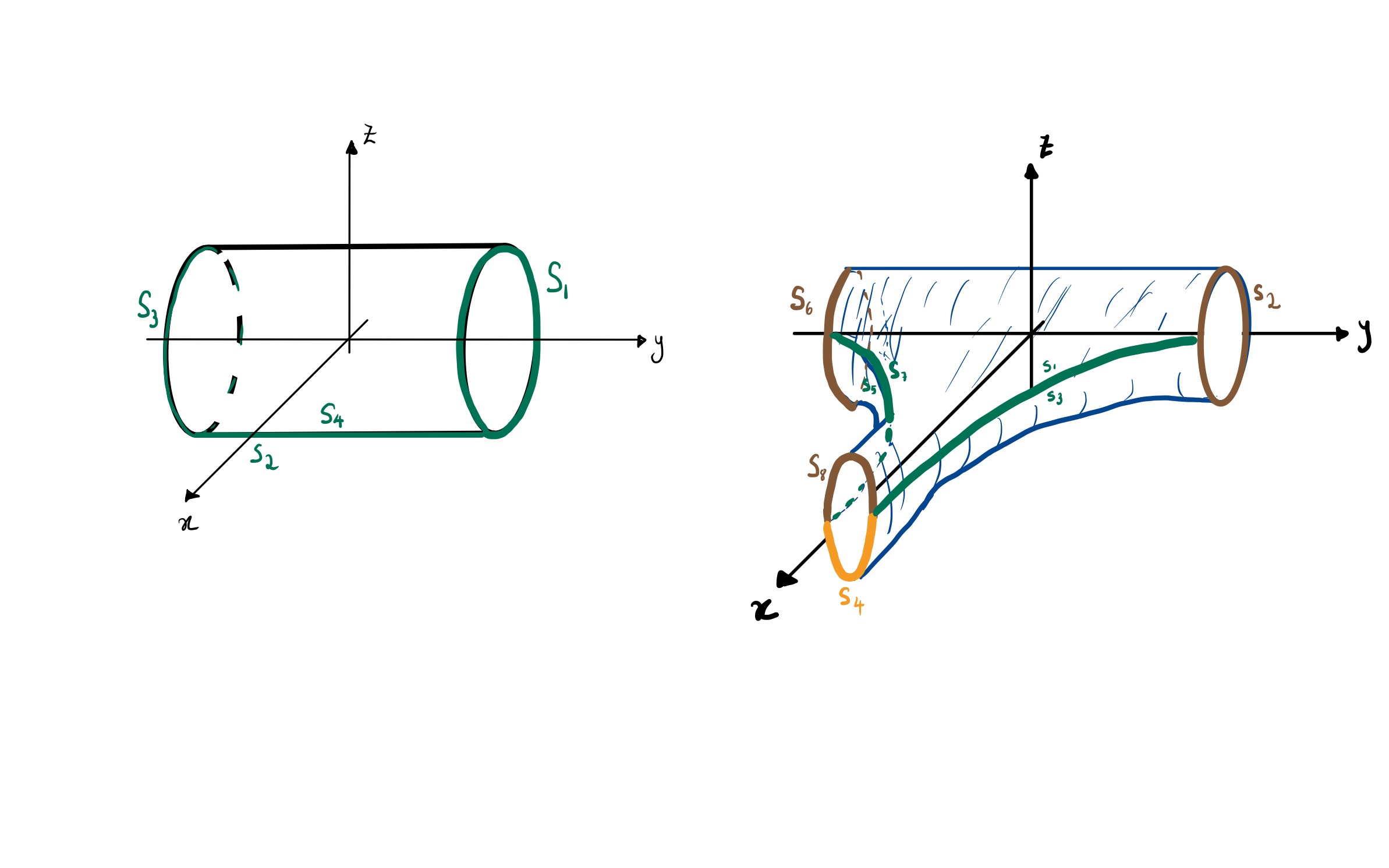}
	\caption{Surface $S$, with $g=2$ (left) and $g=3$ (right).}\label{figure:surfaceS}
\end{figure}	

Moreover, surface $S$ has $g$ boundary components, where each component is a circle $C_k, \, k=1,\ldots,g$, with radius $\varepsilon/4$ that lies in a vertical plane: circles $C_1$ and $C_2$ lie in a vertical planes of the form $y=-b$ and $y=b$, respectively, ($b>0$) and they are obtained from sides $s_{3g-3}$ and $s_{g-1}$, respectively, when identifying the two boundary components of each segment; the remaining circles $C_3,\ldots,C_g$ lie in a plane of the form $x=a$ ($a>0$) and each circle is obtained by identifying the boundary components of a side $s_j$ with, 
\begin{itemize}
	\item   when $g$ is even, 
	\begin{itemize}
		\item the boundary components of $s_{2g-j-2}$, when $j\in\{1, \ldots, g-3\}$ is odd, or   	
		\item the boundary components of $s_{6g-j-6}$, when $j\in\{2g-2, \ldots, 3g-3\}$ is odd and,
	\end{itemize}		
	\item when $g$ is odd, 
	\begin{itemize}
		\item  the boundary components of $s_{2g-j-2}$, when $j\in\{1, \ldots, g-3\}$ is even, or   	
		\item  the boundary components of $s_{6g-j-6}$, when $j\in\{2g-2, \ldots, 3g-3\}$ is even.   	
	\end{itemize}	
	
\end{itemize}  
(Recall Figure~\ref{figure:surfaceS}.)

\begin{remark}
	For instance, $C_3$ is a circle in a plane $x=a$ with an \emph{upper} part obtained from segment $s_{g+1}$ and a \emph{lower} part obtained from segment $s_{g-3}$. (If $g=3$, $s_0:=s_l.$)
\end{remark}

\begin{remark}
	Here $0<a,b<1$ and, in the pictures, numbers $a$ and $b$ are considered \emph{near} $1$, however, this is not crucial to the arguments. 	
\end{remark}

Denote by $\rho\colon D\to S\subset\R^3$ a smooth onto map that translates these gluing procedures on the boundary of $D$ in order to obtain the described $S$. (See remarks~\ref{rmk:rho_2} and~\ref{rmk:nearD_rhog}.) 

\begin{remark}\label{rmk:rho_2}
	When $g=2$, the neighborhood $D$ can be explicitly given by a rectangle of the form $[-b,b]\times [-\pi,\pi]$ and the surface $S$ by the cylinder \[C=\{(x,y,z)\in \R^3\;|\; -b\leq y\leq b, \text{ and } x^2+z^2=r^2\},\] where $r=\varepsilon/4$. Then, let $\rho\colon D\to C$ be defined from the neighborhood $D$ onto $C$ by $\rho(x, y)=(r\cos y, x, r\sin y)$. 
\end{remark}

Let $P$ be a smooth surface obtained from a homotopy between a square and a circle. (See Figure~\ref{figure:surfaceP}.) 
\begin{figure}[htb!]
	\includegraphics[scale=0.11, trim={0 8cm 0 0cm}]{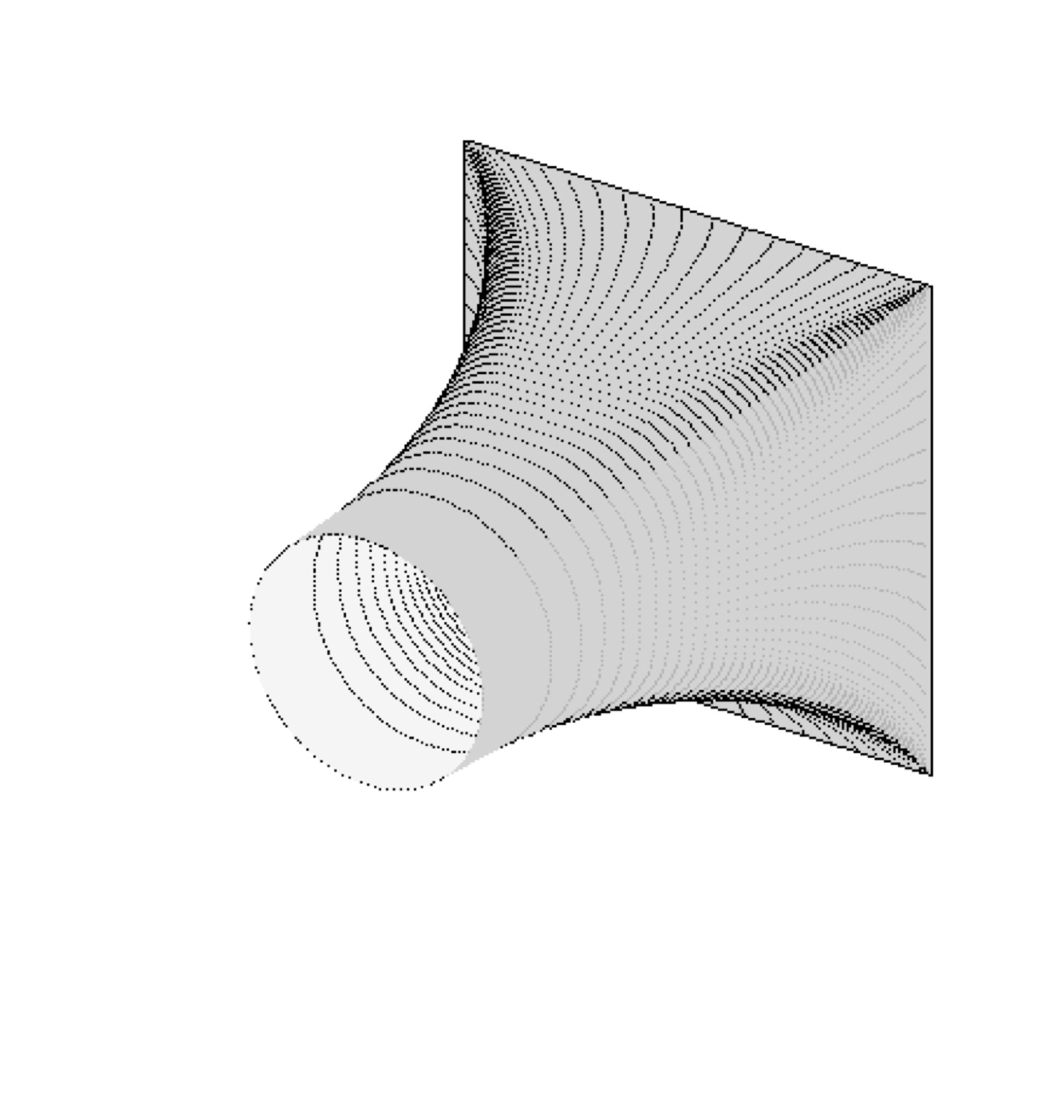}
	\caption{Surface $P$}\label{figure:surfaceP}
\end{figure}
It has two boundary components: a square with side length $\varepsilon$ and a circle with radius $\varepsilon/4$. Finally, $U$ (see Figure~\ref{figure:surfaceU}) is a smooth surface defined piecewise by a surface $S$ together with a surface $P_j \equiv P$  at each end $C_j$ (with circles identified), where $j=1,\ldots, g$. The boundary of $U$ is the disjoint union of $g$ squares $Q_1,\ldots, Q_g$, where the square $Q_1$ lies in the plane $y=-1$, the square $Q_2$ lies in the plane $y=1$ and each of the remaining squares $Q_3,\ldots, Q_g$ lies in the plane $x=1.$

\begin{figure}[htb!]
	\includegraphics[scale=0.14, trim={0cm 8cm 0cm 0cm}]{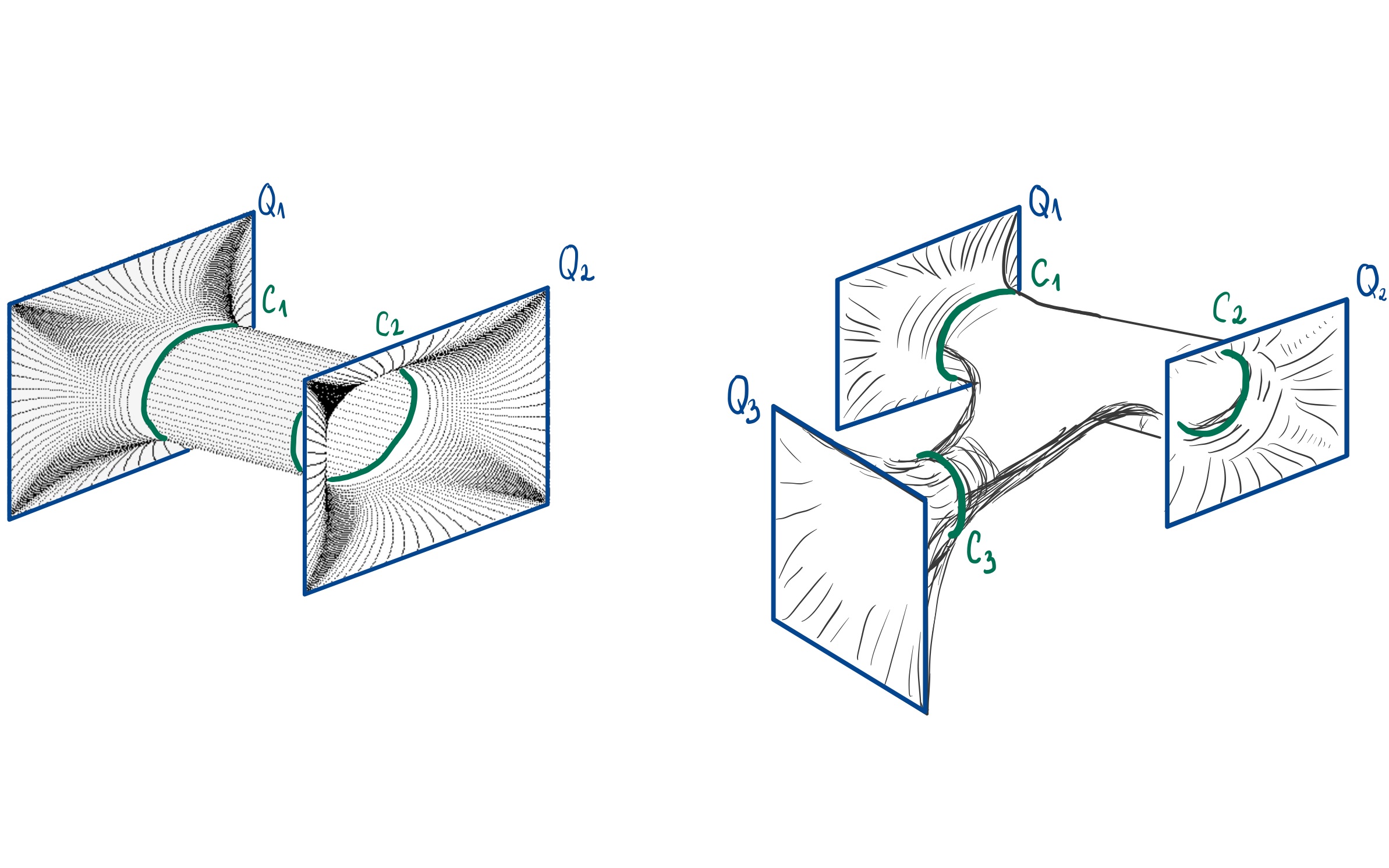}
	\caption{Surface $U$, with $g=2$ (left) and $g=3$ (right).}\label{figure:surfaceU}
\end{figure}

\

\subparagraph{\textbf{Hamiltonian flow on} $\mathbf{S}$}

We will construct a Hamiltonian $H$ on $S$ tracking the following idea: restrict the Hamiltonian $H^0$ in \eqref{eqn:Ham} to the neighborhood $D$, \emph{slightly} change it close to the boundary of $D$ (and obtain a \emph{changed} Hamiltonian $\hat{H}$ on $D$) and then glue, using $\rho$, the boundary components of $D$ and obtain the desired Hamiltonian $H=\hat{H} \circ \rho^{-1}$ on $S$.

Firstly, take the Hamiltonian $H^0\colon \R^2\to \R$ defined in \eqref{eqn:Ham} and a smooth map 
\begin{eqnarray}\label{eqn:H_3^1}
	H^1\colon \R^2\to \R	
\end{eqnarray}
such that $H^1(x,y)=x-\tan(ja)y$, when $(x,y)\in D$, and,
\begin{enumerate}
	\item when $g$ is even,
	\begin{enumerate}
		\item\label{item:1a} $(x,y)$ is above the line that contains $s_j{''}$ or $(x,y)$ is below the line that contains $s_{2g-2+j}{''}$, for each $j\in\{1,\ldots,g-2\}$ even, or
		\item $(x,y)$ is below the line that contains $s_{j-1}{''}$ or $(x,y)$ is above the line that contains $s_{2g+j-3}{''}$, for each $j\in\{g+1,\ldots,2g-1\}$ odd,
	\end{enumerate}
	\item when $g$ is odd,
	\begin{enumerate}
		\item $(x,y)$ is above the line that contains $s_j{''}$ or $(x,y)$ is below the line that contains $s_{2g-2+j}{''}$, for each $j\in\{1,\ldots,g-2\}$ odd, or
		\item\label{item:2b} $(x,y)$ is below the line that contains $s_{j-1}{''}$ or $(x,y)\in D^{0}$ is above the line that contains $s_{2g+j-3}{''}$, for each $j\in\{g+1,\ldots,2g-1\}$ even.
	\end{enumerate}
\end{enumerate}

Define the Hamiltonian $\hat{H}$ on $\R^2$ as a convex combination of $H^0$ and $H^1$, more precisely, 
\begin{eqnarray*}
	\hat{H} (x, y)=(1-\nu(x,y))H^0(x,y)+\nu(x,y)H^1(x,y),
\end{eqnarray*}
where $\nu\colon \R^2 \to [0,1]$ is a smooth function which is $0$ when $(x,y)$ is in the interior of $D^{''}$, $1$ when $(x,y)$ is in $\R^2\backslash D^{'}$ and \emph{strictly monotone}\footnote{The map $\nu$ is strictly monotone in the interior  of $D^{'}\backslash D^{''}$ if it is strictly monotone along level sets of $H^1$ (in the interior of $D^{'}\backslash D^{''}$); more precisely, for $c\in \R$ and $x=\tan(ja)y+c$, the map $y\mapsto \nu(x,y)$ is strictly monotone when $(x,y)$ is in  the interior  of $D^{'}\backslash D^{''}$ and in one of the sets described in \eqref{item:1a}-\eqref{item:2b}.} in the interior of $D^{'}\backslash D^{''}$.

\begin{remark}
	When $g=2$, we denote $D^{'}$ by $[-b,b]\times[-\beta,\beta]$ and $D^{''}$ by $[-b,b]\times[-\alpha,\alpha]$ ($\beta>\alpha>0$). Then, for $(x,y)\in D$, $\hat{H} (x, y)=(1-\nu(y))H^0(x,y)+\nu(y)H^1(x,y)$, where the map $\nu\colon \R\to [0,1]$ is a smooth function which is $0$ when $y$ is in $(-\alpha, \alpha)$ (i.e. $(x,y)$ is in the interior of $D^{''}$), $1$ when $y$ is in $(-\infty, -\beta)\cup (\beta,+\infty)$ (i.e. $(x,y)$ is in $\R^2 \backslash D^{'}$) and is strictly monotone in $(-\beta,-\alpha)\cup(\alpha, \beta)$ (i.e. $(x,y)$ is in the interior of $D^{'}\backslash D^{''}$); in this case, $H^1(x,y)=y$. See Figure~\ref{figure:H_hat_D_3} (left).
\end{remark}

\begin{remark}\label{rmk:choices_D}
	Neighborhoods $D^{'}$ and $D^{''}$ are defined \emph{near} each other and satisfying condition \eqref{eqn:vertices_D} so that, for $j$ even (or odd) when $g$ is even (or odd, respectively), all flow lines of $X_{H^0}$ which intersect $s_j{''}$ either enter or exit the interior of $D^{''}$ and such that there are no \emph{new} (with respect to $X_{H^0}$ and $X_{H^1}$) periodic orbits of $X_{\hat{H}}$ (namely, in the \emph{strips} in $D^{''}\backslash D^{'}$). The periodic points of $\hat{H}$ coincide with the periodic points of $H^0$. 
\end{remark}

The Hamiltonian flow $\varphi^t_{\hat{H}}$ of $X_{\hat{H}}$ has Lefschetz index $2-2g$; moreover, $(0,0)$ is the unique (degenerate) fixed point of $\varphi^t_{\hat{H}}$, with index $2-2g,$ and there are no other periodic orbits; cf. Remarks~\ref{rmk:index_hyperbolicregion} and~\ref{rmk:choices_D}. See Figure~\ref{figure:H_hat_D_3} for
Hamiltonian flow lines of $\hat{H}$ in $D$ when $g=2$ and $g=3$.
\begin{figure}[htb!]
	\includegraphics[scale=0.14,trim={0cm 5cm 0cm 0cm}]{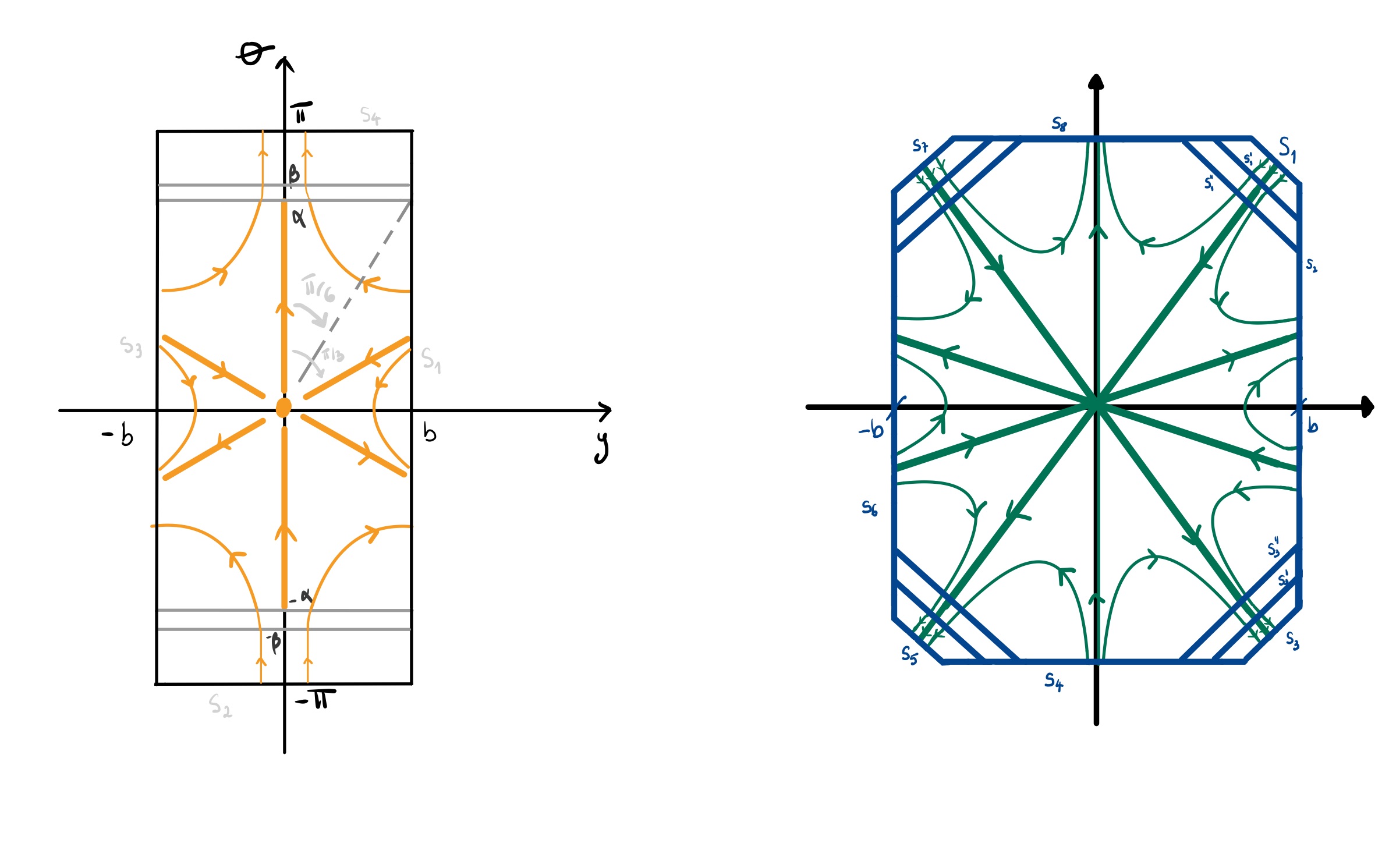}
	\caption{Hamiltonian flow lines of $\hat{H}$ in $D$ when $g=2$ (left) and $g=3$ (right).}\label{figure:H_hat_D_3}
\end{figure}
Then $\hat{H}$ induces a smooth Hamiltonian $H$ on the surface $S$ defined by $H=\hat{H}\circ \rho^{-1}.$ 
\begin{remark}\label{rmk:nearD_rhog}
	The fact that $\hat{H}$ coincides with $H^1$, as defined above, guarantees that, when gluing the sides of $D$ in order to obtain $S$ (using the procedure described on page   and formalized by $\rho$), the Hamiltonian $H$ is well defined and smooth. Recall the case $g=2$ in Remark~\ref{rmk:rho_2}. 
\end{remark}

\

\paragraph{\textbf{Hamiltonian flow on} $\mathbf{U}$} 

Now, we will construct the desired Hamiltonian $F$ on $U$ with the following properties:
\begin{itemize}
	\item $F$ coincides with the height function $(x,y,z)\mapsto z$ on $P_1$ and $P_2$ 
	\item $F$ coincides with the \emph{height} function $(x,y,z)\mapsto y$ on $P_i$ when $i\in \{3,\ldots,g\}$ and 
	\item $F$ coincides with $H$ on $S$ (more precisely, on $S$ except \emph{near} its boundary components). 
\end{itemize}

Firstly, define an auxiliary function $G\colon \R^3\to \R$ by
\begin{eqnarray*}
	G(x,y,z)=(1-\sigma(x,y))z+\sigma(x,y)y
\end{eqnarray*}
where the smooth function $\sigma\colon \R^2 \to \R$ is $0$ when $x<a{''}$, it is $1$ when $x>a{'}$ and it is strictly monotone\footnote{For each $c\in \R$, $y\mapsto \sigma(c,y)$ is strictly monotone.} when $a{''}<x<a{'}$. (Here $\varepsilon<a{''}<a{'}<a$. See Remark~\ref{remark:choices_parameters_F_3}.)

Finally, let $F\colon U \to [-\varepsilon/2, \varepsilon/2]$ be the smooth function defined by 
\[
F(x,y,z)=(1-\tau(x,y))H(x,y,z)+\tau(x,y)G(x,y,z), \quad \text{for } (x,y,z)\in U
\]
where $\tau\colon \R^2\to [0,1]$ is a smooth function which is $0$ when $(x,y)$ is in $(-a{'},a{'})\times(-b{'},b{'})$, $1$ when $(x,y)$ is in $\R^2\backslash [-a,a]\times[-b,b]$ and strictly monotone\footnote{The map $\tau$ is strictly monotone in $(-a,a)\times(-b,b)\backslash [-a{'},a{'}]\times[-b{'},b{'}]$ if, for each $c\in [-b,b]$ and $x<a^{''}$, $x\mapsto \tau(x,c)$ is strictly monotone and, for each $c\in [-a{''},a'']$, $y\mapsto \tau(c,y)$ is strictly monotone.} in $(-a,a)\times(-b,b)\backslash [-a{'},a{'}]\times[-b{'},b{'}]$. Here $0<a{'}<a<1$, $0<b{'}<b<1$, $a{'}$ is \emph{near} $a$ and $b{'}$ is \emph{near} $b$. 
\begin{remark}\label{remark:choices_parameters_F_3}
	The numbers $a{''}$ and $a{'}$ (and $b{'}$) are chosen \emph{near} $a$ ($b$, respectively) so that there are no new (with respect to $X_{H}$ and $X_G$) periodic trajectories of $X_{F}$ arising in $U$, namely, when $(x,y)\in (-a,a)\times(-b,b)\backslash [-a{'},a{'}]\times[-b{'},b{'}].$ Moreover, the number $a{''}$ is an auxiliary parameter in the definition of the function $G$; it avoids the existence of a region where the definition of the function $F$ \emph{interpolates} three functions (the two \emph{height} functions and $H$) instead of just two as constructed.
\end{remark}

\begin{remark}\label{remark:height_symmetric}
	The \emph{height} function $G$ in the definition on $F$ ensures that a flow line of $X_{F}$ which starts at one of the squares of the boundary of $U$ either converges to the unique fixed point $(\rho(0,0)\in S)$ of $F$ or ends at the \emph{symmetric} point of the opposite side of the same square, that is, the point with the same \emph{height} $G$ in the opposite side of the same square; see, e.g, the flow lines in Figure~\ref{figure:glue_flow_3}. 
\end{remark}

\paragraph{\textbf{Step 3}}\label{paragraph:step3}
In this last step, 
\begin{itemize}
	\item cut off $\mathring{R_i}$ from $\mathbb{T}_i,\, i=1,\ldots, g$
	\item identify $R_i$ with $Q_i$, when $i=1,2$, so that the sides of $R_i$ given by segments of a flow line correspond to the sides of $Q_i$ determined by $z=\pm\varepsilon/2$ (see Figure~\ref{figure:glue}) and
	\begin{figure}[htb!]
		\includegraphics[scale=0.14,trim={0cm 5cm 0cm 0cm}]{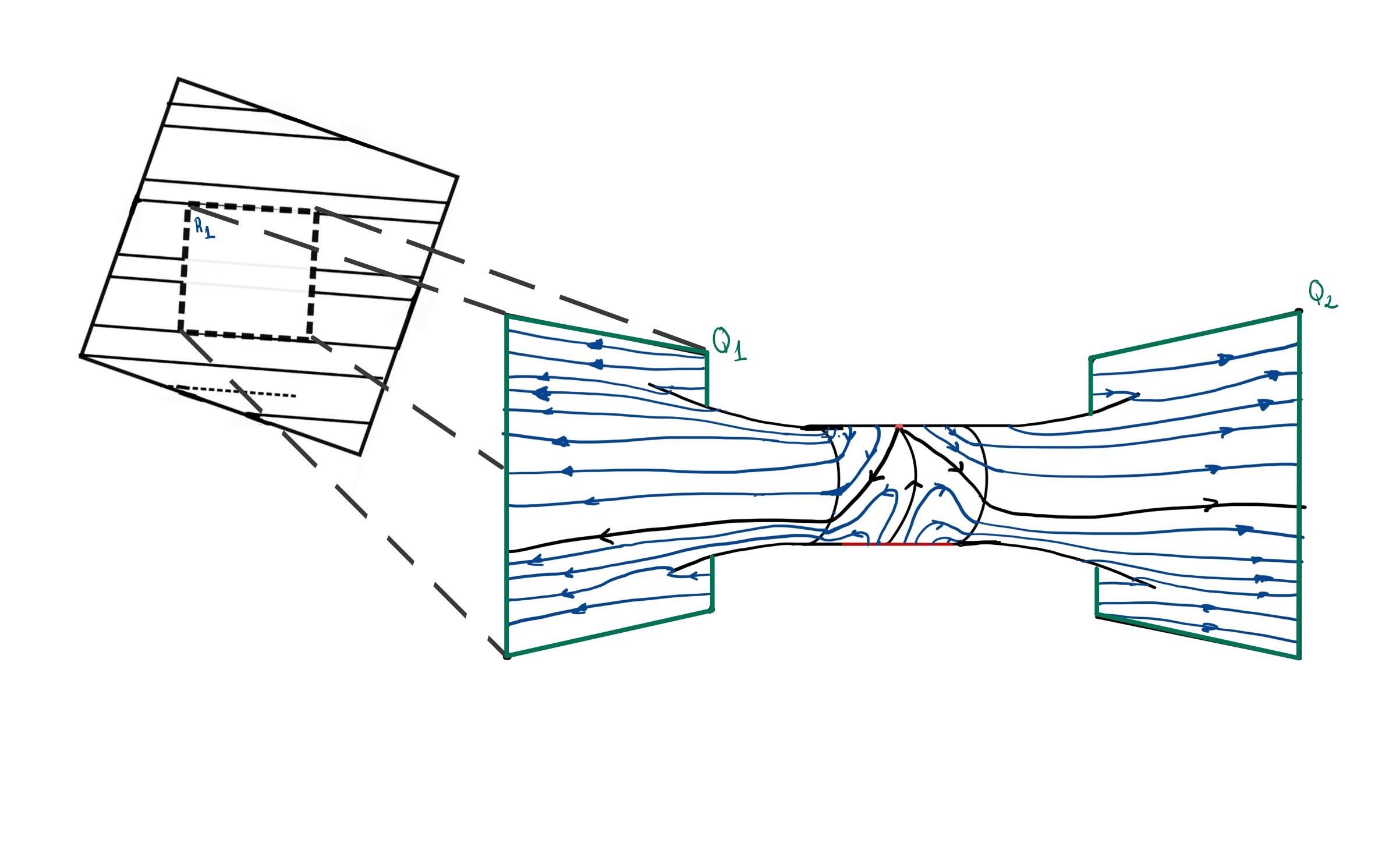}
		\caption{Identification of $R_1$ with $Q_1$, when $g=2$. }\label{figure:glue}
	\end{figure}
	\item identify $R_i$ with $Q_i$, when $i=3,\ldots,g$, so that the sides of $R_i$ given by segments of a flow line correspond to the sides of $Q_i$ determined by $y=\pm\varepsilon/2$ (see Figure~\ref{figure:glue_flow_3}).
\end{itemize}
\begin{figure}[htb!]
	\includegraphics[scale=0.14,trim={0cm 0cm 0cm 0cm}]{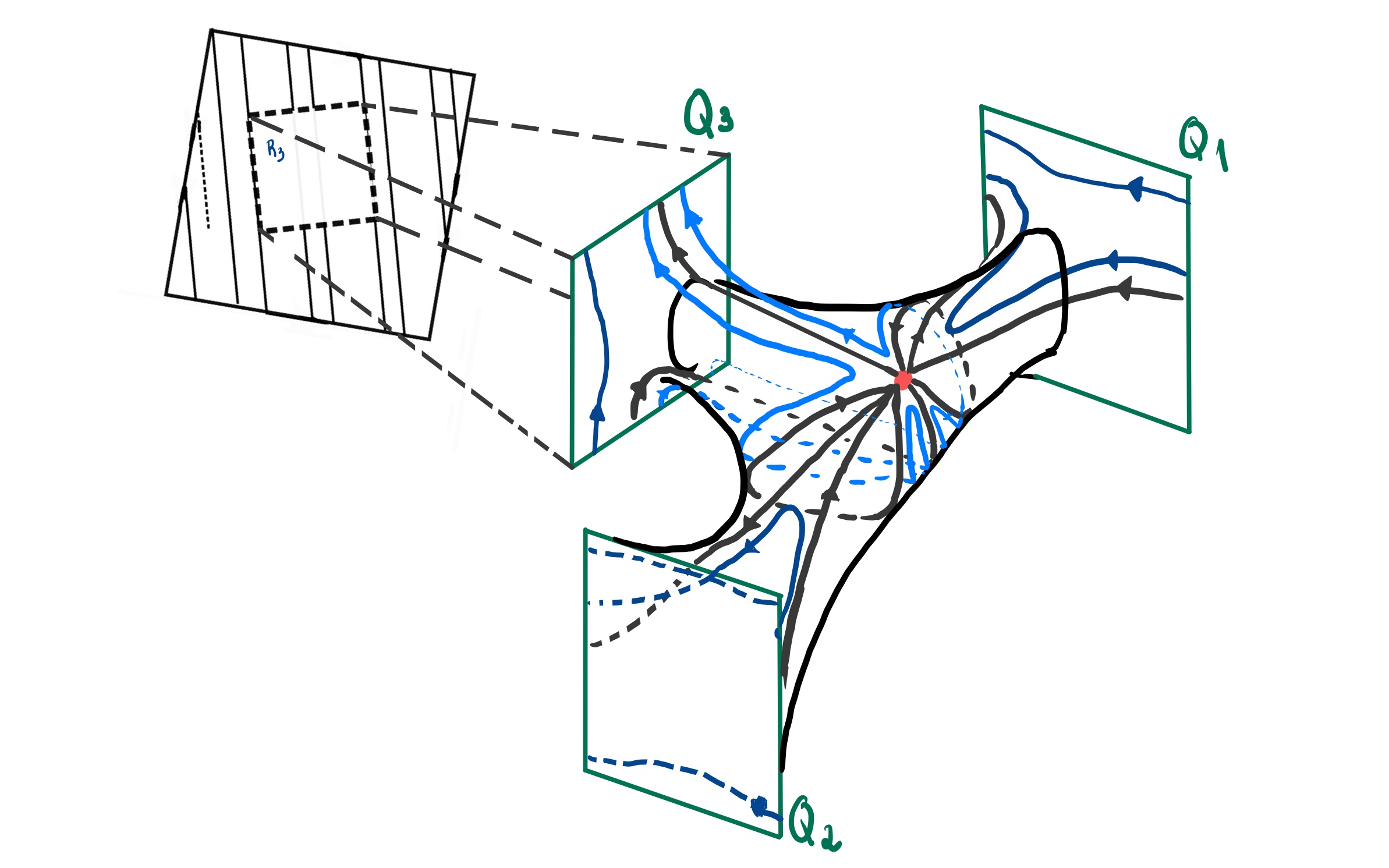}
	\caption{Identification of $R_3$ with $Q_3$ and flow lines of $X_{F}$ in $U$, when $g=3$.}\label{figure:glue_flow_3}
\end{figure}

This construction yields a closed surface $\Sigma$ of genus $g$ with a symplectic flow $\psi_t\colon \Sigma \rightarrow \Sigma$ which coincides with
\begin{itemize}
	\item the linear flow $\phi^t_i$ on $\mathbb{T}_i\backslash \mathring{R_i},\,i=1,\ldots,g$
	\item the Hamiltonian flow of $F$ on $U$. 
\end{itemize}

Consider the following $g$ open regions, $\widetilde{W^1}, \ldots, \widetilde{W^g}$, in the interior of $D$:
\begin{itemize}
	\item $\widetilde{W^1}$ is contained in the interior of $D\cap\{x<0\}$ and it is bounded by the set $\widetilde{L^1}:=\{(x,y)\in D |\; x<0 \;\wedge\; (x-\tan((g-2)a)y=0 \;\vee\; x-\tan((g+1)a)y=0)\}$,
	\item $\widetilde{W^2}$ is contained in the interior of $D\cap\{x>0\}$ and it is bounded by the set $\widetilde{L^2}:=\{(x,y)\in D|\; x>0 \;\wedge\; (x-\tan((g-2)a)y=0 \;\vee\; x-\tan((g+1)a)y=0)\}$ and,
	\item $\widetilde{W^j}:= \{(x,y)\in D|\, (x,y)\in \widetilde{W_{-}^j} \; \vee\;(x,-y)\in \widetilde{W_{-}^j} \}$, for $j=3,\ldots, g$, where
	\item[] $\widetilde{W_{-}^j}$ is contained in the interior of $D\cap\{y<0\}$ and it is bounded by the set $\widetilde{L_{-}^j}:= \{ (x,y)\in D|\; y<0 \; \wedge\; (x-\tan((g+(2j-5))a)y=0 \; \vee\; x-\tan((g+(2j-3))a)y=0)\}$. 
\end{itemize}
Regions $\widetilde{W^1}, \ldots, \widetilde{W^g}$ are illustrated in Figure~\ref{figure:regions_W2}, when $g=2$, and Figure~\ref{figure:regions_W}, when $g=3$.

\begin{figure}[htb!]
	\includegraphics[scale=0.13,trim={0cm 0cm 0cm 0cm}]{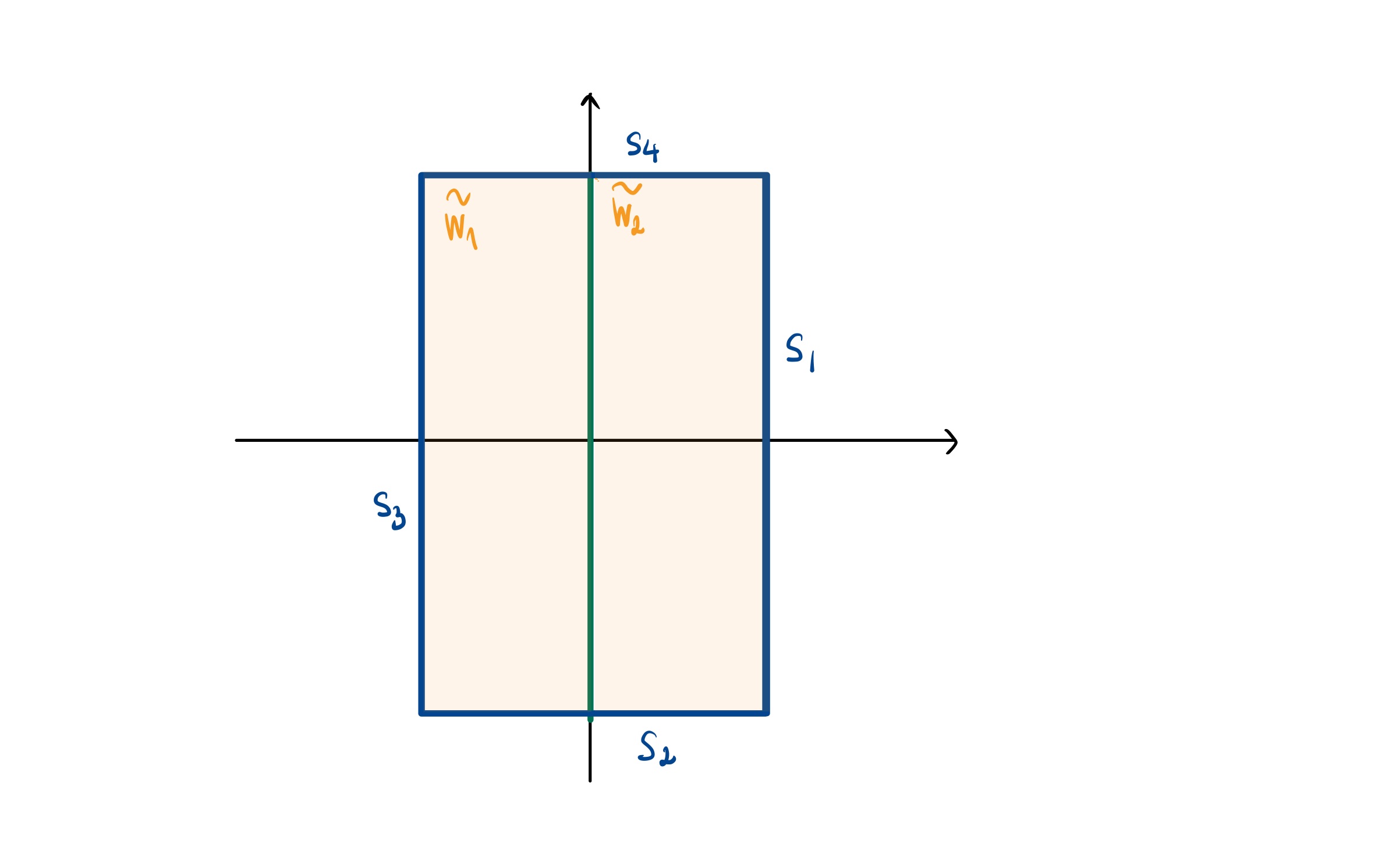}
	\caption{Regions $\widetilde{W^1}$ and $\widetilde{W^2}$, when $g=2$.}\label{figure:regions_W2}
\end{figure}

\begin{figure}[htb!]
	\includegraphics[scale=0.145,trim={0cm 10cm 0cm 0cm}]{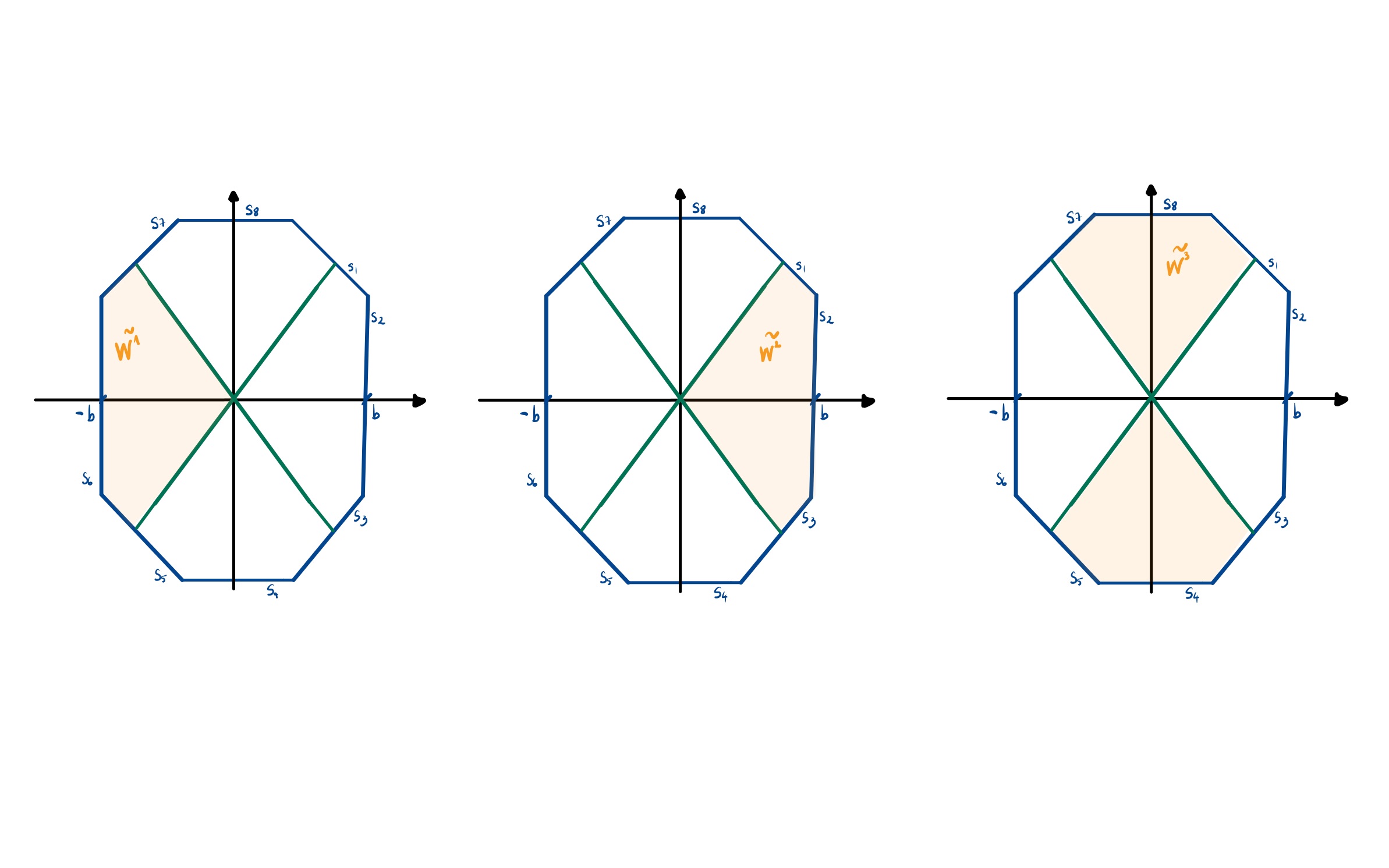}
	\caption{Regions $\widetilde{W^1}$ (left), $\widetilde{W^2}$ (center) and $\widetilde{W^3}$ (right), when $g=3$.}\label{figure:regions_W}
\end{figure}

For each $i=1,\ldots,g$, let $W^i: = \{(x,y,z)\in U|\, (x,y,z)\in P^i \text{ or } (x,y,z) \in \rho(\widetilde{W^i})\}$ and $L^i:= \rho(\widetilde{L^i})$, where $\widetilde{L^i}:=\{(x,y)\in D|\; (x,y) \in \widetilde{L_{-}^i} \text{ or } (x,-y) \in \widetilde{L_{-}^i}\}$. Then, each flow line of $\psi_t$ lies entirely

\begin{enumerate}
	\item\label{item1} either on $U\cap L^i$, for some $i\in \{1,\ldots, g\}$
	\item\label{item2} or on $\mathbb{T}_i\backslash \mathring{R_i} \cup W^i=:V^i$, for some $i\in \{1,\ldots, g\}$.
\end{enumerate}
We observe that a flow line of $\psi_t$ does not intersect both $V^i$ and $V^j$, $i\not=j$. In $U$ (see case~\eqref{item1}), $\psi_t$ has one degenerate fixed point of index $2-2g$ and no other periodic orbits; this fixed point lies in the intersection $\bigcap_{i=1}^{g} L_i \cap U$. In $V^1$ (see case~\eqref{item2}, with $i=1$), $\psi_t$ has no periodic orbits. In fact, by construction, when a flow line of $\psi_t$ given by the irrational flow $\phi_1^t$ reaches $R_1$, it will
\begin{itemize}
	\item either stay on $U$ and converge to the (unique) fixed point
	\item or cross $R_1$ again after some time and continue in the same flow line of $\phi^t_1$ when exiting $\mathbb{T}_1\backslash \mathring{R_1}$ (since at $Q_1$ the Hamiltonian is given by the height function).
\end{itemize}
This property together with the fact that $\phi^t_1$ is an irrational linear flow imply the non existence of (long) periodic orbits of $\psi_t$ on $V^1$. Cases in \eqref{item2} with $i=2,\ldots, g$ are similar to the case in \eqref{item2} with $i=1$, and there are no periodic orbits of $\psi_t$ on neither of the sets $V^i$, $i=2,\ldots,g$.

Therefore, we have obtained a symplectic flow on $\Sigma$ with exactly one degenerate fixed point of index $2-2g$ and no other periodic orbits. 


\begin{remark}
	In Figure~\ref{figure:glue_flow_3}, although the lines are represented in different colors, they may represent the same flow line of $\psi_t$. 
\end{remark}

\subsection{Symplectic flows on $(\Sigma,\omega)$, $g\geq 2$, with a finite number of fixed points (with non-zero indices) and no other periodic orbits.}\label{section:partition}

\

The construction of $\psi_t$ on $(\Sigma,\omega)$ with exactly $2g-2$ fixed points (as in Theorem~\ref{thm:2g-2_fx_pts}) is very similar to the one presented in Section~\ref{section:construction_1fx_pt} and also has three steps: in step~$1$, we consider $g$ tori with irrational linear flows, in step~$2$, we construct a surface (with $g$ \emph{circle-ends}) and define on it a Hamiltonian flow with exactly $2g-2$ fixed points and, in step~$3$, we glue each torus to  one of the \emph{ends} of the constructed surface in order to obtain $\Sigma$ with the desired symplectic flow $\psi_t$. A draft of the construction of $\Sigma$ is, as in Section~\ref{section:construction_1fx_pt}, in Figure~\ref{figure:general}. The only difference between the present construction and the previous one is the Hamiltonian map defined in \eqref{eqn:Ham}; here, one substitutes that map by 
\begin{eqnarray}\label{eqn:H_g^0_new}
	H^0(x,y)=\prod_{j=1}^{h/2} \left(x-\tan(j a) y\right)+\delta x,
\end{eqnarray}
where $|\delta|\not=0$ is a \emph{small} real number. (For instance one can take $\delta=0.1$.)  

Flow lines of the new vector field $X_{H^0}$ are depicted in Figure~\ref{figure:flow_lines_H0_} when $g=2$ (left) and $g=3$ (right).

\begin{figure}[htb!]
	\includegraphics[scale=0.13, trim={0cm 0cm 0cm 0cm}]{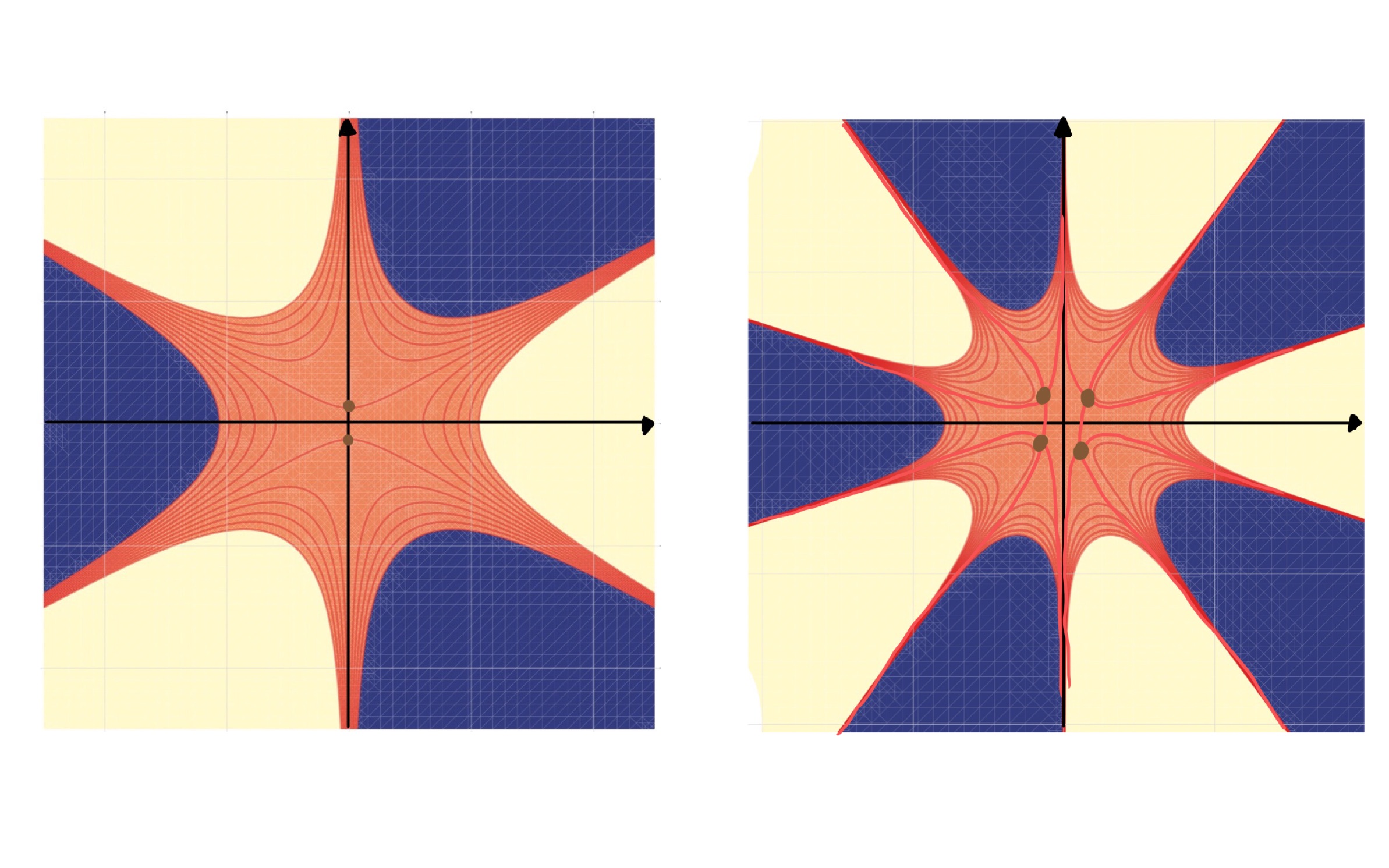}
	\caption{Flow lines of $H^0$, with $g=2$ (left) and $g=3$ (right)}\label{figure:flow_lines_H0_}
\end{figure}

In this case, the Hamiltonian system on $\R^2$ given by the new $X_{H^0}$ has exactly $2g-2$ fixed points. The index of each fixed point is $-1$ and the number of hyperbolic regions of each fixed point is $h=4$, as one can recognize in the examples in Figure~\ref{figure:flow_lines_H0_}. (Recall Remark~\ref{rmk:index_hyperbolicregion}.)

\begin{remark}
	The new Hamiltonian is a perturbation of the Hamiltonian defined in the previous construction: the multiple saddle is \emph{replaced} by $2g-2$ simple saddles. See  \cite[Example on page 321]{katok1995introduction} for a further discussion on this phenomenon.      
\end{remark}

In the procedure in Section~\ref{section:construction_1fx_pt}, after substituting the Hamiltonian \eqref{eqn:Ham} by the new $H^0$, reproduce the following steps of the procedure and obtain the desired symplectic flow $\psi_t$ with $2g-2$ fixed points and no other periodic orbits. Similarly to the previous construction, the obtained final symplectic flow $\psi_t$ is Hamiltonian on the \emph{connecting surface} $U$  and, \emph{locally}, the Hamiltonian function is given by \eqref{eqn:H_g^0_new}.

\begin{remark}
	When $g=2$, the obtained symplectic flow $\psi_t$ is \emph{equivalent} to the one obtained in \cite[Section~3A]{batoreo2018periodic}; namely, the two (hyperbolic) fixed points (with index $-1$) lie on the connecting surface $U$, which, in this case, is a cylinder, and, on each $\mathbb{T}_i\backslash \mathring{R_i}$ the flow is the restriction of an irrational linear flow on $\mathbb{T}_i, i=1,2$.  
	
\end{remark}

As mentioned, the construction in this section differs from the one in Section~\ref{section:construction_1fx_pt} by the convenient substitution of the \emph{local} Hamiltonian $H^0$ in $D{''}$. Trailing this idea and in order to obtain examples that illustrate the \emph{partition condition} in Remark~\ref{remark:partition}, one may consider a convenient Hamiltonian $H^0$ on $D{''}$ where its flow has exactly $k$ periodic orbits which are fixed points, $p_1,\ldots,p_k$, with $ind(p_i)=-a_i$ (where $(a_1,\ldots,a_k)$ is a partition of $2g-2$) so that the resulting flow of $X_{\hat{H}}$ on $D$ has the later property on the periodic orbits. We discuss the \emph{partition condition} in Remark~\ref{remark:partition} when the genus $g$ is \emph{small}, namely, when $g=2$ and $g=3$.

\paragraph{When $g=2$,} there are two partitions of $2g-2=2$: the first is $(2)$, where $k=1$, and the other is $(1,1)$, where $k=2$. As mentioned in Remark~\ref{remark:partition}, these 
correspond to the cases in Theorem~\ref{thm:1_fx_pt} and Theorem~\ref{thm:2g-2_fx_pts}, respectively.  

\paragraph{When $g=3$,} there are five partitions of $2g-2=4$: $(4)$, where $k=1$, $(1,3)$ and $(2,2)$, where $k=2$, $(1,1,2)$, where $k=3$, and $(1,1,1,1)$, where $k=4$. Cases $k=1$ and $k=4$ are considered in Theorems~\ref{thm:1_fx_pt} and~\ref{thm:2g-2_fx_pts}, respectively. The case $k=3$ can be obtained with the Hamiltonian 
\begin{eqnarray}\label{eqn:H_3^0_k3}
	H^0(x,y)= \prod_{j=1}^{h/2} \left(x-\tan(j a) y\right)-\delta x^3,
\end{eqnarray}
with $h=10, a=\pi/5$ and $\delta>0$. See Figure~\ref{figure:flow_lines_H0_k3}. (In fact, as the notation suggests, this Hamiltonian may also be considered for $g>3$ when $k=3$ and $a_i=1$, for $i=1,2$, and $a_i=2g-4$, for $i=3$.) 

\begin{figure}[htb!]
	\includegraphics[scale=0.11, trim={0 0cm 0 0cm}]{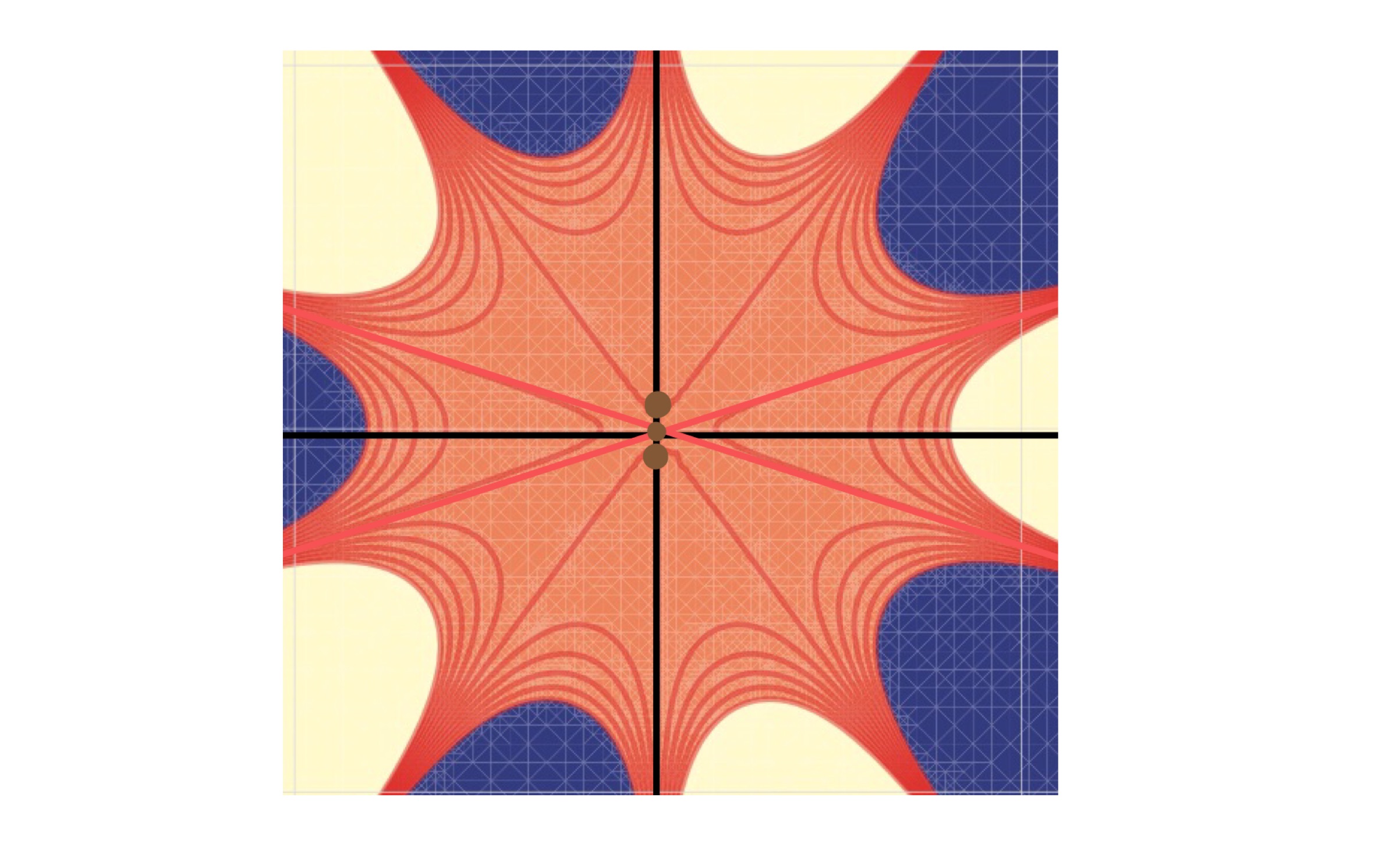}
	\caption{Flow lines of $H^0$ in \eqref{eqn:H_3^0_k3}, with $g=3$, $k=3$, $a_1=a_2=1$ and $a_3=2$.}\label{figure:flow_lines_H0_k3}
\end{figure}

The remaining case, $k=2$, is depicted in Figure~\ref{figure:flow_lines_H0_k2}. The picture on the left corresponds to the partition $(1,3)$ and illustrates the local Hamiltonian flow lines in $D_3{''}$. The picture on the right corresponds to the partition $(2,2)$ and illustrates the Hamiltonian flow lines on the surface $S_3$.      

\begin{figure}[htb!]
	\includegraphics[scale=0.13, trim={0cm 5cm 0 0cm}]{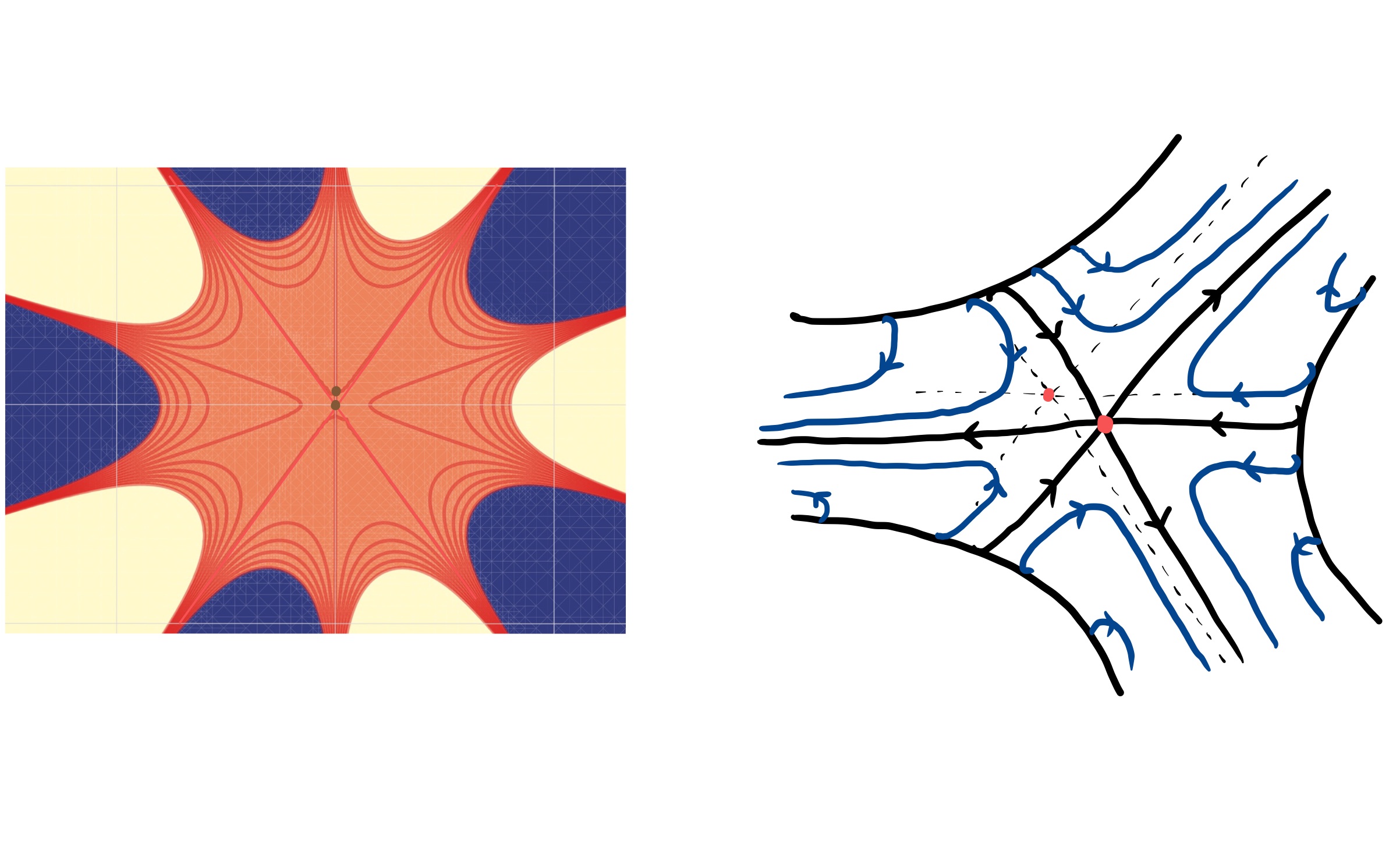}
	\caption{Left: partition $(1,3)$; Right: partition $(2,2)$.}\label{figure:flow_lines_H0_k2}
\end{figure}

\newpage
\bibliographystyle{alpha}
\bibliography{Non-Hamiltonian_HZ}

\newcommand{\etalchar}[1]{$^{#1}$}
\begin{thebibliography}{GHHM13}

\bibitem[AB21]{abouzaid2021arnold}
M.~Abouzaid and A.J. Blumberg.
\newblock Arnold conjecture and {M}orava {K}-theory.
\newblock {\em arXiv preprint arXiv:2103.01507}, 2021.

\bibitem[AL23]{atallah2023hz}
M.S. Atallah and H.~Lou.
\newblock {On the Hofer-Zehnder conjecture for semipositive symplectic
  manifolds}.
\newblock {\em arXiv preprint arXiv:2309.13791}, 2023.

\bibitem[Arn65]{arnold1965proprietes}
V.I. Arnold.
\newblock Sur une propriété topologique des applications globalment canonique
  de la méchanique classique.
\newblock {\em CR. Acad. Sci. Paris}, 261:3719--3722, 1965.

\bibitem[Arn13]{arnol2013mathematical}
V.I. Arnold.
\newblock {\em Mathematical methods of classical mechanics}, volume~60.
\newblock Springer Science \& Business Media, 2013.

\bibitem[Arn14]{arnold2014stability}
V.I. Arnold.
\newblock The stability problem and ergodic properties for classical dynamical
  systems.
\newblock {\em Vladimir I. Arnold-Collected Works: Hydrodynamics, Bifurcation
  Theory, and Algebraic Geometry 1965-1972}, pages 107--113, 2014.

\bibitem[AS23]{atallah2020hamiltonian}
M.S. Atallah and E.~Shelukhin.
\newblock Hamiltonian no-torsion.
\newblock {\em Geometry \& Topology}, 27(7):2833--2897, 2023.

\bibitem[Bat18]{batoreo2018periodic}
M.~Bator{\'e}o.
\newblock On periodic points of symplectomorphisms on surfaces.
\newblock {\em Pacific Journal of Mathematics}, 294(1):19--40, 2018.

\bibitem[BPS03]{biranpolsal}
P.~Biran, L.~Polterovich, and D.~Salamon.
\newblock {Propagation in Hamiltonian dynamics and relative symplectic
  homology}.
\newblock {\em Duke Mathematical Journal}, 119(1):65 -- 118, 2003.

\bibitem[BX22]{bai2022arnold}
S.~Bai and G.~Xu.
\newblock Arnold conjecture over integers.
\newblock {\em arXiv preprint arXiv:2209.08599}, 2022.

\bibitem[BX24]{bai2024hz}
S.~Bai and G.~Xu.
\newblock Franks dichotomy for toric manifolds, hofer-zehnder conjecture, and
  gauged linear sigma model.
\newblock {\em arXiv preprint arXiv:2309.07991}, 2024.

\bibitem[CGK04]{cieliebak2004symplectic}
K.~Cieliebak, V.~Ginzburg, and E.~Kerman.
\newblock Symplectic homology and periodic orbits near symplectic submanifolds.
\newblock {\em Commentarii Mathematici Helvetici}, 79(3):554--581, 2004.

\bibitem[CKR{\etalchar{+}}12]{collier2012symplectic}
B.~Collier, E.~Kerman, B.M. Reiniger, B.~Turmunkh, and A.~Zimmer.
\newblock {A symplectic proof of a theorem of Franks}.
\newblock {\em Compositio Mathematica}, 148(6):1969--1984, 2012.

\bibitem[Con84]{Co}
C.~C. Conley.
\newblock {\em Lecture at the University of Wisconsin}, April 6, 1984.

\bibitem[CZ84]{conley1984subharmonic}
C.~Conley and E.~Zehnder.
\newblock Subharmonic solutions and {M}orse theory.
\newblock {\em Physica A: Statistical Mechanics and its Applications},
  124(1-3):649--657, 1984.

\bibitem[FH03]{franks2003periodic}
J.~Franks and M.~Handel.
\newblock Periodic points of {H}amiltonian surface diffeomorphisms.
\newblock {\em Geometry \& Topology}, 7(2):713--756, 2003.

\bibitem[Flo86]{Floer1}
A.~Floer.
\newblock Proof of the {A}rnol'd conjecture for surfaces and generalizations to
  certain {K}\"ahler manifolds.
\newblock {\em Duke Math. J.}, 53(1):1--32, 1986.

\bibitem[Flo87]{Floer2}
A.~Floer.
\newblock {M}orse theory for fixed points of symplectic diffeomorphisms.
\newblock {\em Bull. Amer. Math. Soc. (N.S.)}, 16(2):279--281, 1987.

\bibitem[Flo89a]{Floer3}
A.~Floer.
\newblock Symplectic fixed points and holomorphic spheres.
\newblock {\em Comm. Math. Phys.}, 120(4):575--611, 1989.

\bibitem[Flo89b]{floer1989witten}
A.~Floer.
\newblock Witten's complex and infinite-dimensional {M}orse theory.
\newblock {\em Journal of differential geometry}, 30(1):207--221, 1989.

\bibitem[FO99]{fukaya1999arnold}
K.~Fukaya and K.~Ono.
\newblock {A}rnold conjecture and {G}romov--{W}itten invariant.
\newblock {\em Topology}, 38(5):933--1048, 1999.

\bibitem[Fra90]{franks1990periodic}
J.~Franks.
\newblock Periodic points and rotation numbers for area preserving
  diffeomorphisms of the plane.
\newblock {\em Publications Math{\'e}matiques de l'IH{\'E}S}, 71:105--120,
  1990.

\bibitem[Fra92]{franks1992geodesics}
J.~Franks.
\newblock Geodesics on ${S}^2$ and periodic points of annulus homeomorphisms.
\newblock {\em Inventiones Mathematicae}, 108(1):403--418, 1992.

\bibitem[Fra96]{franks1996area}
J.~Franks.
\newblock Area preserving homeomorphisms of open surfaces of genus zero.
\newblock {\em New York J. Math}, 2(1):19, 1996.

\bibitem[GG09a]{ginzburg2009action}
V.L. Ginzburg and B.Z. G{\"u}rel.
\newblock Action and index spectra and periodic orbits in {H}amiltonian
  dynamics.
\newblock {\em Geometry \& Topology}, 13(5):2745--2805, 2009.

\bibitem[GG09b]{GG09_genericexistence}
V.L. Ginzburg and B.Z. G{\"u}rel.
\newblock {On the generic existence of periodic orbits in Hamiltonian
  dynamics}.
\newblock {\em Journal of Modern Dynamics}, 3(4):595--610, 2009.

\bibitem[GG10]{ginzburg2010local}
V.L. Ginzburg and B.Z. G{\"u}rel.
\newblock Local {F}loer homology and the action gap.
\newblock {\em Journal of Symplectic Geometry}, 8(3):323--357, 2010.

\bibitem[GG12]{ginzburg2012conley}
V.L. Ginzburg and B.Z. G{\"u}rel.
\newblock {C}onley conjecture for negative monotone symplectic manifolds.
\newblock {\em IMRN: International Mathematics Research Notices}, 2012(8),
  2012.

\bibitem[GG15]{GG15_ccbeyond}
V.L. Ginzburg and B.Z. G{\"u}rel.
\newblock {The Conley Conjecture and Beyond}.
\newblock {\em Arnold Mathematical Journal}, 1:299--337, 2015.

\bibitem[GG19]{ginzburg2019conley}
V.L. Ginzburg and B.Z. G{\"u}rel.
\newblock {C}onley conjecture revisited.
\newblock {\em International Mathematics Research Notices}, 2019(3):761--798,
  2019.

\bibitem[GHHM13]{ginzburg2013closed}
V.~L Ginzburg, D.~Hein, U.~L. Hryniewicz, and L.~Macarini.
\newblock {Closed Reeb orbits on the sphere and symplectically degenerate
  maxima}.
\newblock {\em Acta Mathematica Vietnamica}, 38:55--78, 2013.

\bibitem[Gin10]{ginzburg2010conley}
V.L. Ginzburg.
\newblock The {C}onley conjecture.
\newblock {\em Annals of Mathematics}, pages 1127--1180, 2010.

\bibitem[GLCP23]{guiheneuf2023area}
P.-A. Guih{\'e}neuf, P.~Le~Calvez, and A.~Passeggi.
\newblock Area preserving homeomorphisms of surfaces with rational rotational
  direction.
\newblock {\em arXiv preprint arXiv:2305.05755}, 2023.

\bibitem[Har82]{hartman1982ordinary}
P.~Hartman.
\newblock {\em Ordinary Differential Equations: Second Edition}.
\newblock Classics in Applied Mathematics. Society for Industrial and Applied
  Mathematics (SIAM, 3600 Market Street, Floor 6, Philadelphia, PA 19104),
  1982.

\bibitem[Hei12]{hein2012conley}
D.~Hein.
\newblock The {C}onley conjecture for irrational symplectic manifolds.
\newblock {\em Journal of Symplectic Geometry}, 10(2):183--202, 2012.

\bibitem[Hin09]{hingston2009subharmonic}
N.~Hingston.
\newblock Subharmonic solutions of {H}amiltonian equations on tori.
\newblock {\em Annals of Mathematics}, pages 529--560, 2009.

\bibitem[HS95]{hofer1995floer}
H.~Hofer and D.~Salamon.
\newblock {Floer homology and Novikov rings}.
\newblock In {\em The Floer memorial volume}, pages 483--524. Springer, 1995.

\bibitem[HZ12]{hofer2012symplectic}
H.~Hofer and E.~Zehnder.
\newblock {\em Symplectic invariants and {H}amiltonian dynamics}.
\newblock Birkh{\"a}user, 2012.

\bibitem[JT23]{jang2023non}
D.~Jang and S.~Tolman.
\newblock Non-hamiltonian actions with fewer isolated fixed points.
\newblock {\em International Mathematics Research Notices}, 2023(7):6045--6077,
  2023.

\bibitem[KH95]{katok1995introduction}
A.~Katok and B.~Hasselblatt.
\newblock {\em Introduction to the Modern Theory of Dynamical Systems}.
\newblock Encyclopedia of Mathematics and its Applications. Cambridge
  University Press, 1995.

\bibitem[LC05]{LeCalvez05}
P.~Le~Calvez.
\newblock Une version feuillet\'ee \'equivariante du th\'eor\`eme de
  translation de {Brouwer}.
\newblock {\em Publications Math\'ematiques de l'IH\'ES}, 102:1--98, 2005.

\bibitem[LC06]{le2006periodic}
P.~Le~Calvez.
\newblock Periodic orbits of {H}amiltonian homeomorphisms of surfaces.
\newblock {\em Duke Mathematical Journal}, 133(1), 2006.

\bibitem[LC22]{LeCalvez_conservative2022}
P.~Le~Calvez.
\newblock Conservative surface homeomorphisms with finitely many periodic
  points.
\newblock {\em Journal of Fixed Point Theory and Applications}, 24, 04 2022.

\bibitem[LO95]{van1995symplectic}
H.V. L{\^e} and K.~Ono.
\newblock {Symplectic fixed points, the Calabi invariant and Novikov homology}.
\newblock {\em Topology}, 34(1):155--176, 1995.

\bibitem[LO20]{van2020floer}
H.V. L{\^e} and K.~Ono.
\newblock {Floer--Novikov cohomology and symplectic fixed points, revisited}.
\newblock {\em Proceedings of the International Geometry Center}, 13(4), 2020.

\bibitem[LRSV21]{le2021barcodes}
F.~Le~Roux, S.~Seyfaddini, and C.~Viterbo.
\newblock Barcodes and area-preserving homeomorphisms.
\newblock {\em Geometry \& Topology}, 25(6):2713--2825, 2021.

\bibitem[LT98]{liu1998floer}
G.~Liu and G.~Tian.
\newblock Floer homology and {A}rnold conjecture.
\newblock {\em Journal of Differential Geometry}, 49(1):1--74, 1998.

\bibitem[Mar01]{Markl}
M.~Markl.
\newblock Ideal perturbation lemma.
\newblock {\em Communications in Algebra}, 29(11):5209--5232, 2001.

\bibitem[MvK22]{moreno2022generalized}
A.~Moreno and O.~van Koert.
\newblock A generalized {P}oincar{\'e}--{B}irkhoff theorem.
\newblock {\em Journal of Fixed Point Theory and Applications}, 24(2):32, 2022.

\bibitem[Ono05]{ono05FN}
K.~Ono.
\newblock {Floer--Novikov cohomology and symplectic fixed points}.
\newblock {\em Journal of Symplectic Geometry}, 3(4):545 -- 563, 2005.

\bibitem[Ono06]{ono2006floer}
K.~Ono.
\newblock {Floer--Novikov cohomology and the flux conjecture}.
\newblock {\em Geometric \& Functional Analysis GAFA}, 16(5):981--1020, 2006.

\bibitem[Pie22]{pieloch2022fundamental}
A.~Pieloch.
\newblock Fundamental groups of rationally connected symplectic manifolds.
\newblock {\em arXiv preprint arXiv:2212.07882}, 2022.

\bibitem[Pra23]{prasad2023periodic}
R.~Prasad.
\newblock Periodic points of rational area-preserving homeomorphisms.
\newblock {\em arXiv preprint arXiv:2305.05876}, 2023.

\bibitem[PSS96]{piunikhin1996symplectic}
S.~Piunikhin, D.~Salamon, and M.~Schwarz.
\newblock Symplectic {F}loer-{D}onaldson theory and quantum cohomology.
\newblock {\em Contact and symplectic geometry (Cambridge, 1994)}, 8:171--200,
  1996.

\bibitem[Rez22]{rezchikov2022integral}
S.~Rezchikov.
\newblock Integral {A}rnol'd conjecture.
\newblock {\em arXiv preprint arXiv:2209.11165}, 2022.

\bibitem[Rua99]{ruan1999virtual}
Y.~Ruan.
\newblock Virtual neighborhoods and pseudo-holomorphic curves.
\newblock {\em Turkish Journal of Mathematics}, 23(1):161--232, 1999.

\bibitem[She22]{shelukhin2022hofer}
E.~Shelukhin.
\newblock On the {H}ofer-{Z}ehnder conjecture.
\newblock {\em Annals of Mathematics(2)}, 195(3):775--839, 2022.

\bibitem[Sla93]{slaminka93}
E.~E. Slaminka.
\newblock Removing index 0 fixed points for area preserving maps of
  two-manifolds.
\newblock {\em Transactions of the American Mathematical Society},
  340(1):429--445, 1993.

\bibitem[Sug21]{sugimoto2021hofer}
Y.~Sugimoto.
\newblock On the {H}ofer-{Z}ehnder conjecture for non-contractible periodic
  orbits in {H}amiltonian dynamics.
\newblock {\em arXiv preprint arXiv:2102.05273}, 2021.

\bibitem[SZ92]{salamon1992morse}
D.~Salamon and E.~Zehnder.
\newblock Morse theory for periodic solutions of {H}amiltonian systems and the
  {M}aslov index.
\newblock {\em Communications on pure and applied mathematics},
  45(10):1303--1360, 1992.

\bibitem[Tol17]{tolman2017non}
S.~Tolman.
\newblock {Non-Hamiltonian actions with isolated fixed points}.
\newblock {\em Inventiones mathematicae}, 210(3):877--910, 2017.

\bibitem[Vit99]{viterbo1999functors}
C.~Viterbo.
\newblock {Functors and computations in Floer homology with applications, I}.
\newblock {\em Geometric \& Functional Analysis GAFA}, 9(5):985--1033, 1999.

\end{thebibliography}
\end{document}